\documentclass[11pt]{amsart}

\usepackage[left=1in,top=1in,right=1in,bottom=1in]{geometry}

\usepackage{enumerate}
\usepackage{amsthm}
\usepackage{amsmath}
\usepackage{amsfonts,amssymb}
\usepackage{mathrsfs}
\usepackage[pdftex]{graphicx}

\usepackage[unicode=true,colorlinks=true,linkcolor=blue,urlcolor=blue, citecolor=blue]{hyperref}

\usepackage{amscd}
\usepackage{xcolor}
\usepackage{ifpdf} 
\usepackage[all]{xy}

\ifpdf
\else
\fi 

\numberwithin{equation}{section}
\numberwithin{table}{section}

\theoremstyle{plain}
\newtheorem{thm}{Theorem}[section]
\newtheorem{lemma}[thm]{Lemma}

\newtheorem{prop}[thm]{Proposition}
\newtheorem{corollary}[thm]{Corollary}

\theoremstyle{definition}

\newtheorem{definition}[thm]{Definition}

\newtheorem{exampleth}[thm]{Example}
\newenvironment{example}{\begin{exampleth}}{\hfill $\diamond$\\ \end{exampleth}}

\theoremstyle{remark}
\newtheorem{rem}[thm]{Remark}

\newcommand{\tightcdot}{\!\cdot\!}

\makeatletter
\@namedef{subjclassname@2010}{%
 \textup{2010} Mathematics Subject Classification}
\makeatother

\newcommand \NN {\mathbb{N}} 
 
\newcommand \ZZ {\mathbb{Z}}
\newcommand \CC {\mathbb{C}}
\newcommand \RR {\mathbb{R}}  
\newcommand \pr {\mathbb{P}}
\newcommand \GG {\mathbb{G}}

\renewcommand{\AA}{\ensuremath{\mathbb{A}}}

\newcommand{\fa}{\ensuremath{\mathfrak{a}}} 
\newcommand{\fb}{\ensuremath{\mathfrak{b}}} 
\newcommand{\faJij}{\fa_{J(ij)}}

\renewcommand{\geq}{\geqslant}
\renewcommand{\leq}{\leqslant}

\newcommand{\longhookrightarrow}{\lhook\joinrel\longrightarrow}

\newcommand{\A}{H}

\DeclareMathOperator{\inc}{\operatorname{inc}}

\DeclareMathOperator{\Spec}{\operatorname{Spec}}
\DeclareMathOperator{\Quot}{\operatorname{Quot}}

\DeclareMathOperator{\aff}{\operatorname{aff}}
\DeclareMathOperator{\trop}{trop}
\DeclareMathOperator{\Gr}{Gr}
\DeclareMathOperator{\an}{an}
\DeclareMathOperator{\id}{id}

\newcommand \ii{\theta}
\newcommand\rr{\mu}

\newcommand \cT {\mathcal{T}}

\newcommand \cTG {\mathcal{T}\!\Gr}
\newcommand \cTGo {{\cT\!\Gr_0}}

\newcommand \cTC {\mathscr{C}}

\newcommand \TP {\mathbb{T}}

\newcommand{\init}{\operatorname{in}}

\newcommand \CijTJ{\mathscr{C}^{(ij)}_{T,J}}
\newcommand \CijTJtilde{\mathscr{C}^{(ij)}_{T,\tilde{J}}}
\newcommand \CijTJprime{\mathscr{C}^{(ij)}_{T',J'}}

\title{Faithful tropicalization of the Grassmannian of planes}

\author[M.A.\ Cueto, M.\ H\"abich, A.\ Werner]{Maria Angelica Cueto,
  Mathias H\"abich \and Annette Werner${}^{\S}$
}
\thanks{${\S}$ \emph{Corresponding author}}

\date{\today}


\begin{document}
\begin{abstract}
  We show that the tropical projective Grassmannian of planes is
  homeomorphic to a closed subset of the analytic Grassmannian in
  Berkovich's sense by constructing a continuous section to the
  tropicalization map.  Our main tool is an explicit description of
  the algebraic coordinate rings of the toric strata of the
  Grassmannian. We determine the fibers of the tropicalization map and
  compute the initial degenerations of all the toric strata.  As a
  consequence, we prove that the tropical multiplicities of all points
  in the tropical projective Grassmannian are equal to one.  Finally,
  we determine a piecewise linear structure on the image of our
  section that corresponds to the polyhedral structure on the tropical
  projective Grassmannian.
\end{abstract}
\maketitle
\section{Introduction}
\label{sec:introduction}

In this paper, we investigate the tropical Grassmannian of planes in
$n$-space from the point of view of non-Archimedean analytic
geometry. The deep relations between tropical and non-Archimedean
analytic geometry have been studied by several authors, including
Einsiedler, Kapranov and Lind~\cite{EKL},
Gubler~\cite{gub07-2,gub07-1} and Payne~\cite{PayneAnalytification}.
Analytic spaces in this context are mostly Berkovich analytic spaces,
where, roughly speaking, points can locally be described by certain
seminorms. Spaces of seminorms or valuations have been present in
tropical geometry from its very beginnings by the work of Bieri and Groves~\cite{BG}. Recently, such
spaces have been used by Manon to investigate representations of
reductive groups~\cite{Manon10, Manon11}.

Of particular interest for the present paper is the recent work by
Baker, Payne and Rabinoff~\cite{BPR}, which contains a detailed study
of the connections between tropicalizations and skeleta of Berkovich
analytic curves. Our study of the Grassmannian of planes provides
higher-dimensional results in the same direction.

The tropicalization map on an analytic closed subvariety of a toric
variety is continuous and surjective, and it strongly depends on a
choice of coordinates of $X$. Work of Payne shows that the Berkovich
space $X^{\an}$ associated to a closed subvariety $X$ of a toric
variety is homeomorphic to the projective limit of all
tropicalizations of $X$~\cite[Theorem~4.2]{PayneAnalytification}.

Berkovich introduced skeleta of analytic spaces, roughly speaking, as
polyhedral subsets that are deformation retracts of the whole
space~\cite{ber99}.  For concrete examples, we refer to
Section~\ref{subsec:analyti}. These piecewise linear substructures of
analytic spaces were used in~\cite{ber99, ber04} to prove local
contractibility of smooth analytic spaces.

If $X$ is a curve, the corresponding Berkovich space $X^{\an}$ can be
endowed with a polyhedral structure locally modeled on an
$\mathbb{R}$-tree. The complement of its set of leaves carries a
canonical metric. As Baker, Payne and Rabinoff have shown, every
finite subgraph $\Gamma$ of this complement maps isometrically to a
suitable tropicalization of $X$~\cite[Theorem 6.20]{BPR}, where the
metric on this tropicalization is given locally by lattice lengths. We
say that this tropicalization represents $\Gamma$ faithfully.  When
all tropical multiplicities equal one, any compact connected subset of
a tropicalization is the isometric image of a suitable subgraph of
$X^{\an}$~\cite[Theorem 6.24]{BPR}.  The proofs of these two
statements rely on deep results concerning the structure theory of
analytic curves and their semistable reduction theory. Some of them
were developed by Thuillier in the context of potential theory on
curves~\cite{thu}.

The former results can be seen as a comparison between two polyhedral
approximations of analytic curves. The first one is given by all
tropicalizations (which approximate the analytic space by Payne's
theorem cited above), whereas the second one comes from skeleta of
semistable models. An interesting challenge is to look for a
comparison of these two different polyhedral approximations for higher
dimensional varieties. Two major difficulties arise in this
case. First, there is no polyhedral description of the full Berkovich
space generalizing the $\mathbb{R}$-tree description for curves.  And
second, there are no semistable models available in general.  Also, in
higher dimensions, skeleta are endowed with piecewise linear
structures and not with canonical metrics.  Still, we can ask the
following natural question. Let $X$ be a closed subvariety of a toric
variety. Is there a continuous map from the tropicalization of $X$ to
$X^{\an}$ that is a section to the tropicalization
map? 
If the answer is yes, we call such a
tropicalization \emph{faithful}.

In this paper we show that the tropicalization of the Grassmannian of
planes $\Gr(2,n)$ induced by the Pl\"ucker
embedding~\eqref{eq:Pluecker} into $\pr_{K}^{\binom{n}{2}-1}$
satisfies this property, where $K$ is a complete non-Archimedean
field.  Let us describe our main results in more detail.  Via the
Pl\"ucker map, we embed $\Gr(2,n)^{\an}$ into the analytified
projective space $(\pr_K^{\binom{n}{2}-1})^{\an}$. The composition
with the coordinatewise logarithmic absolute value gives a continuous
map called tropicalization, i.e.,\ $\trop\colon \Gr(2,n)^{\an}
\rightarrow (\overline{\mathbb{R}}^{\binom{n}{2}}\smallsetminus
\{(-\infty,\ldots, -\infty)\})/ \mathbb{R}\tightcdot\mathbf{1}$, where
$\mathbf{1}$ is the all-ones vector.  Here, we put
$\overline{\mathbb{R}} = \mathbb{R} \cup \{- \infty\}$ and set the
logarithmic absolute value of zero to be $- \infty$.  The image of
$\trop$ is the tropical Grassmannian $\cTG(2,n)$. By $\Gr_0(2,n)$ we
denote the Zariski open subset of the Grassmannian mapping to the
complement of the coordinate hyperplanes under the Pl\"ucker
embedding.  Its tropicalization $\cTGo(2,n):= \trop (\Gr_0(2,n)^{\an})
\subset \mathbb{R}^{\binom{n}{2}} / \mathbb{R}\tightcdot\mathbf{1}$
was introduced and studied by Speyer and Sturmfels~\cite{TropGrass}.
Investigating the full space $\cTG(2,n)$ introduces technical
difficulties involving the boundary $\cTG(2,n)\smallsetminus
\cTGo(2,n)$, which we explain in Section~\ref{sec:cont-sect-from}.

\smallskip

The main result in this paper can be phrased as follows:
\begin{thm}\label{thm:MainThm}
  There exists a continuous section $\sigma\colon \cTG(2,n)\to
  \Gr(2,n)^{\an}$ to the tropicalization map $\trop \colon
  \Gr(2,n)^{\an} \to \cTG(2,n)$. Hence, the tropical Grassmannian
  $\cTG(2,n)$ is homeomorphic to a closed subset of the Berkovich
  analytic space $\Gr(2,n)^{\an}$.
\end{thm}
This section is constructed locally and it is defined by skeleton maps
on affine spaces, after a suitable choice of local coordinates. Our
construction relies on the interpretation of the tropical Grassmannian
as a compactification of the space of phylogenetic trees, extending
earlier work of Speyer and Sturmfels~\cite{TropGrass}. To be more
precise, we define $\sigma$ on a covering of $\cTG(2,n)$ by polyhedral
cones which we call $\{\CijTJ\}_{ij,T,J}$. Each cone consists of those
points in $\cTG(2,n)$ associated to the combinatorial type of a tree
$T$ whose only non-finite coordinates involving either $i$ or $j$ are
those contained in $J$.  The pair $ij$ of distinct indices lies
outside $J$.  For every such cone $\CijTJ$, we construct an
algebraically independent subset $I$ of the Pl\"ucker coordinates of
cardinality $2(n-2)$ such that the corresponding affine space
$\mathbb{A}^I_K$ intersects the Grassmannian in a Zariski open
subvariety. The Berkovich skeleton of $(\mathbb{A}^I_K)^{\an}$ can be
identified with $\overline{\RR}^{2 (n-2)}$. The natural inclusion of
the skeleton $\overline{\RR}^{2(n-2)}$ into $(\mathbb{A}^I_K)^{\an}$
induces the section $\sigma$ on $\CijTJ$.  In order to prove that this
map is well-defined, we check by direct computation that for each
$x\in \cTG(2,n)$, the point $\sigma(x)$ is the unique maximal element
in the fiber $\trop^{-1}(x)$ with respect to evaluation on rational
functions (Lemma~\ref{lm:MaxCondition}). We show the continuity of
$\sigma$ in Theorem~\ref{thm:continuity}.

\smallskip

The Pl\"ucker embedding induces a stratification of $\Gr(2,n)$ by
subvarieties of tori.  For every subset $J$ of the Pl\"ucker
coordinates, we denote by $\Gr_J(2,n)$ the subvariety of $\Gr(2,n)$
where precisely the Pl\"ucker coordinates in $J$ vanish. For example,
when $J$ is empty, we recover the open subset $\Gr_0(2,n)$. In
Lemma~\ref{lem:iso}, we use the local coordinate systems $I$ mentioned
above to describe the coordinate ring of $\Gr_J(2,n)$. This step is
crucial to characterize the fiber of $\trop$ over a point $x$ in
$\cTG(2,n)$ as the Berkovich spectrum of an affinoid algebra. This is
the content of Proposition~\ref{prop:fiber} and
Theorem~\ref{thm:reduction}.  Corollary~\ref{cor:Shilov} shows that
each of these affinoid algebras has a unique Shilov boundary point,
which is precisely $\sigma(x)$.  This gives a conceptual explanation
for the fact that $\sigma(x)$ is maximal in the fiber $\trop^{-1}(x)$.

In Section~\ref{sec:trop-mult-integr} we focus our attention on
piecewise linear structures.  In Lemma~\ref{lm:InGrobner}, we use the
local coordinates $I$ to compute the initial degenerations of each
subvariety $\Gr_J(2,n)$. Theorem~\ref{thm:InitDegGrass} states that
these degenerations are integral schemes, and that they coincide with
the reductions of the affinoid algebras given by the fibers of
tropicalizations.  Each stratum $\cTG_J(2,n)$ is endowed with the
induced Gr\"obner fan structure associated to its defining ideal.  In
Corollary~\ref{cor:multiplicities}, we show that the tropical
multiplicity of every point in the tropical Grassmannian, evaluated in
the corresponding stratum $\Gr_J(2,n)$, is equal to one.  This
provides an example of the general comparison
theorem~\cite[Proposition 4.24]{BPR} as well as a test case for a
higher-dimensional version of~\cite[Theorem 6.23]{BPR}.

On the analytic side, Corollary~\ref{cor:skeleton} states that the
image of the section $\sigma$ is a skeleton in the sense of
Ducros~\cite{DucrosPL}.  Note that we do not consider skeleta of
semistable models of the Grassmannian, although it seems likely that
the image of $\sigma$ can be identified with a skeleton of such a
model. Therefore there is, a priori, no natural piecewise linear
structure on the image of $\sigma$.  However, the properties of
$\sigma$ allow us to define a natural piecewise linear structure on
this set. We do so at the end of Section~\ref{sec:trop-mult-integr}.


The rest of the paper is organized as follows. In
Section~\ref{sec:trop-analyt-vari}, we recall basic facts on tropical
and analytic varieties.  In Section~\ref{sec:trop-grassm}, we describe
the combinatorics and topological structure of the tropical
Grassmannian of $2$-planes in $n$-space, following the seminal work of
Speyer and Sturmfels~\cite{TropGrass}.
Section~\ref{sec:cont-sect-from} contains the proof of
Theorem~\ref{thm:MainThm}. In Section~\ref{sec:fibers-trop} we give an
explicit description of the fibers of the tropicalization map in the
language of affinoid domains.  Section~\ref{sec:trop-mult-integr}
deals with piecewise linear structures.  In
Section~\ref{sec:petersen-graph} we discuss how the construction of
the section $\sigma$ leads to an embedding of the quotient of the
tropical fan $\cTGo(2,n) $ by its $(n-1)$-dimensional lineality space
into a quotient of $\Gr(2,n)^{\an}$ by a torus action.  This
application is motivated by a question of Sturmfels to the third
author.

We hope that the exhaustive study of the tropical and analytic
Grassmannian carried out in this paper will provide a helpful
guideline to obtain further general results along the lines
of~\cite{BPR} for higher-dimensional varieties.  Finally, we believe
that the local coordinates introduced to explicitly construct our
section will be a useful tool to further investigate the algebraic
Grassmannian.

\section{Analytification and tropicalization}
\label{sec:trop-analyt-vari}

Throughout this paper, we let $K$ be a field which is complete with
respect to a non-Archimedean valuation $\nu\colon
K^{\ast}=K\smallsetminus \{0\}\to \RR$.  This valuation induces an
absolute value $|\,\mathord{\cdot}\,|= \exp(-\nu(\mathord{\cdot}))$ on
$K$.  Note that we allow $K$ to be an arbitrary field endowed with the
trivial absolute value. Other examples include non-Archimedean local
fields such as $\mathbb{Q}_p$, the completion $\CC_{p}$ of the
algebraic closure of $\mathbb{Q}_p$, the field $\mathbb{C}(\!(t)\!)$
of formal power series, and the field $\CC\{\!\{t\}\!\}$ of Puiseux
series.

In order to study tropicalizations of projective varieties, it is
convenient to extend the field $\RR$ as well as the valuation $\nu$
from $K^{\ast}$ to $K$, setting $\nu(0)=\infty$. Let
$\overline{\RR}=\RR\cup \{-\infty\}$ be the extended field of real
numbers. It is an additive monoid, and its topology extends the
Euclidean topology on $\RR$. The half-open intervals $[-\infty, a)$
for $a\in \RR$ give a basis of open neighborhoods of $-\infty$.

\smallskip In what follows, we use multi-index notation. More
precisely, given $x=(x_1,\ldots, x_n)\in K^n$ and $\alpha\in \ZZ^n$,
we write $x^{\alpha} = \prod_{i = 1}^n x_i^{\alpha_i}$ and $|\alpha| =
\alpha_1 + \ldots + \alpha_n$.

\subsection{Analytic spaces}\label{subsec:analyti}
Let us recall some basic facts about Berkovich analytic spaces.  The
\emph{analytification functor} associates to every $K$-scheme of
finite type an analytic space over $K$~\cite[Sections~3.4
and~3.5]{berkovichbook}. If $X = \Spec B$ is an affine $K$-scheme of
finite type, then $X^{\an}$ can be identified with the set of all
multiplicative seminorms on $B$ extending the absolute value on $K$~\cite[Remark 3.4.2]{berkovichbook}.  Here, a \emph{multiplicative
  seminorm} is a map of multiplicative monoids
$\|\mathord{\cdot}\|\colon B\to \RR_{\geq 0 }$ sending zero to zero
and satisfying the non-Archimedean triangle inequality, i.e., $\|f+g\|
\leq \max\{\|f\| , \|g\|\}$ for all $f,g\in B$.  The space $X^{\an}$
has a natural topology, namely, the coarsest one such that all
evaluation maps $\operatorname{ev}_f\colon \|\mathord{\cdot}\| \mapsto
\|f\|$ with $f$ in $B$ are continuous. When $X$ is a general
$K$-scheme of finite type, $X^{\an}$ is constructed by gluing the
analytifications on any open affine cover~\cite[Proof of
Theorem~3.4.1\,(3)]{berkovichbook}.

As we mentioned in the Introduction, a \emph{skeleton} of an analytic
space is a polyhedral subset satisfying a finiteness
condition. 
Rather than giving the precise definition, we focus on the following
example, which we thoroughly use in the sequel.

\begin{example}[
  {Affine $n$-space}]\label{ex:affinespace}
  The analytic affine $n$-space $(\AA_K^n)^{\an}$ consists of all
  multiplicative seminorms on 
  $K[x_1, \ldots,
  x_n]$ extending the absolute value of $K$.  Given any point
  $\rho
  \in \overline{\RR}^n$, we define a
  multiplicative seminorm $\delta(\rho)$ on $K[x_1,\ldots, x_n] $ as
  follows:
\begin{equation*}
  \label{eq:6}
\delta(\rho)\colon K[x_1,\ldots, x_n] \longrightarrow \RR_{\geq 0} \qquad
  \delta(\rho) (\sum_{\alpha} c_{\alpha}{x}^{\alpha})
 =\max_\alpha \{ |c_\alpha| 
 \exp(\sum_{i=1}^n{\rho_i \alpha_i})\}.
\end{equation*}
Here, we put $\exp(- \infty) = 0$.  The map $\delta$ is continuous and
injective. The image of $\delta$ is the \emph{skeleton} of the
analytic affine space, so we call $\delta$ the \emph{skeleton
  map}. Note that $\delta$ and the corresponding skeleton depend on
the choice of affine coordinates on $\AA^n_K$.

The restriction of $\delta$ to $\mathbb{R}^n$ is a continuous map
$\delta\colon \mathbb{R}^n \to
(\mathbb{G}_m^n)^{\an}$.  The image of $\delta$ is called the skeleton
of the split torus $(\mathbb{G}_m^n)^{\an}$. It is a deformation
retract of $(\mathbb{G}_m^n)^{\an}$.
\end{example}

\begin{example}[Projective $n$-space]\label{ex:ProjAnalytic}
  The analytic projective space $(\mathbb{P}^n_K)^{\an}$ can be
  obtained by gluing the analytifications of the standard open
  covering of $\pr^n_K$ by $n+1$ copies of $\AA_K^{n}$. It can also be
  described with the following equivalence relation of
  $(\AA_K^{n+1})^{\an} \smallsetminus \{0\}$. Two elements
  $\|\mathord{\cdot}\|_1$ and $\|\mathord{\cdot}\|_2$ in
  $(\AA_K^{n+1})^{\an}\smallsetminus \{0\}$ are \emph{equivalent} if
  for every positive integer $d$ there exists a constant $C=C(d)\in
  \RR_{>0}$ such that $\|f\|_1=C^d\|f\|_2$ for any homogeneous
  polynomial $f$ in $K[x_0, \ldots, x_n]$ of degree $d$. The
  topological space $(\pr^n_K)^{\an}$ can be identified with the set
  of equivalence classes in $(\AA_K^{n+1})^{\an} \smallsetminus
  \{0\}$, equipped with the quotient topology.
\end{example}

\subsection{Tropicalizations}\label{subsec:tropi}
We now introduce some notation for tropicalizations.  For every
$K$-split torus $\GG_m^n$ and every basis $\{\chi_1, \ldots, \chi_n\}$
of its character lattice $\Lambda$, we define a tropicalization map
\begin{equation}\label{eq:2}
\trop\colon (\GG_m^{n})^{\an} \longrightarrow \RR^n\qquad \gamma \longmapsto (\log|\chi_1(\gamma)|, \ldots, \log|\chi_n(\gamma)|).
\end{equation}
\begin{definition}
  If $X$ is a closed subscheme of $\GG_m^n$, the \emph{tropical variety} $\cT
  X$ is  the image of $X^{\an}$ under the tropicalization
  map.
\end{definition}
These notions extend the classical tropicalizations of subvarieties of
tori by valuation maps~\cite{ElimTheory}. Indeed, let $L|K$ be an
algebraically closed, complete non-Archimedean valued field with
nontrivial valuation extending the valuation on $K$.  Every closed
point $x$ in $X(L)$ induces a seminorm $\gamma_x$ in $X^{\an}$,
sending $f\mapsto |f(x)|$, where $f\in K[X]$. The restriction of the
tropicalization map~\eqref{eq:2} to $X(L)$ is the negative of the
coordinatewise valuation map on the split torus over $L$. By the
Fundamental theorem of tropical geometry, the closure of $\trop(X(L))$
coincides with $\cT X$ (see~\cite[Proposition~3.8]{gubler12}
and~\cite{DraismaTropSecant,EKL,FibersOfTrop}).

In order to define tropicalizations of projective varieties, we first
need to describe tropical projective space. We write $\mathbf{1}$ for
$(1,1,\ldots,1)$ and $-\mathbf{\infty}\tightcdot\mathbf{1}$ for
$(-\infty, -\infty, \ldots, -\infty)$.
\begin{definition}\label{def:tropProjSpace}
 The \emph{tropical projective space} is the topological space
\begin{equation*}
\TP\pr^n = (\overline{\RR}^{n+1}
\smallsetminus \{-\mathbf{\infty}\tightcdot\mathbf{1}\}) /
\RR\tightcdot\mathbf{1}\text{,}\label{eq:1}
\end{equation*}
endowed with the quotient topology.
\end{definition}

Let $\{x_0, \ldots, x_n\}$ be the projective coordinates on
$\pr^n_K$. Using Example~\ref{ex:ProjAnalytic}, we define the
associated tropicalization map on $(\pr^n_K)^{\an}$:
\begin{equation}\label{eq:3}
  \trop \colon (\pr^n_K)^{\an} \longrightarrow  \TP\pr^n \qquad  [\gamma]\longmapsto
  (\log\gamma(x_0), \ldots,\log \gamma(x_n)) + \RR\tightcdot\mathbf{1}\text{.}
\end{equation}
Here, $[\gamma]$ denotes the equivalence class of a nonzero
multiplicative seminorm $\gamma$ on $K[X_0, \ldots, X_n]$.  By
convention, we set $\log 0 = - \infty$.
\begin{definition}
  Given a projective variety $X \subset \pr^{n}_K$, we define its
  \emph{tropicalization} $\cT X$ as the image of $X^{\an}$ in
  $\TP\pr^n$ under the tropicalization map~\eqref{eq:3}.
\end{definition}

\section{The  tropical Grassmannian \texorpdfstring
  {$\cTG(2,n)$}{T Gr(2,n)}}
\label{sec:trop-grassm}

Our object of study is the Grassmannian $\Gr(2,n)$ of $2$-planes in
$n$-space. We view this variety inside projective space via the
Pl\"ucker map, which we now recall. We write $[n] = \{1,\ldots,n\}$
and $ij \in \binom{[n]}{2}$ for a subset $\{i,j\} \subset [n]$ of size
$2$.  The \emph{Pl\"ucker map} is given by the formula
\begin{equation*}
  \tilde{\varphi}\colon \mathbb{A}_K^{2 \times n} \dashrightarrow \pr^{\binom{n}{2}-1}_K
  \qquad \tilde{\varphi}(X)=\big(\det(X^{(ij)})\big)_{i<j}\text{,}
\end{equation*}
where $X^{(ij)}$ denotes the submatrix of $X$ obtained by choosing the
$i$th and $j$th columns of $X$. The locus of definition of $\tilde{\varphi}$
is the open set of matrices of rank two, which we identify with
the set of $2$-planes in $n$-space. This map induced the \emph{Pl\"ucker embedding}
\begin{equation}
 \label{eq:Pluecker} 
\varphi\colon \Gr(2,n)\to \pr^{\binom{n}{2}}_K.
\end{equation}

We set $p_{ij}:= \det(X^{(ij)})$.  As it is customary, for each pair
$\{i,j\}$, we choose only one projective coordinate among $p_{ij}$ and
$p_{ji}$. They are related by the identity $p_{ij} = -
p_{ji}$. Abusing notation, we consider each index as an unordered
pair.

The \emph{Pl\"ucker ideal} $I_{2,n}$ is the homogeneous prime ideal in
the polynomial ring $K[p_{ij}: ij\in \binom{[n]}{2}]$ of all algebraic
relations among the $2\times 2$-minors of a generic $2 \times
n$-matrix $X$. This ideal is generated by quadrics with coefficients
in $\ZZ$ and it admits a quadratic Gr\"obner basis~\cite[Theorem
3.1.7]{StInv93}.  A particularly nice system of generators for the
ideal $I_{2,n}$ is given by the three-term Pl\"ucker relations
\begin{equation}
\label{eq:3TermPluecker}
p_{ij}p_{kl}-p_{ik}p_{jl}+p_{il}p_{jk} = 0 \qquad \text{ for } ijkl \in 
\binom{[n]}{4}.
\end{equation}
We cover the variety $\Gr(2,n)$ by $\binom{n}{2}$ big open cells, all
of which are isomorphic to $\AA_K^{2(n-2)}$:
\begin{equation}
  \label{eq:4}
  U_{ij} = \varphi^{-1} (\{ p_{ij}\neq 0\}) \qquad \text{ for } ij\in \binom{[n]}{2}.
\end{equation}
\noindent The intersection of all $U_{ij}$ is the
open subvariety
\begin{equation*}
  \Gr_0(2,n):=\varphi^{-1} (\GG_m^{\binom{n}{2}} / \GG_m)\text{,}
\end{equation*}
where $\GG_m$ is embedded diagonally.

We tropicalize the Grassmannian $\Gr(2,n)$ and the open subvariety
$\Gr_0(2,n)$ with respect to its Pl\"ucker embedding by means of the
tropicalization map \eqref{eq:3}. Similarly, for each pair $ij$, we
define $\cT U_{ij}$ as the image of $U_{ij}^{\an}\subset \Gr(2,n)$
under the same map. In particular, we obtain
\begin{equation}
  \cT U_{ij}=\{x\in \cTG(2,n): x_{ij}\neq
  -\infty\}. \label{eq:5}
\end{equation}
To analyze these tropical varieties in more detail, we consider the 
\emph{affine
  cone over the Grassmannian} $\Gr_0(2,n)_{\aff} = \Spec
(K[p_{ij}^\pm: ij\in \binom{[n]}{2}]/I_{2,n})$ in
$\mathbb{G}_{m}^{\binom{n}{2}}$.  The resulting \emph{tropical
  Grassmannian} $\cTGo(2,n)_{\aff}$ of $2$-planes in $n$-space is a
polyhedral fan in $\RR^{\binom{n}{2}}$ of pure dimension $2n-3$, all
of whose cones contain the line spanned by the all-ones vector
$\mathbf{1}$. It is a closed subfan of the Gr\"obner fan of
$I_{2,n}$~\cite[Section 2]{BJSST}. The quotient of $\cTGo(2,n)_{\aff}$
by $\mathbb{R}\tightcdot\mathbf{1}$ equals $\cTGo(2,n)\subset
\TP\pr^{\binom{n}{2}-1}$.

In~\cite{TropGrass}, Speyer and Sturmfels identified the space
$\cTGo(2,n)$ with the \emph{space of phylogenetic trees}, the
definition of which we recall now. A phylogenetic tree on $n$ leaves
is a real weighted tree $(T,\omega)$. Here, $T$ is a finite connected
graph with no cycles and with no degree-two vertices, together with a
labeling of its leaves in bijection with $[n]$. All the trees
in this paper will be labeled. We define the set of inner edges of $T$
as those that do not end in a leaf of $T$.  The weight function
$\omega\colon E(T)\to \RR$ is defined on the set of edges of $T$.  We
impose the condition that the weight of every inner edge of $T$ is
nonnegative. The tree $T$ is called the \emph{combinatorial type} of
the phylogenetic tree $(T, \omega)$.

Given a phylogenetic tree $(T,\omega)$, we construct a distance
function on $[n]$ as follows. For any pair of leaves $i,j$ in $T$, we
let $x_{ij}$ be the sum of the weights $\omega(e)$ of all edges $e$ in
the path from $i$ to $j$ in $T$. Note that this number could be
negative.  Tree-distance functions are characterized by the four-point
condition~\cite[Theorem 1]{Buneman74}.
\begin{thm}[Four-point condition]\label{thm:4ptCondition}
  A point $x\in \RR^{\binom{n}{2}}$ is the distance function
  associated to a phylogenetic tree on $n$ leaves if and only if for
  all pairwise distinct indices $i, j, k, l \in [n]$, the maximum
  among
\begin{equation}\label{eq:4-point-cond}
  x_{ij} + x_{kl}, \quad x_{ik} + x_{jl}, \quad x_{il} + x_{jk}
\end{equation}
is attained at least twice.
\end{thm}
\noindent
As we discussed earlier, any $K$-point $x$ in the Grassmannian
$\Gr(2,n)$ satisfies the Pl\"ucker
relations~\eqref{eq:3TermPluecker}. It is immediate to check that
$\trop(x)$ satisfies Theorem~\ref{thm:4ptCondition}.

Assuming that $i,j,k$ and $l$ are distinct, we can look at the subtree
of $T$ spanned by these four leaves. Such a subtree is called a
\emph{quartet}.  It is a well-known fact that the list of quartets on
a phylogenetic tree characterizes its combinatorial type~\cite[\S
5.4.2]{treeOfLife}. There are exactly four combinatorial types for
quartet trees. These types are distinguished by the pairs of indices
attaining the maximum
in~\eqref{eq:4-point-cond}.  

The  type of a quartet tree is also
determined by its cherries, which are defined as follows.
\begin{definition}\label{def:cherry}
  Let $T$ be a tree. A \emph{cherry} of $T$ is a pair of leaves
  $\{i,j\}$ in the tree $T$ with the property that they are both
  adjacent to the same node in $T$. Equivalently, the induced path
  from $i$ to $j$ in $T$ has only one internal node.
\end{definition}
\noindent 
For example, the unique cherries in the caterpillar tree on $n$ leaves
depicted on the left of Figure~\ref{fig:caterpillar} are $\{i,s_1\}$
and $\{j,s_{n-2}\}$. Notice that for \emph{trivalent trees} (i.e.,
where all its internal nodes have degree three) our definition of
cherry coincides with the standard one~\cite[Chapter 2.4]{ASCB}.  We
can determine the cherries of a quartet
by~\eqref{eq:4-point-cond}. Indeed, the pairing of the indices
$i,j,k,l$ realizing the minimal value in~\eqref{eq:4-point-cond} gives
the two cherries of the quartet $\{i,j,k,l\}$. If all three numbers
in~\eqref{eq:4-point-cond} agree, the quartet has no internal edge. We
call such a tree the \emph{star tree} on four leaves.

From~\cite[Theorem~3.4]{TropGrass} we know that the open cell
$\cTGo(2,n)$ equipped with its Gr\"obner fan structure is the space of
phylogenetic trees. The relative interior of each cone in this fan is associated to one
combinatorial type of a phylogenetic tree.  We make this
correspondence explicit by writing $\cTC_T$ for the cone whose relative interior is associated to
the tree $T$.  Each tree has exactly $n$ edges adjacent to its $n$
labeled leaves.  The inclusion of cones in $\cTGo(2,n)$ corresponds to
the coarsening of trees given by contraction of edges. These
contractions come from setting the weights of the corresponding edges
to be zero. In particular, the maximal cones are indexed by trivalent
tree types.  Each point $x$ in $\cTGo(2,n)$ is associated to a
phylogenetic tree $(T,\omega)$, and the weight function $\omega$ can
be recovered from $x$ using the Neighbor-Joining
Algorithm~\cite[Algorithm 2.41]{ASCB}.

\medskip

Our first result says that the open cell $\cTGo(2,n)$ is dense in
$\cTG(2,n)$:
\begin{lemma}\label{lm:TropGrass} 
  The tropical Grassmannian $\cTG(2,n)$ is the closure of $\cTGo(2,n)$
  in $\TP\pr^{\binom{n}{2}-1}$.
\end{lemma}

\begin{proof} Since the tropicalization map is continuous and
  surjective, and $\trop(\Gr_0(2,n)^{\an})=\cTGo(2,n)$, it suffices to
  show that $\Gr_0(2,n)^{\an}$ is dense in $\Gr(2,n)^{\an}$.  This
  follows from the fact that $\Gr_0(2,n)$ is Zariski dense in
  $\Gr(2,n)$, see \cite[Corollary 3.4.5]{berkovichbook}.
\end{proof}

Recall from~\eqref{eq:4} that we can cover the variety $\Gr(2,n)$ by
the open cells $\{U_{ij}: ij\in \binom{[n]}{2}\}$.  In turn, as a
corollary of Lemma~\ref{lm:TropGrass}, we can cover each cone $\cT
U_{ij}$ by the sets $\overline{\cTC_T} \cap \cT U_{ij}$, where we vary
the combinatorial type of the tree $T$. These sets consist of all
points $x= (x_{kl})_{kl}$ in $\cT U_{ij}$ such that there exists a
sequence $(x^{(m)})_m$ in $\cTC_T$ converging to $x$ in $\cT U_{ij}$,
which means
\[\lim_{m\to \infty}x_{kl}^{(m)} - x_{ij}^{(m)} = x_{kl}- x_{ij}
\qquad \text{ for all } kl \neq ij\text{.}
\] 
In particular, this
implies that given any tree type $T$, all points in
$\overline{\cTC_T}\cap U_{ij}$ satisfy the same four-point conditions
characterizing all quartets in $T$. For example, suppose that
$\{k,l\}$ is a cherry of the quartet $\{a,b,k,l\}$ in $T$. Then, all
points $x \in \overline{\cTC_T}\cap \cT U_{ij}$ satisfy
\begin{equation*} \label{eq:four-point-boundary}
x_{kl} + x_{ab} \leq x_{ak} + x_{bl} = x_{al} + x_{bk}. 
\end{equation*}
If we choose $ab=ij$, and $\{k,l\}$ is a cherry of the quartet with leaves
$\{i,j,k,l\}$, the previous expression is equivalent to
\begin{equation*}
  x_{kl} - x_{ij} \leq x_{ik} + x_{jl} -2x_{ij} = x_{il} + x_{jk}-2x_{ij}.\label{eq:34}
\end{equation*}
The latter will be thoroughly used in Section~\ref{sec:cont-sect-from}.
\medskip

\section{A continuous section on the tropical Grassmannian
  \texorpdfstring{$\cTG(2,n)$}{T Gr(2,n)}}
\label{sec:cont-sect-from}
In this section, we prove Theorem~\ref{thm:MainThm}. More precisely,
we construct sections to the tropicalization map on the
open cover $\{U_{ij}\}_{ij}$ of $\Gr(2,n)$ given in~\eqref{eq:4}, and
we show that these maps agree on the overlaps, thus giving the desired
section $\sigma$. 

In order to write explicit local formulas for $\sigma$, we recall the
algebraic description of the coordinate ring of each open set
$U_{ij}$. Fix a pair of indices $i,j$ in $[n]$ and let $R(ij)$ be the
coordinate ring of $U_{ij}$.  Setting $u_{kl}=p_{kl}/p_{ij}$ for every
$kl\neq ij$, we have
\begin{equation}
  R(ij) = K[u_{kl}: kl\in I(ij)]\text{,} \qquad \text{ where }\quad I(ij)=\{il, jl: l\neq i,j\}.
\label{eq:36}
\end{equation}
Note that by~\eqref{eq:3TermPluecker} we have $u_{kl} = u_{ik} u_{jl}
- u_{il} u_{jk}$ for $k,l \notin \{i,j\}$, thus we may view all
$u_{kl}$ as elements in $R(ij)$. We follow this convention throughout
the remainder of this Section.

\begin{rem}\label{rm:MultClosedSet}
  Given any subset $A \subset \binom{[n]}{2}\smallsetminus \{ij\}$, we
  denote by $S_A$ the multiplicatively closed subset of $R(ij)$
  generated by all $\{u_{kl}: kl\in A\}$.  As a corollary
  of~\eqref{eq:36} we see that $K[\Gr_0(2,n)] \simeq R(ij)_{S_{A_0}}$,
  where $A_0=\binom{[n]}{2}\smallsetminus \{ij\}$.
\end{rem}

From Lemma~\ref{lm:TropGrass}, we know that each tropical set $\cT
U_{ij}$ as in~\eqref{eq:5} has a stratification indexed by
combinatorial types of trees with $n$ leaves, namely,
$\{\overline{\cTC_T}\cap \cT U_{ij}\}_T$.  The next easy lemma gives a
necessary and sufficient condition for a map to be section to
$\trop$. Its proof follows from the definition of the tropicalization
map.
\begin{lemma}\label{lm:sectionProperty}
  Let $W \subset \cT U_{ij}$ be a subset, and let $\sigma\colon W \to
  U_{ij}^{\an}$ be any map.  Then, $\sigma$ is a section to $\trop$
  over $W$ if and only if for all $x\in W$,
  $\sigma(x)(u_{kl})=\exp(x_{kl}-x_{ij})$ for all $kl\in
  \binom{[n]}{2}\smallsetminus \{ij\}$.
\end{lemma}

Since the big open cell $U_{ij}$ in the Grassmannian is a
$2(n-2)$-dimensional affine space, the first idea for constructing a
section to the tropicalization map is to use the skeleton map of
$U_{ij}$ as in Example~\ref{ex:affinespace}.  In order to do so, we
need the following projections.
 
  \begin{definition}\label{def:prIandPartialSection}
    Let $I$ be any subset of $\binom{[n]}{2}$ not containing the pair
    $ij$.  We define the projection
\begin{equation*}
  \pi_I\colon \cT U_{ij}  \longrightarrow \overline{\RR}^I \qquad
  [(x_{kl})_{kl \in \binom{[n]}{2}}] \longmapsto (x_{kl}- x_{ij})_{kl \in I}\text{,}\label{eq:8}
\end{equation*}
with the convention that $-\infty -a = -\infty$ for all $a\in
\RR$. Here, $[\mathord{\cdot}]$ denotes the class of a point in
$\TP\pr^{\binom{n}{2}-1}$.
  \end{definition}
\noindent
Note that the map $\pi_I$ is well defined and continuous on $\cT
U_{ij} \subset \TP\pr^{\binom{n}{2}-1}$ because $x_{ij}\neq -\infty$.
However, not all choices of a set $I\subset \binom{[n]}{2}$ of size
$2(n-2)$ will be suitable for our purposes. Indeed, if we compose the
map $\pi_{I(ij)}$ with the skeleton map $\delta\colon
\overline{\mathbb{R}}^{2(n-2)} \rightarrow (\AA_K^{2(n-2)})^{\an}
\simeq U_{ij}^{\an}$ from Example \ref{ex:affinespace}, we do not in
general obtain a section to the tropicalization map over
$U_{ij}^{\an}$. The set $I$ must be picked in a subtle way.  The core
of our proof explains precisely how to find such sets.
  
  \smallskip

  From now on, we consider index sets $I\subset \binom{[n]}{2}$ of
  size $2(n-2)$ not containing $ij$ and such that the variables
  $\{u_{kl}: kl \in I\}$ are algebraically independent in the function
  field of $\Gr(2,n)$. Under this condition, $K[u_{kl}: kl\in I]$ is
  the coordinate ring of an affine space $\AA_K^I$ embedded in the
  ambient projective space $\pr^{\binom{n}{2}-1}_K$.  We let
  $\delta_I\colon \overline{\RR}^I \rightarrow (\mathbb{A}^I_K)^{\an}$
  be the associated skeleton map.
  
  In this situation, given a point $x\in \cT U_{ij}$ we define the map
  $\sigma_I^{(ij)}(x):= \delta_I\circ \pi_I(x)$, i.e.,
\begin{equation}\label{eq:sigma}
  \sigma^{(ij)}_I(x)\colon K[u_{kl}: kl\in I]
  \longrightarrow \RR_{\geq 0}, \qquad \sum_{\alpha}
  c_{\alpha} {u}^{\alpha}\longmapsto \max_{\alpha}\{|c_{\alpha}|
  \exp(\sum_{kl\in I} \alpha_{kl} (x_{kl}-x_{ij}))\}\text{.} 
\end{equation}
Here, $c_{\alpha} \in K$ for all $\alpha$ and $|\mathord{\cdot}|\colon
K\to \RR_{\geq 0}$ is the absolute value on $K$.  Notice that
${\sigma}^{(ij)}_{I}(x)(u_{kl})\neq 0$ for all $kl\in I$ satisfying
$x_{kl}\neq -\infty$.

Given a point $x\in \cT U_{ij}$, our choice of $I$ will depend on
three quantities: the pair of indices $ij$, a combinatorial tree type
$T$ satisfying $x\in \overline{\cTC_T}$, and a set $J\subset
\binom{[n]}{2}$ encoding those coordinates of the point $x$ that have
value $-\infty$. To be more precise, given a point $x\in \cT U_{ij}$,
we define
\begin{equation}
  \label{eq:J}
  J(x):=\{kl\in \binom{[n]}{2}: x_{kl}=-\infty\}.
\end{equation}
We omit the point $x$ when understood from the context.
The set $J(x)$ has the following property:
\begin{lemma}\label{lm:ConditionsIAndJ}
  Let $x$ be a point in $\overline{\cTC_T}\cap \cT U_{ij}$, and assume
  that the tree $T$ is arranged as in the right of
  Figure~\ref{fig:caterpillar}. Let $J = J(x)$. Suppose that there
  exists $t \neq i,j$ satisfying $it,jt\notin J$, and let $a$ be such
  that $t\in T_a$. Then, for every $l\in T_a$, we have that $il \in J$
  if and only if $jl\in J$.
\end{lemma}
\begin{proof}
  The four-point condition on the quartet $\{i,j,l,t\}$ implies that
  $x_{il}+x_{jt}=x_{it}+x_{jl}$. Since $x_{it},x_{jt}\neq -\infty$,
  the statement holds.
\end{proof}

\medskip

The proof of Theorem~\ref{thm:MainThm} will be done in four steps.
Starting from a fixed pair of distinct indices $i,j\in [n]$ and an
affine open $U_{ij}$, Section~\ref{sec:caterpillar-case} gives a map
$\sigma^{(ij)}_{T,I}\colon \overline{\cTC_T} \cap \cT U_{ij} \to
U_{ij}^{\an} $, where $T$ is a caterpillar tree with endpoints $i$ and
$j$ (as in the left of Figure~\ref{fig:caterpillar}) and $I = I(ij)$
is the index set from~\eqref{eq:36}. In
Section~\ref{sec:exist-gener-tree}, we consider general trees. Here,
the candidate section is patched together from various maps. These
maps are defined on locally closed subsets $\CijTJ \subset
\overline{\cTC_T}\cap U_{ij}$, whose points have entries equal to
$-\infty$ on all coordinates indexed by $J\cap I(ij)$, and real values
on all coordinates indexed by $I(ij) \smallsetminus J$. The technical
heart of our proof lies in the construction of suitable index sets $I$
well adapted to both $T$ and $J$. This is the content of
Proposition~\ref{pr:CompatibleI}.  In algebraic terms, our subsets $I$
will allow us to associate to each point $x\in \CijTJ$ a
multiplicative seminorm on a localization of $R(ij)$, by defining on
an isomorphic ring, namely a localization of $K[u_{kl}: kl\in
I]$. This isomorphism guarantees that the affine subspaces $\AA_K^I$
considered above intersect the big open cell $U_{ij}$ in a Zariski
open subset, and so the multiplicative seminorm will yield a point in
the analytic Grassmannian. Defined in this way, the maps
$\sigma_I^{(ij)}$ will induce local sections to the tropicalization
map.

Section~\ref{sec:uniqueness} discusses the uniqueness and continuity
of such maps: the image of a point $x$ in a cone $\CijTJ$ is maximal
with respect to evaluation on functions among all elements in
$\trop^{-1}(x)$. In the terminology of analytic geometry, $\sigma(x)$
is a Shilov boundary point in its fiber (see
Corollary~\ref{cor:Shilov}).  This property ensures that our section
maps agree on the overlaps of the cones covering $\cT U_{ij}$, and
that they are independent of all choices of $T$ and $I$.  We glue them
together to define the desired map $\sigma\colon \cTG(2,n) \to
\Gr(2,n)^{\an}$.  Theorem~\ref{thm:continuity} shows that $\sigma$ is
continuous.


\subsection{The caterpillar case}
\label{sec:caterpillar-case}
Let $T$ be a caterpillar tree on $n$ leaves where $i,j$ are the
endpoint leaves of the backbone of $T$, as in the left of
Figure~\ref{fig:caterpillar}.

\begin{figure}[hbf]
  \centering
  \begin{minipage}[c]{0.45\linewidth}
    \includegraphics[scale=0.45]{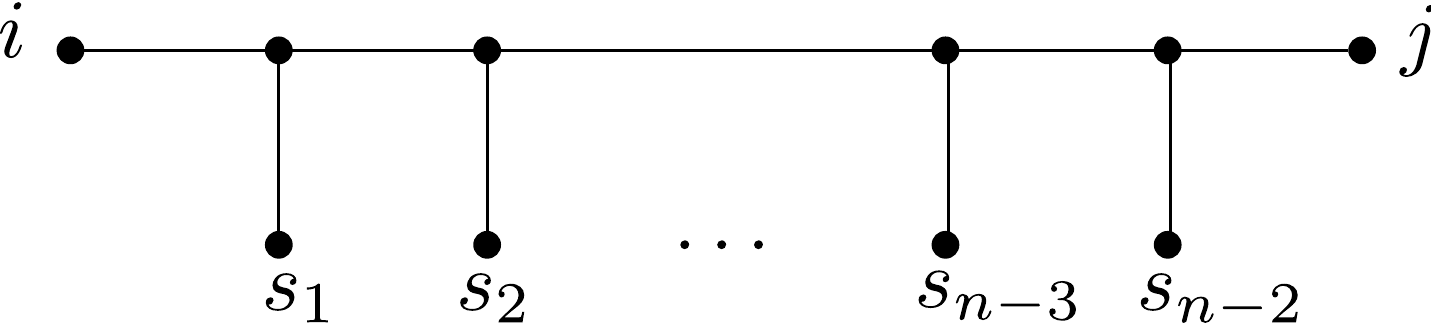}
  \end{minipage}
  \begin{minipage}[c]{0.4\linewidth}
\vspace{-3ex}  \includegraphics[scale=0.55]{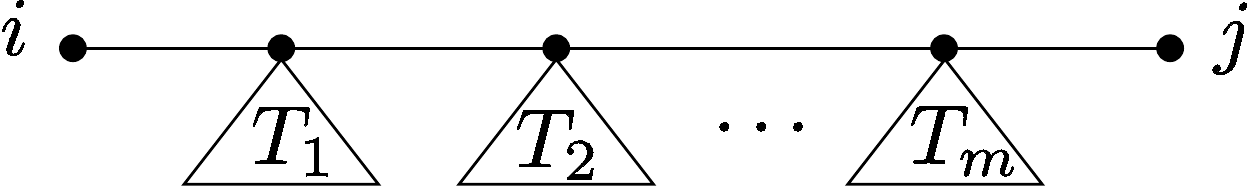}
\end{minipage}
\caption{From left to right: the caterpillar tree on $n$ leaves with
  endpoint leaves $i$ and $j$, and the path from leaf $i$ to $j$ on a
  tree arranged in caterpillar-like form. The labeled triangles
  indicate subtrees of the original tree. The backbone of the
  caterpillar tree is the chain graph with $m+2$ nodes given by the
  horizontal path from $i$ to $j$. The trees $T_1,\ldots, T_m$ need
  not be trivalent.}
  \label{fig:caterpillar}
\end{figure}

The shape of our caterpillar tree ensures that for any pair of indices
$k,l \subset [n]\smallsetminus\{ i,j\}$, the quartets of $T$ with
leaves $\{i,j,k,l\}$ satisfy
\begin{equation}\label{eq:CaterpillarIneq}
  x_{ij}+x_{kl}  \geq
  x_{ik}+x_{jl}\;  \qquad \text{and}\qquad\;
  x_{ij}+x_{kl} \geq
  x_{il}+ x_{jk}\text{,}
\end{equation}
where one of the two inequalities is an equality.  The following
result ensures that the embedding of the skeleton of $U_{ij}$ gives a
section to $\trop$ over $\overline{\cTC_T}\cap \cT U_{ij}$ as
in~\eqref{eq:sigma}.
\begin{prop}\label{pr:SectionCaterpillar}
  Let $T$ be the caterpillar tree with endpoints $i$ and $j$, and let
  $I=I(ij)$. Then, $\sigma^{(ij)}_{T,I} := \sigma_I^{(ij)}\colon
  \overline{\cTC_T} \cap \cT U_{ij}\hookrightarrow \cT U_{ij}
  \rightarrow U_{ij}^{\an}$ is a section to the tropicalization map
  over $\overline{\cTC_T}\cap \cT U_{ij}$.
\end{prop}
\begin{proof} 
  Following~\eqref{eq:36}, we write
  $u_{kl}=u_{ik}u_{jl}-u_{il}u_{jk}\in R(ij)$ whenever $kl\notin
  I(ij)$. The inequalities in~\eqref{eq:CaterpillarIneq} imply that $
  \sigma^{(ij)}_{I} (x)(u_{kl}) = \exp(x_{kl}-x_{ij})$. Indeed,
\begin{equation*}\label{eq:26}
  \begin{split}
    \sigma_I^{(ij)} (x)(u_{ik} u_{jl} - u_{il} u_{jk})=
    \max\{\exp(x_{il}+x_{jk}-2x_{ij}), \exp(x_{ik}+x_{jl}-2x_{ij})\}=
    \exp(x_{kl}-x_{ij}).
  \end{split}
\end{equation*}
Likewise, $\sigma_I^{(ij)}(x)(u_{ik})=\exp(x_{ik}-x_{ij})$ and
$\sigma_I^{(ij)}(x)(u_{jk})=\exp(x_{jk}-x_{ij})$ for $k\neq i,j$.  By
Lemma~\ref{lm:sectionProperty}, the map $\sigma_I^{(ij)}$ is a section
to $\trop$ over $\overline{\cTC_T}\cap \cT U_{ij}$.
\end{proof}

\subsection{Existence for arbitrary trees}
\label{sec:exist-gener-tree}
Let $J = J(x)$ be a vanishing set of a point $x \in \cT U_{ij}$, as
in~\eqref{eq:J}, and fix $J(ij):=J\cap I(ij)$. This set defines a
monomial prime ideal in $R(ij)$
\begin{equation*}
  \label{eq:48}
  \faJij=\langle u_{kl}: kl\in J(ij)\rangle, 
\end{equation*}
and, in turn, a closed irreducible subvariety of the affine space
$U_{ij}$, namely,
\begin{equation*}
  \label{eq:35}
  Y_J(ij)=\Spec
(R(ij)/\faJij).
\end{equation*}
The variety $Y_J(ij)$ maps to $U_{ij} \cap \bigcap_{kl \in J(ij)}
\{u_{kl} = 0\} \subset \mathbb{P}^{\binom{n}{2}-1}$ under the
Pl\"ucker embedding from~\eqref{eq:Pluecker}.  Using these varieties,
we define a family of cones in $\cT U_{ij}$:
\begin{equation*}\label{eq:37} 
  \CijTJ:=\overline{\cTC_{T}} \cap \cT 
  Y_{J}(ij)\cap \!\!\!\!\!\!\bigcap_{kl\in I(ij)\smallsetminus J}\!\!\!\!\!\!\cT U_{kl}\;\;\longhookrightarrow \cT U_{ij}.
\end{equation*}
The collection $\{\CijTJ: T,J\}$ give a stratification of $\cT
U_{ij}$, so its union of $\CijTJ$ over all $T$ and $J$ equals $\cT
U_{ij}$.  Notice that the previous definition can be applied to any
subset $J\subset \binom{[n]}{2}$. However, only subsets of the form
$J=J(x)$ will yield nonempty cones $\CijTJ$.  We give a combinatorial
characterization of these sets in Lemma~\ref{lm:CharJ}.

\medskip

Given the indices $i,j$, we represent our tree $T$ as in the right of
Figure~\ref{fig:caterpillar}. If all the subtrees $T_1,\ldots, T_m$
have exactly one leaf each, then $T$ is a caterpillar tree with
endpoints $i$ and $j$, and Section~\ref{sec:caterpillar-case} tells us
how to construct a section to $\trop$ on $\CijTJ $: we define it
independently of $J$, setting $I=I(ij)$.

If $T$ is an arbitrary tree, the construction of a suitable index set
$I$ is more cumbersome. Indeed, the na\"{\i}ve choice $I=I(ij)$ will
in general not give a section to the tropicalization map. The reason
is very simple.  Suppose that we can pick two elements $k,l\in
[n]\smallsetminus \{i,j\}$ satisfying $x_{ij}+x_{kl} < x_{ik}+x_{jl} =
x_{il}+x_{jk}$. In particular, the leaves $k,l$ must belong to the
same subtree $T_a$.  If $ik,jk,il,jl\in I$, then the defining formula
for the map $\sigma^{(ij)}_{I}(x)$ as in~\eqref{eq:sigma} implies
\begin{equation*}
\begin{aligned}
  \trop \circ\,  \sigma^{(ij)}_{I}(x)(u_{kl})  &= \log(
  {\sigma}^{(ij)}_{I}(x)(u_{ik}\,u_{jl}-
  u_{il}\,u_{jk}))\\ & =\log(\max\{\exp(x_{ik}+x_{jl}-2x_{ij}),
  \exp(x_{il}+x_{jk}-2x_{ij})\}) \\ & =x_{ik}+x_{jl}-2x_{ij} >
  x_{kl}-x_{ij}.
\end{aligned}\label{eq:8bis}
\end{equation*}
Hence, the choice of $I=I(ij)$ does not induce a section to $\trop$ in
this setting. Notice that if one of $ik,jk,il,jl$ belongs to $J$, we
necessarily have
\[x_{ij}+x_{kl} = x_{ik}+x_{jl} = x_{il}+x_{jk} = -\infty\text{,}\]
and so $\trop \circ\, \sigma^{(ij)}_{I}(x)(u_{kl}) = x_{kl}-x_{ij}$.
This shows that the vanishing sets $J$ play a central role in the
construction of $I$.

\medskip

In order to make a systematic choice of $I$, we start by defining a
partial order $\preceq$ on $[n]\smallsetminus \{i,j\}$ that reflects
the combinatorial type of $T$ with respect to the leaves $i$ and
$j$. We denote the corresponding strict order by $\prec$.

\begin{definition}\label{def:orderWithCherryProperty} Let $i,j$ be a pair of indices, and let $\preceq$ be a partial order on the set
  $[n]\smallsetminus \{i,j\}$. Let $T$ be a tree on $n$ leaves
  arranged as in the right of Figure~\ref{fig:caterpillar}.  We say
  that $\preceq$ has the \emph{cherry property on $T$} with respect to
  $i$ and $j$ if the following conditions hold:
\begin{enumerate}[(i)]
\item Two leaves of different subtrees $T_a$ and $T_b$ cannot be
  compared by $\preceq$.
\item The partial order $\preceq$ restricts to a total order on the leaf
  set of each $T_a$, $a=1,\ldots, m$.
\item If $k\prec l\prec v$, then either $\{k,l\}$ or $\{l,v\}$ is a
  cherry of the quartet $\{i,k,l,v\}$ (and hence also of
  $\{j,k,l,v\}$).
\end{enumerate}
  \end{definition}

  The following lemma ensures the existence of partial orders with the
  cherry property on a given
  tree. Figure~\ref{fig:OrderedLabelingInduction} gives an example of
  such a partial ordering for $n=12$ and $m=1$, where the order agrees
  with the standard one on the set $[10]$.
\begin{lemma}\label{lm:order} Fix a pair of indices $i,j$, and let $T$
  be a tree on $n$ leaves. Then, there exists a partial order
  $\preceq$ on the set $[n]\smallsetminus \{i,j\}$ that has the cherry
  property on $T$ with respect to $i$ and $j$.
\end{lemma}
\begin{proof} It suffices to show the result for $m=1$. Indeed, if we
  let $\preceq_a$ be a total order with the cherry property on $T_a$
  for each $a=1,\ldots, m$, then the partial order
  $\mathord{\preceq}=\bigcup_{a=1}^m\mathord{\preceq}_a$ will verify
  the statement.

  Suppose $m = 1$.  Figure~\ref{fig:Induction} illustrates the method
  for building $\preceq$.  We proceed by induction on $n$. If $n=3$,
  then $T_1$ consists of a single leaf $l$, and we set $\preceq$ as
  $\{(l,l)\}$. Assume $n>3$. Then, we know that $T_1$ must contain a
  cherry, say $\{s,t\}$. We arrange $T_1$ vertically as in
  Figure~\ref{fig:Induction}, with one end being the internal node of
  $T$ to which $T_1$ is attached, and the other end equal to the
  cherry $\{s,t\}$.  We declare $t\prec s$. In order to satisfy the
  condition {(iii)} in Definition~\ref{def:orderWithCherryProperty},
  we set $l\prec t$ for all $l\neq t,s$. In addition, if $l\in V_c$
  and $k\in V_b$ with $c<b$, we set $l\prec k$. We indicate this last
  condition by the red vertical arrow and the vertical sign $\prec$ on
  the left of Figure~\ref{fig:Induction}. By the inductive hypothesis,
  we can construct an order $\preceq_b$ with the cherry property on
  each tree spanned by $V_b$ and $\{i,j\}$, where $b=1,\ldots, d$. We
  define the order $\preceq$ as
\[
\bigcup_{b=1}^d\mathord{\preceq}_b \cup \{(t,s),(t,t),(s,s)\} \cup
\{(l,t), (l,s): l\in \bigcup_{b=1}^d V_b\}\cup \{(l,k): l\in V_c, k\in
V_b,c<b\} .
\]
It is easy to check that $\preceq$ satisfies the required properties.
\begin{figure}
  \centering
  \includegraphics[scale=0.5]{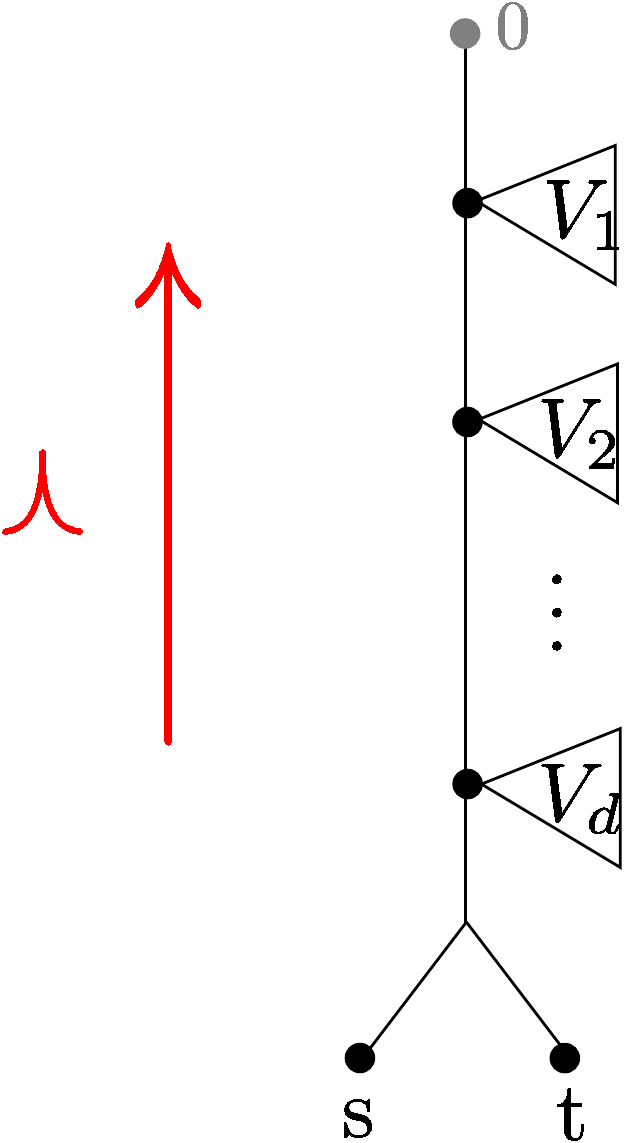}

  \caption{Inductive construction of the partial order $\preceq$,
    starting from a cherry of the subtree $T_1$. Notice that the edge
    connecting the cherry $\{s,t\}$ to the internal node to which
    $V_d$ is attached could be contracted, as in
    Figure~\ref{fig:OrderedLabelingInduction}. The grey node labeled
    with $0$ is the internal node in $T$ to which $T_1$ is
    attached. As before, the edge adjacent to this node could be
    contracted.}
  \label{fig:Induction}
\end{figure}
\end{proof}

\begin{figure}[htb]
  \centering
  \begin{minipage}{0.59\linewidth}
 \hfill   \includegraphics[scale=0.35]{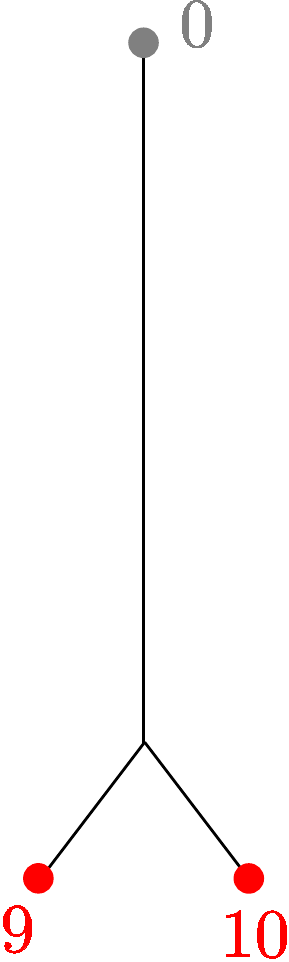}\;
  \includegraphics[scale=0.35]{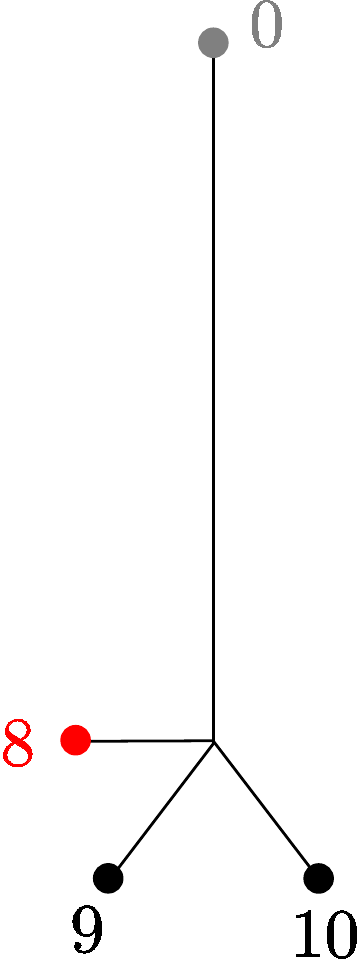}\;
\includegraphics[scale=0.35]{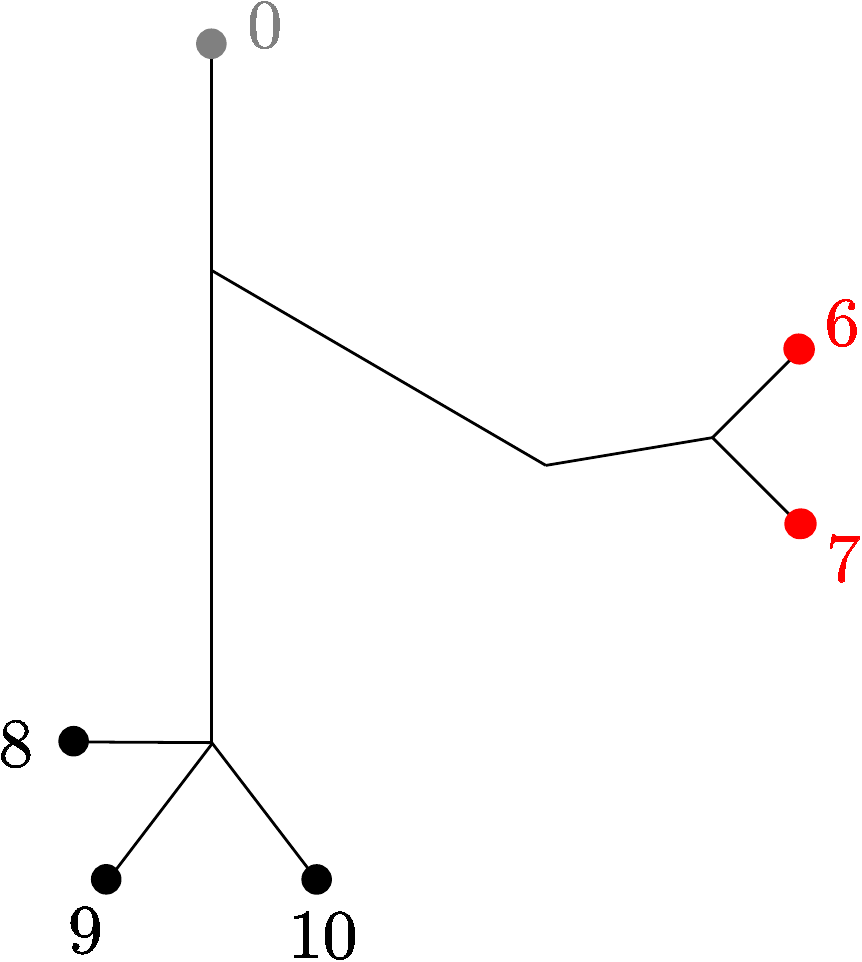}\;
\includegraphics[scale=0.35]{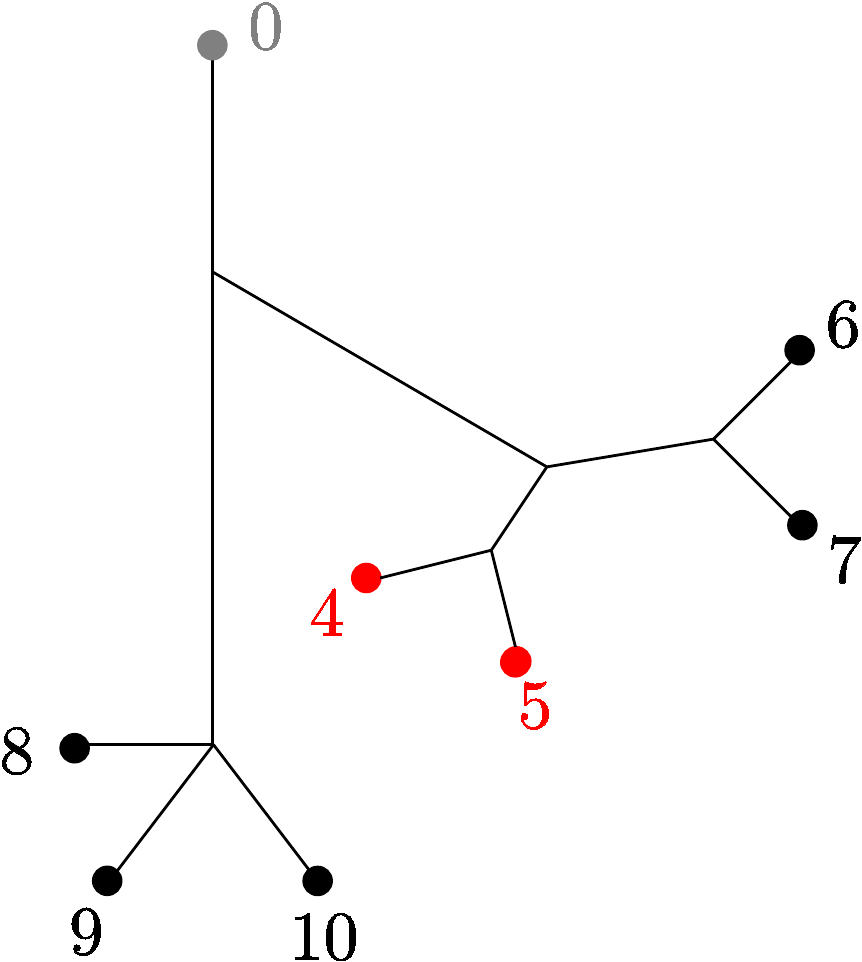}\;

\end{minipage}
  \begin{minipage}{0.4\linewidth}
\includegraphics[scale=0.35]{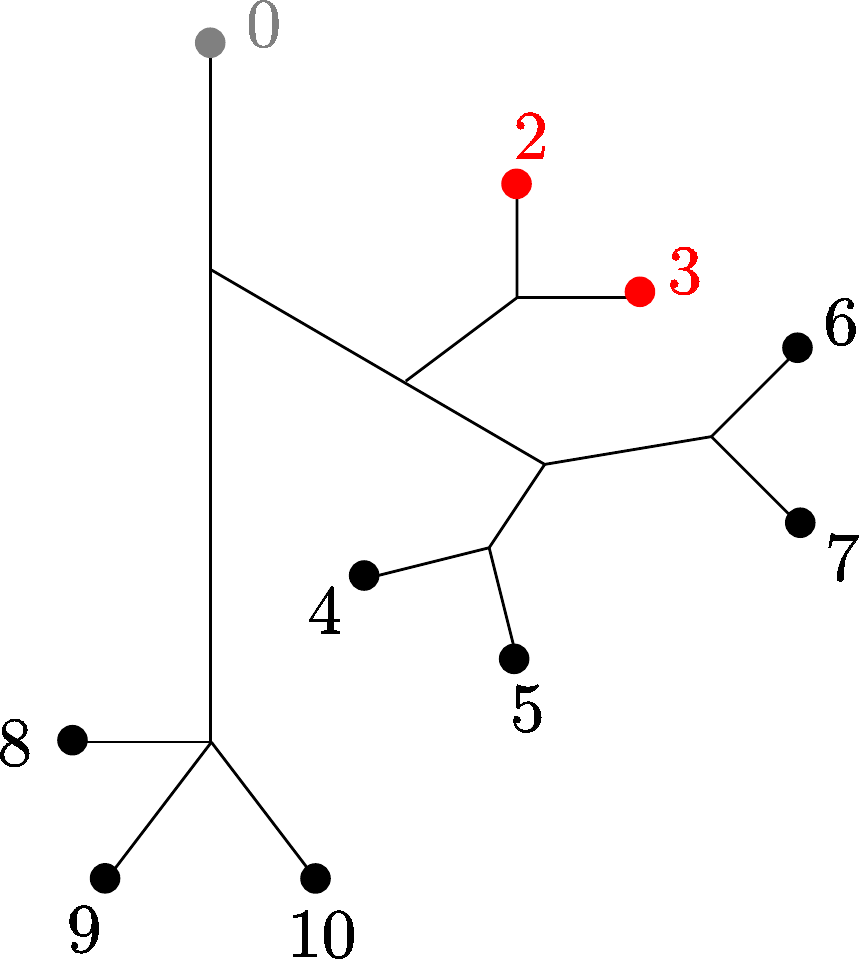}\;
\includegraphics[scale=0.35]{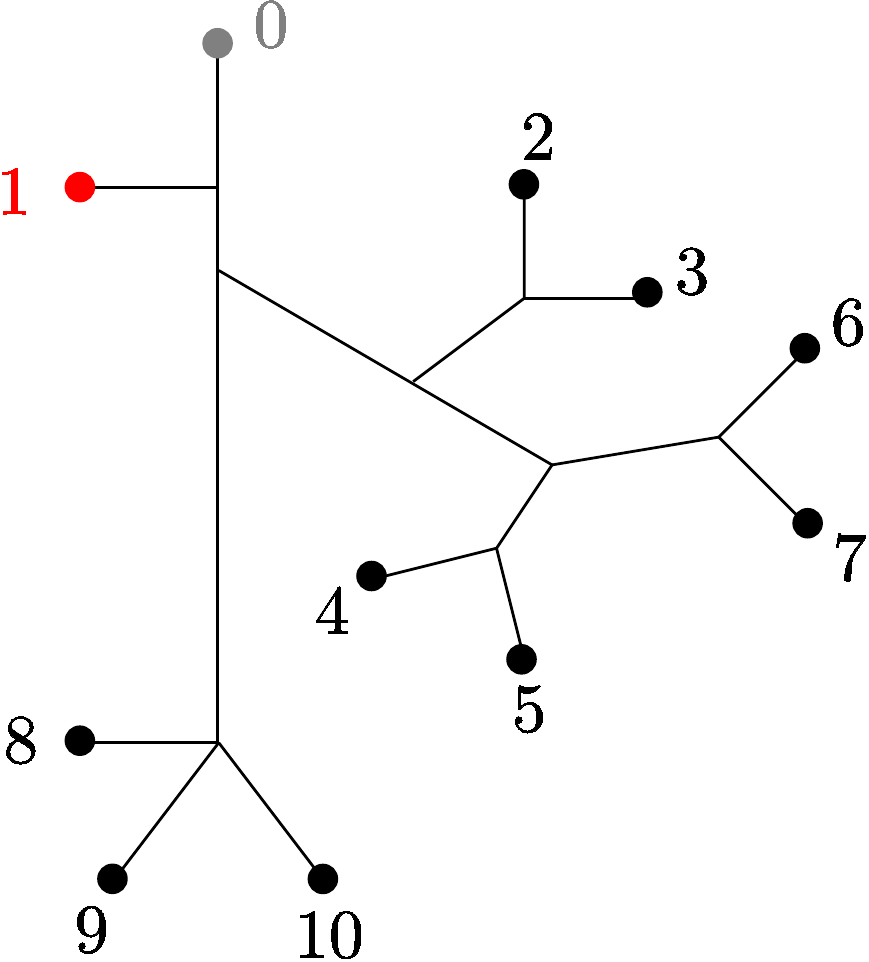}
\end{minipage}
\caption{From left to right: inductive definition of the order
  $\preceq_a$ on the leaves of the subtree $T_a$ in the right of
  Figure~\ref{fig:caterpillar}. We add one leaf or one cherry at a
  time in such a way that the corresponding new leaf or leaves are
  smaller than the previous ones in the order $\preceq_a$. When adding
  a cherry, we arbitrarily order its two leaves as well (both
  orderings will be valid).  The grey dot with label $0$ in $T_a$
  corresponds to an internal node of $T$. Broken leaf edges, such as
  the one in the third tree from the left, should be thought of as
  straight edges. As in Figure~\ref{fig:Induction}, the edge adjacent
  to the grey node with label $0$ could be contracted.}
  \label{fig:OrderedLabelingInduction}
\end{figure}

We now introduce the notion of compatibility of a set with a partial
order and a vanishing set. Recall that $J(ij)=J\cap I(ij)$.
\begin{definition}\label{def:CorrectI}
  Fix two indices $i,j$, a tree $T$ as in the right of
  Figure~\ref{fig:caterpillar}, and a vanishing set $J$. Let $\preceq$
  be a partial order on $[n]\smallsetminus \{i,j\}$ having the cherry
  property on $T$. Let $I\subset \binom{[n]}{2}$ be a set of size
  $2(n-2)$ not containing the pair $ij$. We say that $I$ is
  \emph{compatible with $\preceq$ and $J(ij)$} if for each index
  $a=1,\ldots, m$ and each leaf $k\in T_a$, exactly one of the
  following condition holds:
  \begin{enumerate}[(i)]
  \item $ik$ and $jk\in I$, and for all $l\prec k$ we have $il$ or
    $jl\in J(ij)$; or
\item  $ik\notin I$,
$jl\in I$ for all $l\in T_a$,
  and there exists $t\prec k$ in $T_a$ where $it,jt\notin J(ij)$. If $t$ is the
  maximal element with this property, then $kt\in I$; or
\item $jk\notin I$, $il\in I$ for all $l\in T_a$
  and there exists $t\prec k$ in $T_a$ where $it,jt\notin J(ij)$. If $t$ is the
  maximal element with this property, then $kt\in I$.
  \end{enumerate}
\end{definition}
\noindent
Observe that, for every $k\neq i,j$, at least one of the two elements
$ik,jk$ lies in $I$.  In addition, the following holds. If $s$ is the
maximal element of a subtree $T_a$ of size at least two, and $is,js\in
I$, we can infer that $il$ or $jl\in J$ for every $l\prec s$, and
hence $il,jl \in I$ for every $l \prec s$.

Figure~\ref{fig:compatibility} allows us to give a graphical
explanation of the compatibility property described above. We fix the
tree $T$ with $m=3$ and a partial order $\preceq$ on $[7]$ as in the
picture. For each $a=1,2,3$, we let $I_a:=\{kl\in I: k \text{ or } l\in
T_a\}$. Thus, $I=I_1\sqcup I_2\sqcup I_3$. By construction, since
$|T_1|=1$, we know that $I_1=\{i1,j1\}$ independently of the choice of
$J$.  If $i2$ or $j2$ belong to $J$, then $I_3=\{i2,j2,i3,j3\}$ in
agreement with condition \emph{(i)}. On the contrary, if $i2,j2\notin
J$ then we can choose between $I_3=\{i2,j2,j3,32\}$ (since 
\emph{(ii)} is satisfied) or $I_3=\{i2,j2,i3,32\}$ (so condition
\emph{(iii)} holds). There are many options for $I_2$, depending on
the set $J_2(ij):=\{ik\in J: k\in T_2\}\cup \{jk\in J: k\in T_2\}$. We
provide three examples. If $\emptyset\neq J_2(ij)\subseteq \{i4,j4\}$,
then we can take either $I_2=\{i4,j4,i5,j5,i6,65,i7,76\}$ or
$I_2=\{i4,j4,i5,j5,j6,65,j7,76\}$. Notice that in both cases $i5,j5\in
I_2$ by condition \emph{(i)}. If $\emptyset\neq J_2(ij)\subseteq
\{i7,j7\}$, we can take either $I_2=\{i4,j4,i5,54,i6,65,i7,76\}$ or
$I_2=\{i4,j4,j5,54,j6,65,j7,76\}$. Finally, assume
$J_2(ij)=\{j5,j6\}$. Then, we may choose
$I_2=\{i4,j4,i5,54,i6,64,i7,74\}$ or
$I_2=\{i4,j4,j5,54,j6,64,j7,74\}$.

\begin{figure}[htb]
  \centering
  \includegraphics[scale=0.7]{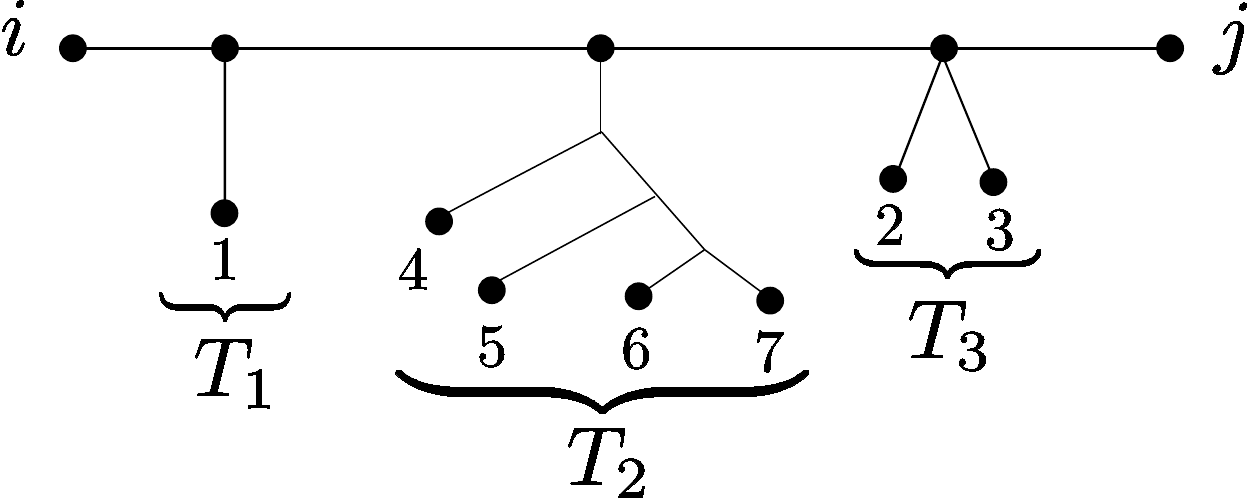}
  \caption{How to choose the set $I(ij)$ in an example of a tree $T$ with $m=3$. We
    define the order $\preceq$ as $\{ (2,3), (4,5), (4,6), (4,7),
    (5,6), (5,7), (6,7)\} \cup \{(k,k): 1\leq k\leq 7\}$. By
    construction, $\preceq$ has the cherry property on $T$ with
    respect to $i$ and $j$. For any choice of $J$, we have $I_1=\{i1,
    j1\}$. The choices of $I_2$ and $I_3$ will strongly depend on
    $J$. We take $I=I_1\sqcup I_2\sqcup I_3$.}
\label{fig:compatibility}
\end{figure}

\medskip Our first result ensures that compatible sets satisfy strong
algebraic properties:
\begin{prop}\label{pr:transBasis}
  Fix a pair of indices $i,j$, and let $T$ be a tree on $n$
  leaves. Let $\preceq$ be a partial order on the set
  $[n]\smallsetminus \{i,j\}$ that has the cherry property on $T$. Fix
  a set $J=J(x)$, where $x\in \overline{\cTC_T}\cap \cT U_{ij}$, and
  let $I$ be a set of size $2(n-2)$ that is compatible with $\preceq$
  and $J(ij)$. Then:
  \begin{enumerate}[(i)]
  \item The set of coordinates $\{u_{kl}: kl \in I\}$ is algebraically
    independent in the function field $\mbox{Quot}(R(ij))$ of the
    Grassmannian.
  \item Let $\A = I \cap I(ij)$. Then, the corresponding polynomial
    subring $K[u_{kl}: kl \in I]$ of $\mbox{Quot}(R(ij))$ fits into a
    diagram of the form
\begin{equation}
  \xymatrix@1{
    K[u_{kl}: kl \in I]\; \ar@{^{(}->}[r]^-{\ii} & R(ij)\; \ar@{^{(}->}[r]^-{\rr} & K[ u_{kl}: kl \in
    I][u_{kl}^{-1}: kl\in \A\smallsetminus J]\text{,}}
  \label{eq:diagram}
\end{equation}
where $\ii(u_{kl})=u_{ik}u_{jl}-u_{il}u_{jk}$ if $kl\notin I(ij)$, and
$\ii(u_{kl})=u_{kl}$ otherwise. Furthermore, $\rr$ is the unique map
compatible with the inclusions of both rings in $\mbox{Quot}(R(ij))$.
\item For each $l\in [n]\smallsetminus\{i,j\}$, we have
  \[ \rr(u_{il}), \rr(u_{jl}) \in K[u_{rs}: rs \in I_{\preceq
    l}][u_{rs}^{-1}: rs\in H_{\preceq l}\smallsetminus J]\text{,} \]
  where $I_{\preceq l}= \{rs \in I: (r \preceq l \mbox{ or } r \in
  \{i,j\}) \mbox{ and } (s \preceq l \mbox{ or } s \in \{i,j\})\}$ and
  $H_{\preceq l} = H \cap I_{\preceq l}$.
  \end{enumerate}
\end{prop}

\begin{proof}
  Notice that \emph{(ii)} implies that the elements $\{u_{kl}: kl \in
  I\}$ in $K[U_{ij}]$ generate the quotient field $\Quot(R(ij))$ as a
  transcendental extension of $K$. Thus, it suffices to prove
  \emph{(ii)} and \emph{(iii)}.

  We first show that the statements hold for $m=1$. We proceed by
  induction on $n$. If $n=3$, then $T$ is a caterpillar tree and
  $I=I(ij)$, and there is nothing to prove because $\rr$ is the
  identity on $R(ij)$.  Assume $n>3$ and $I\neq I(ij)$. By symmetry
  between $i$ and $j$, we suppose that $ik\in I$ and $jk\notin I$ for
  some $k\neq i,j$.  The compatibility of $I$ with $\preceq$ and
  $J(ij)$ ensures that $\{il: l\neq i,j\}\subset I$.

  We let $s$ be the maximal leaf in $T_1$ and set $T'$ and $T'_1$ be
  the trees obtained by removing $s$ from $T$ and $T_1$,
  respectively. Let $\preceq'$ be the order on $T'$ induced by
  $\preceq$, and let $J'=J\cap \binom{[n]\smallsetminus \{s\}}{2}$ and
  $I'=I\cap \binom{[n]\smallsetminus \{s\}}{2} $. Notice that $J'$ is
  the vanishing set of a point in $\cTG(2,n-1)$, namely, $J'=J(x')$,
  where $x'$ is the projection of $x$ to those coordinates not
  containing the index $s$. The point $x'$ belongs to
  $\overline{\cTC_{T'}}\cap \cT U_{ij}'$ in $\cTG(2, n-1)$, where $\cT
  U_{ij}'=\{x_{ij}\neq -\infty\}\subset \TP\pr^{\binom{n-1}{2}-1}$. In
  addition, $\preceq'$ has the cherry property on $T'$ with respect to
  $i$ and $j$.

  We define $R(ij)'=K[u_{ik}, u_{jk}: k\neq i,j,s]$ and $\A'=I' \cap
  I(ij)$. We know that $I'$ is compatible with $\preceq'$ and
  $J'(ij)$, so by the inductive hypothesis, it satisfies \emph{(ii)}
  and \emph{(iii)}. In particular, we can uniquely define $\rr'$ to
  fit into the diagram
\begin{equation*}
  \xymatrix@1{ K[u_{kl}: kl \in I']\; \ar@{^{(}->}[r]^-{\ii'} & R(ij)'\;
    \ar@{^{(}->}[r]^-{\rr'} & K[ u_{kl}: kl \in I'][u_{kl}^{-1}: kl\in
    \A'\smallsetminus J'],}
\end{equation*}
where $\ii'$ is the restriction of $\ii$.
 
We claim that we can extend $\rr'$ to a map $\rr$ satisfying
\eqref{eq:diagram}.  Since we assumed $I \neq I(ij)$ and $\{il: l\neq
i,j\}\subset I$, the observation following
Definition~\ref{def:CorrectI} implies that the element $js$ cannot be
in $I$.  Instead, there exists $t\neq i,j$ such that $st\in I$, so $I
= I' \cup \{is, st\}$. We assign $\rr(u_{is}) :=u_{is}$ and
$\rr(u_{js}):= u_{it}^{-1}(-u_{st}+u_{is}\rr'(u_{jt}))$.  By the
induction hypothesis, we conclude that the image of $\rr$ lies in the
algebra $K[u_{kl}:kl\in I][u_{kl}^{-1}: kl\in \A\smallsetminus J]$.  A
straightforward calculation using the Pl\"ucker relations shows that
$\rr$ and $\ii$ are compatible with the inclusion in the function
field $\Quot R(ij)$.  The statement \emph{(iii)} holds by combining
the construction together with the induction hypothesis.

Assume now that $m>1$. For each $a=1,\ldots, m$, we let $I_a$ be the
subset of $I$ consisting of all pairs involving $i,j$, and the leaves
of $T_a$. To simplify notation, we identify each tree $T_a$ with its
set of leaves.  Set $I(ij)_a=\{ik,jk: k\text{ is a leaf in }T_a\}$,
$R(ij)_a=K[u_{kl}: kl\in I(ij)_a]$, and let $\A_a=I_a\cap I(ij)_a$.
Similarly, we let $\preceq_a$ be the restriction of $\preceq$ to the
leaves of $T_a$ and set $J(ij)_a = J\cap I(ij)_a$. The order
$\preceq_a$ has the cherry property on the tree spanned by
$\{i,j\}\cup T_a$ with respect to $i$ and $j$. Similarly, $J(ij)_a$ is
the vanishing set of the point $\pi_a(x)$, where
$\pi_a\colon\TP\pr^{\binom{n}{2}-1}\dashrightarrow
\TP\pr^{\binom{|T_a|+2}{2}-1}$ is the projection to those coordinates
in 
$\binom{\{i,j\}\cup T_a}{2}$.  The point $\pi_a(x)$ belongs to the
cone $\overline{\cTC_{\{i,j\}\cup T_a}}\cap \pi_a(\cT U_{ij})$ and
$\pi_a(\cT U_{ij})\subset \{x_{ij}\neq -\infty\}\subset
\TP\pr^{\binom{|T_a|+2}{2}-1}$.

By construction, we know that $I=\bigcup_{a=1}^mI_a$,
$\mathord{\preceq} =\bigcup_{a=1}^m \mathord{\preceq}_a$ and
$J(ij)=\bigcup_{a=1}^mJ(ij)_a$.  From the $m=1$ case, we know that the
set $I_a$ satisfies~\eqref{eq:diagram} for two injective maps $\ii_a$
and $\rr_a$.  We define the maps $\ii$ and $\rr$ associated to the set
$I$ by restriction, i.e., $\ii(u_{kl})=\ii_a(u_{kl})$ for $kl \in I_a$
and $\rr(u_{kl})=\rr_a(u_{kl})$ for $kl\in I(ij)_a$. Since $\ii_a$ and
$\rr_a$ satisfy \emph{(ii)} and \emph{(iii)}, the result follows.
\end{proof}

\medskip

Our next result shows that from a partial order $\preceq$ on
$[n]\smallsetminus \{i,j\}$ with the cherry property on a tree $T$ and
a vanishing set $J$, we can construct a set that is compatible with
$\preceq$ and $J(ij)$.

\begin{prop}\label{pr:CompatibleI}
  Assume $T$ is arranged as in the right of
  Figure~\ref{fig:caterpillar}, and let $\preceq$ be a partial order
  on $[n]\smallsetminus\{i,j\}$ that has the cherry property on $T$
  with respect to $i$ and $j$. Fix a set $J=J(x)$, where $x\in
  \overline{\cTC_T}\cap \cT U_{ij}$.  Then, there exists a subset
  $I=I(ij,T,J)\subset \binom{[n]}{2}$ of size $2(n-2)$ that is
  compatible with $\preceq$ and $J(ij)$.
\end{prop}
\begin{proof} Let us first consider the case $m=1$. We proceed by
  induction on $n$.  If $n=3$, then $T$ is a caterpillar tree with
  backbone $\{i,j\}$, and we take $I = I(ij)$.  Next, we assume that
  $n>3$, so $T_1$ has more than one leaf. Let $s$ be the maximal leaf
  in $T_1$ with respect to the order~$\preceq$. Remove $s$ from $T$
  and $T_1$, and call $T'$ and $T_1'$ the corresponding subtrees.  Let
  $\preceq'$ be the restriction of the order $\preceq$ to the leaves
  of $T_1'$. This order has the cherry property on $T'$ with respect
  to $i$ and $j$.

  We let $J' = J \cap \binom{[n] \smallsetminus \{s\}}{2}$. As in the
  proof of Proposition~\ref{pr:transBasis}, $J'$ is the vanishing set
  of the point $x'$, obtained by removing from $x$ all those
  coordinates involving the index $s$. The point $x'$ lies in
  $\overline{\cTC_{T'}}\cap \cT U_{ij}'\subset \cTG(2,n-1)$.  By the
  inductive hypothesis, there is a subset $I'=I(ij,T',J')$ of
  $\binom{[n] \smallsetminus \{s\}}{2}$ of size $2(n-3)$ that is
  compatible with $\preceq'$ and $J'(ij)$. Furthermore, the set
  $\A'=I'\cap I(ij)$ satisfies $\A'\supset \{il: l\in T_1'\}$ or
  $\A'\supset \{jl: l\in T_1'\}$.

  We define $I=I(ij,T,J)$ by adding two elements to $I'$. We have a
  priori two possible scenarios.  First, assume that, for all $l\neq
  i,j,s$, we have $il$ or $jl$ belonging to $J$. In this case, we
  define $I=I'\cup \{is,js\}=I(ij)$.  In other case, there exists
  $l\neq i,j,s$ with $il,jl\notin J$. We let $t$ be the maximal leaf
  with that property with respect to the partial order $\preceq'$.  In
  this case, we have two valid options to define $I$, namely,
 \begin{equation*}
   \label{eq:13}
   I := I' \cup \{ st, is \} \text{ or } I := I' \cup  \{ st, js \}.
 \end{equation*}
 To decide which one to choose, we proceed as follows. If $H'\supset
 \{il:l\in T_1'\}$, then we take $I = I' \cup \{ st, is \}$. If
 $H'\supset \{jl:l\in T_1'\}$, then we take $I = I' \cup \{ st, js
 \}$.  If $H'$ contains both sets $\{il:l\in T_1'\}$ and $\{jl:l\in
 T_1'\}$, both options are valid and we must choose one of them.  Note
 that $\A = I\cap I(ij)$ satisfies $\A\supset \{il: l\neq i,j\}$ or
 $\A\supset \{jl: l\neq i,j\}$. By construction, $I$ is compatible
 with $\preceq$ and $J(ij)$.

 We now prove the statement for general $m$. We keep the notation from
 the proof of Proposition~\ref{pr:transBasis}.  For each
 $a=1,\ldots,m$, we let $\preceq_a$ be the restriction of $\preceq$ to
 the leaf set of each $T_a$, and similarly we define $J_a$ as the
 subset of $J$ consisting of pairs in $\{i,j\}\cup T_a$. The order
 $\preceq_a$ has the cherry property on the tree spanned by the leaves
 of $T_a$, $i$, and $j$, with respect to the indices $i$ and
 $j$. Similarly, $J_a$ is the vanishing set of the projection
 $\pi_a(x)$ to the coordinates indexed by pairs in $\{i,j\}\cup
 T_a$. The point $\pi_a(x)$ belongs to $\overline{\cTC_{\{i,j\}\cup
     T_a}}\cap \pi_a(\cT U_{ij})$.  From the $m=1$ case, we construct
 sets $I_a=I_a(ij,T_a,J_a)$ that are compatible with $\preceq_a$ and
 $J_a(ij)$ for each $a$.  We define
\[
I=I(ij,T,J)=\bigcup_{a=1}^m I_a(ij,T_a,J_a).
\]
By construction, $I$ is compatible with $\preceq$ and $J(ij)$. This
concludes our proof.
\end{proof}
\begin{rem}\label{rm:ManyI}
  From the proof of Proposition~\ref{pr:CompatibleI}, we see that the
  set $I$ is not uniquely determined by $\preceq$, $J$ and $i,j$. This
  is so because for every $a=1,\ldots,m$, one of the following is
  true: either $\A\supset \{il: l\in T_a\}$ or $\A\supset \{jl: l\in
  T_a\}$, and we can choose freely between any of these two options.
  In particular, there is one choice of the set $I$ for which
  $H\supset \{il: l\neq i,j\}$ and one choice for which $H\supset
  \{jl: l \neq i,j\}$.  Allowing the flexibility to choose among two
  options for each $H_a=I_a\cap I(ij)$ whenever possible preserves the
  symmetry of our objects with respect to $i$ and $j$, and will
  simplify our proofs.

  Observe that $H\supset \{il,jl: l\in T_a\}$ if and only if either
  $T_a$ has exactly one leaf, or, for every $k\in T_a$ that is not
  maximal with respect to $\preceq$, we have $il\in J_a$ or $jl\in
  J_a$. The reason for imposing this condition comes from the
  Pl\"ucker relations.
\end{rem}

Notice that the properties of Proposition~\ref{pr:transBasis} ensure
that $\rr(u_{kl})=u_{kl}$ for all $kl\in \A$. We use this observation
in the following two examples and in the remainder of this Section.
\begin{example}\label{ex:caterpillar} 
  Let $T$ be a caterpillar tree with backbone spanned by $i$ and $j$.
  When setting $\preceq$ as $\{(k,k): k\neq i,j\}$, $I=I(ij)$,
  $\ii=\id$ and $\rr=\inc$, we recover the results from
  Section~\ref{sec:caterpillar-case}.
\end{example}
\begin{example}\label{ex:cherry}
  Let $i=1,j=2$ and $T$ be a trivalent tree on four leaves with cherry
  $\{1,2\}$. Then, we take $\preceq$ to be $\{(3,3), (3,4), (4,4)\}$.
  If $13,23\notin J(12)$, we let $I=\{13,23,34,14\}$, and $\rr$ be the
  map defined by $\rr(u_{24})=u_{13}^{-1}(u_{34}+u_{14}u_{23})$ and
  $\rr(u_{kl})=u_{kl}$ if $kl=13,23$ or $14$.

  On the contrary, if $J(12)$ contains either $13$ or $23$, we define
  $I=\{13,23,14,24\} = I(12)$ and let $\rr$ be the inclusion map.
\end{example}

\medskip In what follows, we outline the construction of a section to
$\trop$ over a covering of $\cT U_{ij}$.  Given a tree $T$ on $n$
leaves and a vanishing set $J=J(x)$ associated to a point $x\in
\overline{\cTC_T}\cap \cT U_{ij}$, we choose a partial order with the
cherry property on $T$ as in Lemma~\ref{lm:order} and a compatible set
$I = I(ij,T,J)$ as in Proposition ~\ref{pr:CompatibleI}.  We define
the map
\begin{equation}
  \sigma_{T,I,J}^{(ij)}\colon \CijTJ
  \longhookrightarrow  \cT U_{ij}
  \longrightarrow (\Spec K[u_{kl}: kl\in I])^{\an} \quad \text{ by } \quad \sigma_{T,I,J}^{(ij)}(x) =  {\sigma}^{(ij)}_I(x)\text{.}\label{eq:section}
\end{equation}
The next result show that all points in the image of
$\sigma^{(ij)}_{T,I,J}$ are multiplicative seminorms on the
localization $K[u_{kl}: kl\in I]_{S_{H\smallsetminus J}}$:
\begin{lemma}\label{lm:ExtensionFromRij}
  Let $i,j$, $T$, $\preceq$, and $J$ be as above. Assume that $I$ is
  compatible with $\preceq$ and $J(ij)$. Then, given any point $x\in
  \CijTJ$, the seminorm $\sigma^{(ij)}_{T,I,J}(x)$ as defined
  in~\eqref{eq:section} extends uniquely to a multiplicative seminorm
  on $K[ u_{kl}: kl \in I][u_{kl}^{-1}: kl\in \A\smallsetminus J]$,
  that is, to an element of $(\Spec K[ u_{kl}: kl \in I][u_{kl}^{-1}:
  kl\in \A\smallsetminus J])^{\an}\subset U_{ij}^{\an}$.
\end{lemma}
\begin{proof} Since
  $\sigma^{(ij)}_{T,I,J}(x)(u_{kl})=\exp(x_{kl}-x_{ij})\neq 0$ for all
  $kl\in \A\smallsetminus J$, we can uniquely extend this
  multiplicative seminorm to the Laurent polynomial ring $K[u_{kl}:
  kl\in I][u_{kl}^{-1}: kl\in \A\smallsetminus J]$.
\end{proof}

The following lemma guarantees that the ideal $\faJij$ maps to $0$
under $\sigma^{(ij)}_{T,I,J}(x)$.

\begin{lemma}~\label{lm:zero} Let $i,j$, $T$, $\preceq$, and $J$ be as
  above. Then, $\sigma^{(ij)}_{T,I,J}(x)(\faJij)=0$ for all $x\in
  \CijTJ$.
\end{lemma}
\begin{proof} 
  For simplicity, we write $\sigma(x)=\sigma^{(ij)}_{T,I,J}(x)$. By
  symmetry between $i$ and $j$, it suffices to show that
  $\sigma(x)(u_{jl})=0$ for all $jl\in J(ij)$. If $jl\in J(ij)\cap I$,
  then $\sigma(x)(u_{jl})=\exp(x_{jl}-x_{ij})=0$.

  On the contrary, assume $jl\in J(ij)\smallsetminus I$. Since $I$ is
  compatible with $\preceq$ and $J(ij)$, and $jl\notin I$, this
  implies that $il\in I$, and there exists $t\prec l$ maximal such
  that $it, jt\notin J$, forcing $lt,it\in I$ and $\rr(u_{jl})=
  u_{it}^{-1}(-u_{lt}+u_{il}\,\rr(u_{jt}))$.
  Lemma~\ref{lm:ConditionsIAndJ} ensures that $il\in J(ij)$, so
  $\sigma(x)(u_{il})=0$. The four-point condition on the quartet
  $\{i,j,t,l\}$ implies that $x_{lt}+x_{ij} =-\infty$, so
  $x_{lt}=-\infty$ and thus $\sigma(x)(u_{lt})=0$.  The definition of
  $\rr(u_{jl})$, together with the non-Archimedean triangle inequality
  for $\sigma(x)$, yield
\[
\sigma(x)(u_{jl})
=\sigma(x)(u_{it}^{-1})\sigma(x)(-u_{lt}+u_{il}\,\rr(u_{jt})) \leq
\sigma(x)(u_{it}^{-1})\max\{\sigma(x)(u_{lt}),
\sigma(x)(u_{il})\sigma(x)(u_{jt})\}=0. \qedhere
\]
\end{proof}

By Proposition~\ref{pr:transBasis}~\emph{(ii)}, the map $\rr$ induces
an isomorphism of localizations
\begin{equation*}
  \xymatrix@1{
    \rr\colon  R(ij)_{S_{\A \smallsetminus J}}
    \ar[rr]^-{\sim} &&  K[u_{kl}: kl\in I]_{S_{\A \smallsetminus J}},}
\end{equation*}
where $S_{\A\smallsetminus J}$ is as in Remark~\ref{rm:MultClosedSet}.
Using this map and Lemma~\ref{lm:ExtensionFromRij}, we define
$\sigma^{(ij)}_{T,I,J}(x)$ on $R(ij)$ by restriction.

\medskip

We now state the main result in this Section, keeping the previous
notation.
\begin{thm}\label{thm:section} 
  Let $i,j$, $T$, $\preceq$, and $J$ be as above, and let
  $I=I(ij,T,J)$ be a set of size $2(n-2)$ that is compatible with
  $\preceq$ and $J(ij)$.  Then, $\sigma^{(ij)}_{T,I,J}$ is a section
  to the tropicalization map over the cone $\CijTJ$.
\end{thm}

\begin{proof}
  We divide the proof in two parts. First, we discuss the case when
  $m=1$, and second, we show how the general result can be deduced
  from this special case.  Using the diagram~\eqref{eq:diagram} and
  the Pl\"ucker relations, we view all functions
  $u_{kl}:=u_{ik}u_{jl}-u_{il}u_{jk}$ in the ring $R(ij)$. By
  Lemma~\ref{lm:ExtensionFromRij}, we can define
  $\sigma^{(ij)}_{T,I,J}(x)$ on this ring. For simplicity, write
  $\sigma:= \sigma^{(ij)}_{T,I,J}$.

  In what follows, we consider polynomials in all variables $\{u_{kl}:
  kl\neq ij\}$. To simplify notation, rather than writing
  $\rr(u_{kl})$, we underline the variable $u_{kl}$ whenever $kl \in
  I(ij)\smallsetminus I$, indicating that we should consider its image
  under $\rr$ in $K[u_{kl}: kl\in I][u_{kl}^{-1}: kl \in
  \A\smallsetminus J]$.

  \smallskip Assume $m=1$. We proceed by induction on $n$ and use
  Lemma~\ref{lm:sectionProperty} to check the section property. If
  $n=3$, then $T$ is the caterpillar tree, and so $I=I(ij)$. The
  result follows by Proposition~\ref{pr:SectionCaterpillar}. Next,
  suppose that $n>3$, and let $s$ be the maximal element in $T_1$. Let
  $T'$ and $T_1'$ be the trees obtained by removing the leaf $s$ from
  $T$ and $T_1$, respectively. We keep the notation from the proof of
  Proposition~\ref{pr:transBasis} and define $I' = I \cap
  \binom{[n]\smallsetminus \{s\}}{2}$, $J'=J \cap
  \binom{[n]\smallsetminus \{s\}}{2}$, $R(ij)'= K[u_{ik},u_{jk}: k\neq
  i,j,s]$. We also let $\preceq'$ be the order on $T'_1$ induced by
  $\preceq$.  Since $I'$ is compatible with $\preceq'$ and $J'(ij)$,
  we find that the restriction of $\sigma(x)$ to $R(ij)'$ agrees with
  $\sigma^{(ij)}_{T',I',J'}(x')$, where $x'$ is the projection of $x$
  to those coordinates not involving the index $s$.  It suffices to
  prove that $\sigma(x)(u_{sl})=\exp(x_{sl}-x_{ij})$ for all $l\neq
  s$. All other identities will follow by the inductive
  hypothesis. Using the symmetry between $i$ and $j$, we suppose that
  $\{il: l\neq i,j\}\subset \A$ (see Remark~\ref{rm:ManyI}). We
  analyze two cases, depending on the nature of $I$.

\smallskip
\noindent
\textbf{Case 1:} Assume that for all $l\neq i,j,s$ we have $il$ or
$jl\in J$. In this situation, we have $I=I(ij)$. In particular,
$\sigma(x)(u_{is})=\exp(x_{is}-x_{ij})$ and
$\sigma(x)(u_{js})=\exp(x_{js}-x_{ij})$.

Fix an index $l\neq i,j,s$. By the four-point condition on the quartet
$\{i,j,s,l\}$, we have
$x_{sl}+x_{ij}=x_{il}+x_{js}=x_{jl}+x_{is}=-\infty$, so $u_{sl} =
u_{is}u_{jl}-u_{il}u_{js} \in \faJij$. Lemma~\ref{lm:zero} implies
$\sigma(x)(u_{sl})=0=\exp(x_{sl}-x_{ij})$, as we wanted to show.

\smallskip\noindent \textbf{Case 2:} Suppose that there exists $t\neq
i,j,s$ satisfying $it,jt\notin J$. Choose $t$ maximal with this
property. By Definition~\ref{def:CorrectI}, we know that $st\in I$, so
$\sigma(x)(u_{st})=\exp(x_{st}-x_{ij})$. Similarly, since $is\in I$ we
have $\sigma(x)(u_{is})=\exp(x_{is}-x_{ij})$. Using the identity
\[
\underline{u_{js}}=
u_{it}^{-1}u_{is}\,\underline{u_{jt}}-u_{it}^{-1}u_{st},
\] 
we compute $\sigma(x)(u_{js})$ by expanding $\underline{u_{jt}}$ as a
polynomial in the $I$-coordinates, and taking the maximum over the
values on all monomials occurring in $\underline{u_{js}}$. By
Proposition~\ref{pr:transBasis}~\emph{(iii)}, the monomial
$u_{it}^{-1} u_{st}$ does not cancel with any term in
$u_{it}^{-1}u_{is}\,\underline{u_{jt}}$. Since $\sigma(x)$ is
multiplicative and it satisfies $\sigma(x)(u_{kl}) = \exp(x_{kl} -
x_{ij})$ for all $kl \in I$ as well as $\sigma(x) (u_{jt}) =
\exp(x_{jt} - x_{ij})$ by the inductive hypothesis, we deduce
  \begin{align*}
    \sigma(x)(u_{js}) & =
    \sigma(x)(u_{it}^{-1}u_{is}\,\underline{u_{jt}}-u_{it}^{-1}u_{st})
    =
    \max\{\sigma(x)(u_{it}^{-1}u_{is}\,\underline{u_{jt}}), \sigma(x)(u_{it}^{-1}u_{st})\}\\
    & =  \max\{\exp(-x_{it}+ x_{is}) \sigma(x)(\underline{u_{jt}}), \, \exp(-x_{it}+ x_{st})\}\\
    & = \max\{\exp(-x_{it} + x_{is}+ x_{jt} - x_{ij}), \, \exp(-x_{it}
    + x_{st})\}.
  \end{align*}
  The four-point condition on the quartet $\{i,j,s,t\}$ yields
  $\sigma(x)(u_{js})= \exp(x_{js}-x_{ij})$.

  Finally, we prove the claim for $u_{sl}$, where $l\neq
  i,j,s,t$. First, assume that $t\prec l$. Then, the maximality of $t$
  and Lemma~\ref{lm:ConditionsIAndJ} imply that both $il$ and $jl$
  belong to $J$, hence $u_{sl}\in \faJij$ and $\sigma(x)(u_{sl})=0$ by
  Lemma~\ref{lm:zero}. The four-point condition on $\{i,j,s,l\}$
  yields $x_{sl}=-\infty$, and the result holds.

  On the contrary, suppose that $l\prec t$. By the Pl\"ucker
  relations, we write $u_{sl}$ as
\begin{equation*}
  u_{sl}=u_{it}^{-1}u_{is}{u_{tl}}+u_{it}^{-1}u_{il}u_{st}.
\end{equation*}
In order to evaluate $\sigma(x)$ on $u_{sl}$, we expand $u_{tl}$ in
$I$-coordinates and then take the maximum over the values of
$\sigma(x)$ on all monomials.  By Proposition
~\ref{pr:transBasis}~\emph{(iii)}, the monomial $u_{it}^{-1}
u_{il}u_{st}$ does not cancel with any term in the expansion of
$u_{it}^{-1}u_{is} {u_{tl}}$. Note that $\sigma(x) (u_{tl}) =
\exp(x_{tl} - x_{ij})$ by the inductive hypothesis.  Since $\sigma(x)$
is multiplicative and it satisfies $\sigma(x)(u_{pq}) = \exp(x_{pq} -
x_{ij})$ for all $pq \in I$, this implies
  \begin{align*}
    \sigma(x)(u_{sl}) & =
    \max\{\sigma(x)(u_{it}^{-1}u_{is}\,{u_{tl}}), \sigma(x)(u_{it}^{-1}u_{il}u_{st})\}\\
    & =  \max\{\exp(-x_{it}+ x_{is}) \sigma(x)({u_{tl}}), \, \exp(-x_{it}+x_{il} + x_{st} -x_{ij})\}\\
    & = \max\{\exp(-x_{it} + x_{is} +x_{tl} - x_{ij}\}, \, \exp( -
    x_{it} + x_{il} + x_{st}-x_{ij})\}.
  \end{align*}

  Since $l\prec t\prec s$, and $\preceq$ has the cherry property on
  $T$ with respect to $i$ and $j$, we know by
  Definition~\ref{def:orderWithCherryProperty}~{(iii)} that either
  $\{l,t\}$ or $\{t,s\}$ are cherries of the quartet $\{i,l,t,s\}$.
  Therefore, one of the following identities hold:
\begin{equation*}
  x_{is}+x_{tl}\leq x_{it}+x_{sl}=x_{il}+x_{st} \quad \text{ or }\quad
  x_{il}+x_{st}\leq x_{it}+x_{sl}=x_{is}+x_{tl}.\label{eq:9}
\end{equation*}
In both cases it is easy to check that $\sigma(x)(u_{sl})=
\exp(x_{sl}-x_{ij})$.  \smallskip

\medskip

Next, we discuss the case when $m>1$. As before, we show that
$\sigma(x)(u_{kl})= \exp(x_{kl}-x_{ij})$ for all $kl\in
\binom{[n]}{2}\smallsetminus\{ij\}$.  We follow the notation of
Figure~\ref{fig:caterpillar}. Proposition~\ref{pr:transBasis}~\emph{(iii)}
and the proof for the $m=1$ case guarantee that the result holds when
$kl\in I(ij)$ or when we pick two indices $k,l$ in the same subtree
$T_a$, for $a=1,\ldots, m$.

It remains to check the case when $k$ and $l$ belong to two different
subtrees, say $k\in T_a$ and $l\in T_b$ with $a \neq b$. The quartet
with leaves $\{i,j,k,l\}$ contains the cherry $\{i,k\}$ or the cherry
$\{i,l\}$. Hence, by the four-point condition, one of the following is
true:
\begin{equation}
  x_{ik}+x_{jl}\leq x_{ij}+x_{kl}=x_{il}+x_{jk} \quad \mbox{ or }  \quad x_{il} + x_{jk} \leq x_{ij} + x_{kl} = x_{ik} + x_{jl}.\label{eq:25}
\end{equation}

We distinguish four cases.  If $ik,jk\in \A$, then $u_{kl} =
u_{ik}\,\underline{u_{jl}}-\underline{u_{il}}u_{jk}$.  The
compatibility of $I$ with $\preceq$ and $J(ij)$, and
Proposition~\ref{pr:transBasis}~\emph{(iii)} ensure that none of the
monomials in $u_{ik}\,\underline{u_{jl}}$ cancel out with any monomial
in $\underline{u_{il}}u_{jk}$.  Using~\eqref{eq:25} we conclude that
\[
\sigma(x)(u_{kl}) = \max\{ \exp(x_{ik} + x_{jl} - 2 x_{ij}), \,
\exp(x_{il} +x_{jk} - 2 x_{ij})\} = \exp(x_{kl}-x_{ij}).
\]
An analogous argument holds when $il, jl\in \A$.

It remains to study the situation where both pairs $ik,jk$ and $il,jl$
have one member outside $\A$. In particular, by
Definition~\ref{def:CorrectI}, we know that $T_a$ and $T_b$ have more
than one leaf.  By the symmetry between $i$ and $j$, we are left with
two possibilities: either $ik, jl\notin \A$ or $ik,il\notin \A$.  By
the compatibility of $I$ with $\preceq$ and $J(ij)$, we know that
there exist maximal elements $k'\prec k$ and $l'\prec l$, satisfying
that $kk',ll'\in I$, and $ik',jk', il', jl'\notin J$.

Let us discuss the first scenario, where $ik,jl\notin \A$.  Notice
that $H\supset \{jq: q\in T_a\}\cup \{iq: q\in T_b\}$ by
Remark~\ref{rm:ManyI}.  We write $u_{kl} =
\underline{u_{ik}}\underline{u_{jl}} - u_{il} u_{jk}$.
Proposition~\ref{pr:transBasis}~\emph{(iii)} shows that the expansion
of $\underline{u_{ik}}\underline{u_{jl}}$ in $I$-coordinates does not
contain the monomial $u_{il}u_{jk}$.  Hence, by \eqref{eq:25} we
conclude
\begin{align*}
  \sigma(x) ( u_{kl}) & = \max\{ \sigma(x)
  (\underline{u_{ik}}\underline{u_{jl}}),
  \sigma(x)(u_{il}u_{jk})\}\\
  &= \max\{ \exp(x_{ik}+ x_{jl} - 2 x_{ij}), \exp(x_{il} + x_{jk} - 2
  x_{ij}) \} = \exp( x_{kl} - x_{ij}).
\end{align*}

Finally, let us analyze the case when $ik, il\notin \A$. In this
situation, we know that $H\supset \{jq: q\in T_a\cup T_b\}$, and we
have $u_{kl} = \underline{u_{ik}} u_{jl} - \underline{u_{il}} u_{jk}$,
where
\[
\underline{u_{ik}}=u_{jk'}^{-1}(-u_{kk'}+\underline{u_{ik'}}u_{jk})\qquad
\text{and} \qquad
\underline{u_{il}}=u_{jl'}^{-1}(-u_{ll'}+\underline{u_{il'}}u_{jl}).
\]
Plugging these expressions into the Pl\"ucker expression for $u_{kl}$,
we see that
\begin{equation*}
  u_{kl}=\underline{u_{ik}}u_{jl}-\underline{u_{il}}u_{jk}=u_{jk}u_{jl}(u_{jk'}^{-1}\,\underline{u_{ik'}}-u_{jl'}^{-1}\,\underline{u_{il'}})-u_{jk'}^{-1}u_{kk'}u_{jl}+u_{jl'}^{-1}u_{ll'}u_{jk}.\label{eq:41}
\end{equation*}
The conditions that $k'\prec k$ and $l'\prec l$ together with
Proposition~\ref{pr:transBasis}~\emph{(iii)} ensure that the monomials
$u_{jl}u_{kk'}u_{jk'}^{-1}$ and $u_{jl'}^{-1}u_{ll'}u_{jk}$ cannot be
present in the polynomial
$u_{jk}u_{jl}(u_{jk'}^{-1}\,\underline{u_{ik'}}-u_{jl'}^{-1}\,\underline{u_{il'}})$. Hence,
\begin{equation*}\label{eq:16}
  \sigma(x)(u_{kl})=\max\{\sigma(x)(u_{jk}u_{jl}(u_{jk'}^{-1}\,\underline{u_{ik'}}-u_{jl'}^{-1}\,\underline{u_{il'}}),
  \sigma(x)(u_{jl}u_{kk'}u_{jk'}^{-1}), \sigma(x)(u_{jl'}^{-1}u_{ll'}u_{jk})\}.
\end{equation*}
We now study these three terms, starting from the last two.  Since
$k,k'\in T_a$ and $l,l'\in T_b$, we deduce from the four-point
conditions on the quartets $\{j,l,k,k'\}$ and $\{j,k,l,l'\}$ that
\begin{equation*}
  \sigma(x)(u_{jl}u_{kk'}u_{jk'}^{-1})\leq \exp(x_{kl}-x_{ij})\quad
  \text{and} \quad \sigma(x)(u_{jl'}^{-1}u_{ll'}u_{jk})\leq \exp(x_{kl}-x_{ij}).\label{eq:17}
\end{equation*}

We claim that the value of $\sigma(x)$ at the polynomial
$u_{jk}u_{jl}(u_{jk'}^{-1}\,\underline{u_{ik'}}-u_{jl'}^{-1}\,\underline{u_{il'}})$
equals $\exp(x_{kl}-x_{ij})$.  As before, by
Proposition~\ref{pr:transBasis}~\emph{(iii)}, we know that both
expressions $u_{jk'}^{-1}\,\underline{u_{ik'}}$ and
$u_{jl'}^{-1}\,\underline{u_{il'}}$ involve disjoint sets of
monomials, so no cancellations can occur.  The four-point conditions
on the quartets $\{i,j,k',k\}$, $\{i,j,l',l\}$ imply $x_{jk} + x_{ik'}
= x_{ik} + x_{j k'}$ and $ x_{jl} + x_{i l'}= x_{il} + x_{jl'}$, so
\begin{align*}
  \sigma(x)(u_{jk}u_{jl}u_{jk'}^{-1}\,\underline{u_{ik'}}) &=
  \exp(x_{jk} + x_{jl} - x_{j k'} + x_{i k'} - 2 x_{ij})= \exp(x_{ik}
  + x_{jl} -2 x_{ij}), \\
  \sigma(x)(u_{jk}u_{jl}u_{jl'}^{-1}\,\underline{u_{il'}}) &= \exp(
  x_{jk} + x_{jl} - x_{j l'}+ x_{i l'} - 2 x_{ij}) = \exp(x_{jk} +
  x_{il} - 2 x_{ij} ).
\end{align*}
Hence, $\sigma(x)(u_{kl})= \max\{\exp(x_{ik} + x_{jl} -2
x_{ij}),\exp(x_{jk} + x_{il} - 2 x_{ij} )\} = \exp(x_{kl}-x_{ij})$ by
~\eqref{eq:25}. This concludes our proof.
\end{proof}

\subsection{Uniqueness and continuity}
\label{sec:uniqueness}
In the previous two Sections, we constructed explicit maps
$\sigma^{(ij)}_{T,I,J}$ on a covering of $\cTG(2,n)$ by cones labeled
$\CijTJ$, and we showed that these maps are a section to $\trop$ on
each domain.  Our next task is to glue these sections together and,
thus, define a map over $\cTG(2,n)$. In order to do so, we show that
these sections are the unique ones satisfying a maximality
property. In Section~\ref{sec:fibers-trop} we give a more conceptual
proof of this fact via the existence of a unique Shilov boundary point
in the fibers of the tropicalization map.

\begin{lemma}\label{lm:MaxCondition}
  Let $x\in \cTG(2,n)$ and $i,j$ be such that $x\in \cT U_{ij}$. Let
  $T$ and $J$ be such that $x\in \CijTJ$. Then, for any index set $I$
  compatible with $\preceq$ and $J(ij)$, and for any $\rho \in
  \trop^{-1}(x)$, we have that
  \[
  \rho(f) \leq \sigma^{(ij)}_{T,I,J}(x )(f) \qquad \text{ for all }
  f\in R(ij).\]
\end{lemma}
\begin{proof}

  Let $\A= I\cap I(i,j)$ be as in Sections~\ref{sec:caterpillar-case}
  and~\ref{sec:exist-gener-tree}. Since $\rho$ and
  $\sigma^{(ij)}_{T,I,J}(x)$ belong to the fiber of $\trop$ over $x$,
  Lemma~\ref{lm:sectionProperty} ensures that
  $\rho(u_{kl})=\exp(x_{kl}-x_{ij})=\sigma^{(ij)}_{T,I,J}(x)(u_{kl})$
  for all $k,l$.

  Notice that since $\rho(u_{kl})\neq 0$ for any $kl\in
  \A\smallsetminus J$, we can uniquely extend the seminorm $\rho$ from
  $R(ij)$ to the Laurent polynomial ring $K[u_{kl}: kl\in
  I][u_{kl}^{-1}: kl\in \A\smallsetminus J]$ via the map $\rr$
  in~\eqref{eq:diagram}, as we did in
  Lemma~\ref{lm:ExtensionFromRij}. For simplicity, we also call this
  extension by $\rho$. Given any polynomial $f\in R(ij)$, we write it
  as a Laurent polynomial $f= \sum_{\alpha} c_{\alpha} u^{\alpha}$,
  where $\alpha\in \NN_{0}^{I\smallsetminus (\A \smallsetminus
    J)}\times \ZZ^{\A\smallsetminus J}$.  The definition of the map
  $\sigma^{(ij)}_{T,I,J}$ and the non-Archimedean triangle inequality
  for $\rho$ ensure that
\[
\rho(f) \leq \max\limits_{\alpha}\{\rho(c_{\alpha}u^{\alpha})\}
=\max_{\alpha}\{|c_{\alpha}| \sigma^{(ij)}_{T,I,J}(x)(u^{\alpha})\} =
\sigma^{(ij)}_{T,I,J}(x)(f)\text{.} \qedhere
\]
\end{proof}

The intrinsic characterization of the maps $\sigma^{(ij)}_{T,I,J}$ in
Lemma~\ref{lm:MaxCondition} has two important consequences. First, the
function $\sigma^{(ij)}_{T,I,J}$ is independent of the choice of the
index set $I$ as long as it is compatible with $\preceq$ and $J(ij)$.
Likewise, the functions $\sigma^{(ij)}_{T,I,J}$ and
$\sigma^{(ij)}_{T',I',J'}$ agree on the overlaps of the sets $\CijTJ$
and $\CijTJprime$, corresponding to two trees and two vanishing sets.
Therefore, these functions glue together to yield a unique map
$\sigma^{(ij)}$ on $\cT U_{ij}$ for every choice of a pair $ij$.

\smallskip

Our next result ensures that the collection $\{\sigma^{(ij)}\}_{ij}$
glues to a map $\sigma\colon \cTG(2,n)\to \Gr(2,n)^{\an}$.

\begin{prop}~\label{pr:gluing} The map $\sigma^{(ij)}$ is independent
  of our starting choice of indices $i$ and $j$.
\end{prop}
\begin{proof}
  The result follows from Lemma~\ref{lm:MaxCondition}, as we now
  explain. Fix two pairs of indices $ij, pq\in \binom{[n]}{2}$, and
  pick any point $x\in \cT U_{ij} \cap \cT U_{pq}$, so $x_{ij}, x_{pq}
  \neq -\infty$. Let $J=J(x)$ and fix a tree $T$ with $x\in
  \overline{\cTC_T}$.  The multiplicative seminorm $\sigma^{(ij)}(x)$
  is defined over $R(ij)$, whereas $\sigma^{(pq)}(x)$ is defined over
  $R(pq)$.  To avoid confusions, we denote by $u_{kl}$ the functions
  in $R(ij)$ and by $v_{kl}$ the functions in $R(pq)$.  As usual, we
  write $u_{kl}=u_{ik}u_{jl}-u_{ik}u_{jl}$ in $R(ij)$ whenever
  $kl\notin I(ij)$, and $v_{kl}=v_{pk}v_{ql}-v_{pl}v_{qk}$ in $R(pq)$
  whenever $kl\notin I(pq)$.

  Since $x\in \cT U_{ij} \cap \cT U_{pq}$, we know that
  $\sigma^{(ij)}(x)(u_{pq})=\exp(x_{pq}-x_{ij})\neq 0$ and
  $\sigma^{(pq)}(x)(v_{ij})=\exp(x_{ij}-x_{pq})\neq 0$.  Thus, the
  multiplicative seminorms $\sigma^{(ij)}(x)$ and $\sigma^{(pq)}(x)$
  extend uniquely to the localizations $R(ij)_{S_{\{pq\}}}$ and
  $R(pq)_{S_{\{ij\}}}$ defined as in
  Remark~\ref{rm:MultClosedSet}. These localizations are related by
  the natural isomorphism
  \[r\colon R(pq)_{S_{\{ij\}}} \longrightarrow R(ij)_{S_{\{pq\}}},
  \qquad v_{pk}\longmapsto u_{pk}/ u_{pq} \quad \text{ and } \quad
  v_{qk}\longmapsto u_{qk}/u_{pq}\text{.}\] Both seminorms
  $\sigma^{(ij)}(x)$ and $\sigma^{(pq)}(x)$ are maximal with respect
  to evaluation on these localizations.  By
  Lemma~\ref{lm:MaxCondition}, $\sigma^{(ij)}(x)\circ r=$
  $\sigma^{(pq)}(x)$ on $R(pq)_{S_{\{ij\}}}$, as we wanted to show.
\end{proof}

Theorem~\ref{thm:section} shows that $\sigma$ is a section to $\trop$.
We end by proving the continuity of $\sigma$.
\begin{thm}\label{thm:continuity}
  The map $\sigma\colon \cTG(2,n)\to \Gr(2,n)^{\an}$ is continuous.
\end{thm}
\begin{proof} Since the sets $U_{ij}$ for $ ij\in \binom{[n]}{2}$ are
  an open cover of $\Gr(2,n)$, it suffices to prove that each
  restriction $\sigma^{(ij)}\colon \cT U_{ij}\to U_{ij}^{\an}$ is
  continuous.  Recall that the topology on $U_{ij}^{\an}$ is the
  topology of pointwise convergence on functions in the coordinate
  ring $R(ij)$. Given a sequence $(x^{(m)})_{m\in \NN}\subset \cT
  U_{ij}$ converging to $x\in \cT U_{ij}$ and any element $f\in
  R(ij)$, we wish to show that the sequence $ \sigma(x^{(m)})(f)$ has
  a limit and, moreover, that its limit is
  $\sigma(x)(f)$. Equivalently,
  \begin{equation}
    \sigma(x)(f)\leq \liminf_m \,\sigma(x^{(m)})(f)\leq \limsup_m\,
    \sigma(x^{(m)})(f)\leq \sigma(x)(f)\qquad \text{ for all } f\in R(ij).
\label{eq:31}
\end{equation}
It suffices to prove that any subsequence of $(x^{(m)})_m$ has a
sub-subsequence satisfying~\eqref{eq:31}.

In order to establish the desired inequalities, we set some notation.
From the topology of $\cTG(2,n)$, we know that $\lim\limits_{m\to
  \infty} x^{(m)}_{kl}-x^{(m)}_{ij}=x_{kl}-x_{ij}$ in $\overline{\RR}$
for all $kl\in \binom{[n]}{2}$. In addition, by the pigeonhole
principle, any subsequence of $(x^{(m)})_m$ has a subsequence
$(x^{(m_k)})_{k}$ satisfying
$J(x^{(m_k)})=J(x^{(m_l)})=:\tilde{J}\subset J$ for all $k,l\in \NN$,
as well as $(x^{(m_{k})})_k\subset \CijTJtilde$ and $x\in \CijTJ$ for
the same tree $T$.  To simplify notation, we assume the original
sequence has this property.  The result follows from Lemmas
~\ref{lm:RightIneq} and~\ref{lm:LeftIneq} below.\end{proof}

\begin{lemma}\label{lm:RightIneq}
  Fix a pair of indices $i,j$ in $[n]$, an arbitrary tree type $T$ on
  $n$ leaves and a point $x\in \CijTJ$.  Suppose that there is a
  vanishing set $\tilde{J} \subset J$ and a sequence $(x^{(m)})_m $ in
  the cone $ \CijTJtilde$ converging to $x$. Then,
\[
\limsup\limits_m\, \sigma^{(ij)}(x^{(m)})(f)\leq
\sigma^{(ij)}(x)(f)\qquad \text{ for all }f\in R(ij).
\]
\end{lemma}

\begin{proof} 
  We let $\tilde{I}=I(ij,T,\tilde{J})$ and $I=I(ij,T,J)$ be two
  compatible sets, and set $\tilde{\A}=\tilde{I}\cap I(ij)$, $\A=I\cap
  I(ij)$. We construct $\sigma(x^{(m)})$ and $\sigma(x)$ using
  formula~\eqref{eq:section}.  Since $\sigma$ is independent of all
  choices of index sets, we may assume that both $I$ and $\tilde{I}$
  are compatible with the same order $\preceq$ on $[n]\smallsetminus
  \{i,j\}$ having the cherry property on $T$, and that $\tilde{\A}$
  satisfies $\tilde{\A}\subset \A$ by Remark~\ref{rm:ManyI}.  Since
  $\tilde{J}\subset J$, we extend the multiplicative seminorms
  $\sigma(x^{(m)})$ to $R(ij)_{S_{\A\smallsetminus J}}$ using
  Lemma~\ref{lm:ExtensionFromRij}.
  The topology of $\cTG(2,n)$ and the continuity of exponentiation 
on $\overline{\RR}$ imply that
 \begin{equation}
   \lim_{m\to \infty}\sigma(x^{(m)})(u^{\alpha}) =\sigma(x)(u^{\alpha})\qquad
   \text{ for all } \alpha\in 
   \NN_0^{I\smallsetminus (H\smallsetminus J)}\times \ZZ^{H\smallsetminus J}  .\label{eq:34b}
\end{equation}

Given any $f\in R(ij)$, we write $\rr(f)=\sum_{\alpha}
c_{\alpha}u^{\alpha}\in K[u_{kl}: kl\in I][u_{kl}^{-1}: kl\in
\A\smallsetminus {J}]$. Then, by definition, $\sigma(x)(f) =
\max\limits_{\alpha}\{\sigma(x)(c_{\alpha}u^{\alpha})\}$.
Combining~\eqref{eq:section},~\eqref{eq:34b} and the non-Archimedean
triangle inequality for each $\sigma(x^{(m)})$, we conclude that
  \begin{equation*}
  \limsup_{m}\sigma(x^{(m)})(f)\leq
  \limsup_m\max\limits_{\alpha}\{\sigma({x^{(m)}})(c_{\alpha}u^{\alpha})\}\!=\!
  \max\limits_{\alpha}\{\limsup_m\sigma(x^{(m)})(c_{\alpha}u^{\alpha})\}=
\sigma(x)(f).\!\! \qedhere
\end{equation*}
\end{proof}
\begin{lemma}\label{lm:LeftIneq}
  With the same hypothesis as in Lemma~\ref{lm:RightIneq}, we have
\[\sigma(x)(f)\leq \liminf\limits_m
\,\sigma(x^{(m)})(f) \qquad \text{for all }f\in R(ij).
\]
\end{lemma}

\begin{proof} 
  We let $\tilde{\rr}$ be the map from~\eqref{eq:diagram}
  corresponding to any given compatible set
  $\tilde{I}=I(ij,T,\tilde{J})$.  We extend the multiplicative
  seminorm $\sigma(x^{(m)})$ to $K[u_{kl}: kl\in
  \tilde{I}][u_{kl}^{-1}: kl\in \tilde{\A}\smallsetminus \tilde{J}]$
  using Lemma~\ref{lm:ExtensionFromRij}. The main obstruction to
  mimicking the proof of Lemma~\ref{lm:RightIneq} lies in the
  possibility of having a monomial $u^{\alpha}$ in the expression of
  $\tilde{\rr}(f)$ with a negative exponent corresponding to an
  unknown $u_{kl}$ with $kl\in \tilde{\A}\cap J\smallsetminus
  \tilde{J}$. If so, we know that $\lim\limits_{m\to +\infty}
  \sigma(x^{(m)})(u^{\alpha})=+\infty$.  However,
  Lemma~\ref{lm:RightIneq} implies that the sequence
  $\sigma(x^{(m)})(f)$ is bounded above by $\sigma(x)(f)$. This fact
  ensures that these bad monomials will not realize the maximum
  defining $\sigma(x^{(m)})(f)$.

  Fix a polynomial $f\in R(ij)$.  If $\sigma(x)(f)=0$, there is
  nothing to prove.  Now, suppose $\sigma(x)(f)\neq 0$.  In
  particular, this implies that $f\notin \faJij$ by
  Lemma~\ref{lm:zero}.  Moreover, we can assume that no monomial of
  $f$ lies in $\faJij$. Indeed, we know that
  $\sigma(x)(f)=\sigma(x)(f-c_{\alpha}u^{\alpha})$ for any monomial
  $u^{\alpha}$ in $\faJij$ by the strong non-Archimedean triangle
  inequality. We claim that
  \begin{equation*}
    \liminf\limits_m\sigma(x^{(m)})(f)=
    \liminf\limits_m\sigma(x^{(m)})(f-c_{\alpha}u^{\alpha}).\label{eq:18}  
\end{equation*}
If the left-hand side equals $0$, then the identity follows from the
non-Archimedean triangle inequality and the continuity of $\sigma$
when evaluated at monomials in $R(ij)$. If the left-hand side is
strictly positive, then we know that
$\sigma(x^{(m)})(f)>\sigma(x^{(m)})(c_{\alpha}u^{\alpha})$ for $m\gg
0$, and the result is again a consequence of the strong
non-Archimedean triangle inequality.

We prove the statement by induction on $n$, always assuming that $f$
contains no monomial in $\faJij$. If $n = 3$, then $T$ is the
caterpillar tree on three leaves, and we can choose $I = \tilde I =
I(ij)$. In this situation, the result follows immediately. When $n >
3$, we distinguish two cases. For the remainder of the proof, we fix a
partial order $\preceq$ on $[n]\smallsetminus \{i,j\}$ with the cherry
property on $T$ with respect to $i,j$.

\smallskip\noindent \textbf{Case 1:} Assume that for some maximal
element $s$ in the partial order $\preceq$ at least one of $is, js \in
J$.  By symmetry, we may assume that $js\in J$. In particular, we know
that $f$ contains no monomial involving the unknown $u_{js}$.  We
remove $s$ from $T$ and call $T'$ the induced tree and $\preceq'$ the
restriction of the order $\preceq$ to the leaves in $T'$.  We let
$J'=J\cap \binom{[n]\smallsetminus \{s\}}{2}$ and
$\tilde{J}'=\tilde{J}\cap \binom{[n]\smallsetminus \{s\}}{2}$.  We
choose $I$ compatible with $\preceq$ and $J(ij)$ and $\tilde I$
compatible with $\preceq$ and $\tilde{J}(ij)$, both containing the
pair $is$.  We let $R(ij)'=K[u_{il},u_{jl}: l\neq i,j,s]$.

Consider the projection $\pi_s\colon
\TP\pr^{\binom{n}{2}-1}\dashrightarrow \TP\pr^{\binom{n-1}{2}-1}$
obtained by removing all coordinates indexed by pairs containing
$s$. This map corresponds to the natural inclusion
$R(ij)'\hookrightarrow R(ij)$.  When restricted to $U_{ij}$, the
projection $\pi_s$ is well defined and its image lies in the affine
patch $\cT U_{ij}'= \{x_{ij}\neq -\infty\}\subset
\TP\pr^{\binom{n-1}{2}-1}$.  We define $y := \pi_s(x)$ and $y^{(m)} :=
\pi_s(x^{(m)})$ for all $m>0$.
 
We let $\sigma'$ be the section to $\trop$ over $\cTG(2,n-1)$ defined
in Theorem~\ref{thm:section}.  By construction, the intersections of
the sets $I$ and $\tilde{I}$ with $\binom{[n]\smallsetminus \{s\}}{2}$
are compatible with $\preceq'$ and $J'(ij)$ (resp.\ $\tilde{J}'(ij)$),
so we have $\sigma(x)_{|_{R(ij)'}}=\sigma'(y)$ and
$\sigma(x^{(m)})_{|_{R(ij)'}}=\sigma'(y^{(m)})$.  We write
$f=\sum_{b\geq 0} f_b u_{is}^b$, where $f_b\in R(ij)'$.  Then,
$\rr(f)=\sum_b \rr(f_b) u_{is}^b$, $\tilde{\rr}(f)=\sum_b
\tilde{\rr}(f_b) u_{is}^b$, and we know that $\rr(f_b)$ and
$\tilde{\rr}(f_b)$ contain no monomial involving the unknown
$u_{js}$. The inductive hypothesis on each $f_b$ and the continuity of
the exponential function on $\overline{\RR}$ yield
  \begin{align*}
    \sigma(x)(f)& =\max_b\{\sigma'(y)(f_b) \,\sigma(x)(u_{is}^b)\}
    \leq
    \max_b\{ \liminf_m \sigma'(y^{(m)})(f_b)\sigma(x^{(m)})(u_{is}^b)\} \\
    & \leq \liminf_m \,\max_b\{\sigma(x^{(m)})(f_b u_{is}^b) \}=
    \liminf_m \sigma(x^{(m)})(f)\text{.}
  \end{align*}

  \smallskip\noindent \textbf{Case 2:} Assume that every maximal
  element $s$ in the partial order $\preceq$ satisfies $is,js\notin
  J$.  In this situation, Lemma~\ref{lm:ConditionsIAndJ} ensures that
  for all $l \neq i,j$ either $il,jl\in J$ or $il,jl\notin {J}$, and
  similarly for the set $\tilde{J}$.

  First suppose that there is a leaf $l$ with $il, jl\in \tilde{J}$,
  say $l\in T_a$ for some $a=1,\ldots, m$. We remove $l$ from $T$ and
  call $T'$ the resulting tree and $\preceq'$ the induced order.  As
  in Case 1, we consider the projection $\pi_l\colon \cT U_{ij}\to \cT
  U_{ij}'\subset \TP\pr^{\binom{n-1}{2}-1}$ that deletes all
  coordinates involving $l$.  We define $R(ij)'=K[u_{ik}, u_{jk}:
  k\neq i,j,l]$ and we let $\sigma'$ be the section to $\trop$ over
  $\cT U_{ij}'\subset\cTG(2,n-1)$. Choose $I'$ and $\tilde{I}'$ in
  $\binom{[n]\smallsetminus \{l\}}{2}$ compatible with $\preceq'$ and
  $J'(ij)$ (resp.\ $\tilde{J}'(ij)$), with $\{ik: k\neq i,j,l\}\subset
  I'$ and $\{ik: k\neq i,j,l\}\subset \tilde{I}'$. We write $\rr'$ and
  $\tilde{\rr}'$ for the associated embeddings given by
  Proposition~\ref{pr:transBasis}~\emph{(ii)}.

  By assumption, no monomial of $f$ includes the unknowns $u_{il}$ nor
  $u_{jl}$, thus we know $f\in R(ij)'$. In addition, the expressions
  $\rr'(u_{ik})$, $\rr'(u_{jk})$, $\tilde{\rr}'(u_{ik})$ and
  $\tilde{\rr}'(u_{jk})$ for $k\neq i,j,l$ do not involve any unknown
  indexed by a pair containing $l$. Since $il,jl\in \tilde{J}$, we can
  define a set $\tilde{I}$ compatible with $\preceq$ and
  $\tilde{J}(ij)$ by adding to $\tilde{I}'$ a set of the form $\{il,
  jl\}$, $\{il, lt\}$, or $\{jl,lt\}$ for some $t\prec l$. The same
  method applies to $I$ and $I'$.  In particular, $\tilde{I}'\subset
  \tilde{I}$ and $I'\subset I$, so
  Proposition~\ref{pr:transBasis}~\emph{(iii)} ensures that
  $\rr(f)={\rr'}(f)$ and $\tilde{\rr}(f)=\tilde{\rr}'(f)$. The result
  follows by the inductive hypothesis.

  Second, we assume that no pair $il,jl$ lies in $\tilde{J}$. If the
  set $J(ij)$ is also empty, we can take $\tilde{I}=I$ and
  $\tilde{\rr}=\rr$, and the result follows immediately. Thus, we may
  suppose that some subtree $T_a$ has a leaf $l$ satisfying $il,jl\in
  J$. Pick the minimal element $l$ with this property. As before, we
  remove $l$ from $T$ and $T_a$, obtaining trees $T'$ and $T_a'$, and
  define $R(ij)' = K[u_{ik}, u_{jk}: k \neq i,j,l]$. By assumption,
  $f\in R(ij)'$.  As above, we let $\sigma'$ be the section to $\trop$
  over $\cTG(2,n-1)$, and $\pi_l$ the projection map.  We define $y :=
  \pi_l(x)$ and $y^{(m)} := \pi_l(x^{(m)})$ for all $m>0$.  Following
  Proposition~\ref{pr:CompatibleI}, we construct sets $I'$ and
  $\tilde{I}'$ compatible with $\preceq$ and $J'(ij)$ (resp.,\
  $\tilde{J}'(ij)$) such that $\tilde{\A}' \subset \A'$ and $ik\in
  \tilde{\A}'$ for all $k\in T_a'$. We let $\rr'$ and $\tilde{\rr}'$
  be the corresponding maps from~\eqref{eq:diagram} defined on
  $R(ij)'$.

  We claim that for all $f \in R(ij)'$
 \begin{equation}
   \sigma(x)(f)=\sigma'(y)(f)
   \quad \text{ and }\quad
   \sigma(x^{(m)})(f)=
   \sigma'(y^{(m)})(f).\label{eq:52}
 \end{equation}
 The original statement will follow immediately from these two
 identities and the inductive hypothesis.

 We start by proving the left-hand expression of~\eqref{eq:52}. Since
 $il,jl\in J$, we can extend $I'$ to a compatible set $I$ with respect
 to $\preceq$ and $J(ij)$, with associated map $\rr$ defined on
 $R(ij)$. As before, we conclude that $\rr'=\rr$ on $R(ij)'$ so $
 \sigma(x)(f)=\sigma'(y)(f)$ for all $f\in R(ij)'$.

 Recall that $il,jl \in J\smallsetminus \tilde{J}$ and that
 $\tilde{J}(ij)=\emptyset$. As opposed to the previous scenario, the
 difficulty in proving the right-hand side of~\eqref{eq:52} arises
 because we will not be able to extend $\tilde{I}'$ to a compatible
 set $\tilde{I}$ with respect to $\preceq$ and $\tilde{J}(ij)$
 whenever $l$ is not maximal with respect to $\preceq$. Moreover,
 given any compatible set $\tilde{I}$, its associated $\tilde{\rr}$
 will not agree with $\tilde{\rr}'$ over $R(ij)'$. We will need to
 modify both the set $\tilde{I}'$ and the map $\tilde{\rr}'$ to build
 $\tilde{I}$ and $\tilde{\rr}$.

 From now on, we assume $\tilde{I}$ and $I$ are compatible sets
 satisfying the conditions $\{ik: k\neq i,j,l\}\subset\tilde{H}'
 \subset \tilde{H} \subset H$ and $\{ik: k\neq i,j,l\}\subset
 H'\subset H$.  To simplify notation, we let $l-1$ and $l+1$ be the
 (possibly nonexistent) predecessor and successor elements of $l$ with
 respect to $\preceq$.  Both leaves lie in $T_a$.  We analyze two
 cases: whether $l$ is the first leaf of $T_a$ or not.

 If $l$ is the first leaf of $T_a$, the condition that $i(l+1), j(l+1)
 \notin \tilde{J}$ ensures that
 \[
 \tilde{I}=(\tilde{I}'\smallsetminus \{j(l+1)\}) \cup \{il,jl, (l+1)l\}.
 \]
 Notice that $\tilde{\rr}'(u_{j(l+1)})=u_{j(l+1)}$, whereas
 $\tilde{\rr}(u_{j(l+1)})=u_{il}^{-1}(-u_{(l+1)l}+u_{i(l+1)}u_{jl})$.
 The proof of Proposition~\ref{pr:transBasis}~\emph{(iii)} shows that
 no element in the image of $\tilde{\rr}'$ contains a monomial with a
 negative power of $u_{j(l+1)}$. We write
 \[\tilde{\rr}'(f)=\sum_{b\geq 0} f_b \,u_{j(l+1)}^b,\quad \text{
   where all } f_b\in K[u_{rs}: rs\in \tilde{I}'\smallsetminus
 \{j(l+1)\}][u_{rs}^{-1}: rs\in \tilde{H}'\smallsetminus 
 \{j(l+1)\}].
\]
In order to obtain $\tilde{\rr}(f)$ from $\tilde{\rr}'(f)$, we replace
$u_{j(l+1)}^b$ with $\tilde{\rr}(u_{j(l+1)}^b)$ in the expression of
$\tilde{\rr}'(f)$:
 \begin{equation}
   \tilde{\rr}(f)=\sum_{b\geq 0} f_b\, u_{il}^{-b}\sum_{k=0}^b
   \binom{b}{k}(-1)^k  u_{(l+1)l}^k
   u_{i(l+1)}^{b-k} u_{jl}^{b-k}.\label{eq:49}
 \end{equation}
 None of the variables $u_{il}$ and $u_{jl}$ appear in $f_b$, so there
 are no cancellations among the summands of~\eqref{eq:49}.  The
 multiplicativity of $\sigma(x^{(m)})$ and the definition of
 $\sigma(x^{(m)})$ yield
 \[
\sigma(x^{(m)})(f) = \max_{ b,k} \{
 \sigma(x^{(m)})(f_b)\,\sigma(x^{(m)}) (\binom{n}{k} u_{il}^{-b}
 u_{(l+1)l}^k u_{i(l+1)}^{b-k} u_{jl}^{b-k})\}.
\] 
By the four-point
 condition on the quartet $\{i,j,l,l+1\}$ and the definition of
 $\tilde{\rr}(u_{j(l+1)})$ we have
 \[\sigma'(y^{(m)})(u_{j(l+1)})=\sigma(x^{(m)})(u_{j(l+1)})=
 \sigma(x^{(m)})(-u_{il}^{-1}u_{i(l+1)}u_{jl})\geq
 \sigma(x^{(m)})(u_{il}^{-1}u_{(l+1)l}).
 \]
 In addition, we have $\sigma(x^{(m)})(f_b)= \sigma'(y^{(m)})(f_b)$
 for all $b$ by Proposition~\ref{pr:transBasis}~\emph{(iii)}.
 Therefore,
\[\sigma(x^{(m)}) (f) = \max_{b} \{ \sigma(x^{(m)})(f_b) \,\sigma(x^{(m)}) (u_{il}^{-b}
u_{i(l+1)}^{b} u_{jl}^{b})\} = \max_{b}
\{\sigma'(y^{(m)})(f_b\,u_{j(l+1)}^{b})\} = \sigma'(y^{(m)}) (f).\]

Finally, assume that $l$ is not the first leaf of $T_a$. Since $l$ is
also not the maximal leaf of $T_a$, we know that $l-1$ and $l+1$ are
true leaves of $T_a$. In this case
 \[
 \tilde{I}= (\tilde{I}'\smallsetminus \{(l+1)(l-1)\})\cup \{il,
 l(l-1), (l+1)l\}.
 \]
As before,  we write the expression of $\tilde{\rr}'(f)$:
 \[\tilde{\rr}'(f)=\sum_{b\geq 0} f_b \,u_{(l+1)(l-1)}^b,\quad \text{ where } f_b\in K[u_{kl}: kl\in \tilde{I}'\smallsetminus
 \{(l+1)(l-1)\}][u_{kl}^{-1}: kl\in \tilde{H}'].
\]

In order to obtain $\tilde{\rr}(f)$ from the expression of
$\tilde{\rr}'(f)$, we replace the power $u_{(l+1)(l-1)}^b$ by a
Laurent polynomial in $K[u_{kl}: kl\in \tilde{I}][u_{kl}^{-1}: kl \in
\tilde{\A}\smallsetminus \tilde{J}]$. Using the Pl\"ucker relations,
we write $u_{(l+1)(l-1)} = u_{il}^{-1}(u_{i(l+1)}u_{l(l-1)} +
u_{i(l-1)}u_{(l+1)l}) $. We obtain:
 \begin{equation}
   \tilde{\rr}(f)=\sum_{b\geq 0}f_b \,u_{il}^{-b} \sum_{k=0}^b \binom{b}{k}
   u_{i(l+1)}^k u_{l(l-1)}^k u_{i(l-1)}^{b-k} u_{(l+1)l}^{b-k}.\label{eq:11}
\end{equation}
Again, no cancellations occur among the summands in~\eqref{eq:11} by
Proposition~\ref{pr:transBasis}~\emph{(iii)}.

As before, the extensions of $\sigma'(y^{(m)})$ and $\sigma(x^{(m)})$
to this Laurent polynomial ring agree on all $f_b$.  Next, we find the
term on the right-hand side of~\eqref{eq:11} achieving the value of
$\sigma(x^{(m)})(f)$. Since $l-1\prec l \prec l+1$, the cherry
property of $\preceq$ with respect to $T$ ensures that either
$\{l,l+1\}$ or $\{l-1,l\}$ is a cherry of this quartet.

In the former case, the four-point condition on the quartet
$\{i,l-1,l, l+1\}$ ensures that
 \[
 \sigma(x^{(m)})(u_{(l+1)(l-1)})=
 \sigma(x^{(m)})(-u_{il}^{-1}u_{i(l+1)}u_{l(l-1)})\geq
 \sigma(x^{(m)})(u_{il}^{-1}u_{i(l-1)}u_{(l+1)l}).
 \]
 In particular, this implies that
 \begin{align*}
   \sigma(x^{(m)})(f) &= \max_{b} \{
   \sigma(x^{(m)})(f_b)\,\sigma(x^{(m)})(u_{il}^{-b}u_{i(l+1)}^{b}u_{l(l-1)}^{b})\}
   = \max_{b} \{ \sigma(x^{(m)})(f_b) \sigma(x^{(m)})(
   u_{(l+1)(l-1)}^b)\}\\
   & = \max_{b} \{ \sigma'(y^{(m)})(f_b) \sigma'(y^{(m)})(
   u_{(l+1)(l-1)}^b)\} =\sigma'(y^{(m)})(f).
 \end{align*}

 Finally, if $\{l-1,l\}$ is a cherry of the quartet $\{i,l-1,l,l+1\}$
 we have
 \[
 \sigma(x^{(m)})(u_{(l-1)(l+1)})=
 \sigma(x^{(m)})(-u_{il}^{-1}u_{i(l-1)}u_{(l+1)l})\geq
 \sigma(x^{(m)})(u_{il}^{-1}u_{i(l+1)}u_{l(l-1)}).
 \]
 We conclude that
 \begin{equation*}
   \sigma(x^{(m)})(f) = \max_{ b} \{
   \sigma(x^{(m)})(f_b u_{il}^{-b}u_{i(l-1)}^{b}u_{(l+1)l}^{b})\}
   = \max_{b} \{
   \sigma'(y^{(m)})(f_b u_{(l+1)(l-1)}^{b})\}
   =\sigma'(y^{(m)})(f).
   \qedhere
 \end{equation*}
\end{proof}

\section{Fibers of tropicalization}
\label{sec:fibers-trop}

In this section, we study the fibers of the tropicalization map
$\trop\colon \Gr(2,n)^{\an} \rightarrow \cT \Gr(2,n)$. These fibers
are affinoid spaces (see~\cite[Section 4.13]{BPR}). We describe them explicitly in
Proposition~\ref{prop:fiber}.  This perspective gives a natural
geometric explanation for the maximality property of our section
$\sigma$ by means of Shilov boundaries. For a different approach to
investigate these fibers, we refer to~\cite{FibersOfTrop}.

As we saw in Section~\ref{sec:exist-gener-tree}, the vanishing sets of
points in $\cTG(2,n)$ play a crucial role when constructing the
section $\sigma$. They induce a stratification of $\Gr(2,n)$ into
subvarieties of tori associated to complements of coordinate
hyperplanes. For every $J\subsetneq \binom{[n]}{2}$, we define
\begin{equation}
E_J = \{p \in \pr^{\binom{n}{2}-1}_K: p_{kl} = 0 \mbox{ if and
  only if } kl \in J\}.\label{eq:22}
\end{equation}
Using the Pl\"ucker embedding $\varphi$ from~\eqref{eq:Pluecker}, we
define a locally closed subscheme of $\Gr(2,n)$ (endowed with the
reduced-induced structure):
\begin{equation}
  \label{eq:20}
  \Gr_J(2,n) = \varphi^{-1}(E_J).
\end{equation}
For example, the stratum associated to the empty set is
$\Gr_{\emptyset}(2,n)=\Gr_0(2,n)$.  Since we only consider $J \neq
(\binom{[n]}{2})$, we can always choose $ij \notin J$ and regard
$Gr_J(2,n)$ inside the big open cell $U_{ij}$. Notice that
$\Gr_J(2,n)$ will be nonempty if and only if $J$ is the vanishing set
of some point in $\cTG(2,n)$. The next result explains how to certify
this condition.  We discuss the case of $n=4$ in
Example~\ref{ex:degn=4}.
\begin{lemma}
  \label{lm:CharJ} Let $J\subsetneq \binom{[n]}{2}$ and assume
  $ij\notin J$. Then, $\Gr_J(2,n)$ is nonempty if and only if the set
  $J$ satisfies the following saturation conditions:
  \begin{enumerate}[(i)]
  \item if $kl,ls\in J$ and $ks\neq ij$, then $ks\in J$ or $il,jl\in J$,
\item if $ik,jk\in J$, then $kl\in J$ for all $l$.
  \end{enumerate}
\end{lemma}
\begin{proof}
  Assume $\Gr_J(2,n)$ is nonempty. By the previous discussion, we know
  that $J=J(x)$ for some $x\in \cTG(2,n)$.  Let $T$ be such that $x\in
  \overline{\cTC_T}$. We arrange $T$ as in the right of
  Figure~\ref{fig:caterpillar}. Assume $kl,ls\in J$. The four-point
  conditions on the quartets $\{i,s,k,l\}$ and $\{j, s,k,l\}$ give
  $x_{il}+x_{ks}=x_{jl}+x_{ks}=-\infty$. Thus, condition \emph{(i)}
  holds.  Similarly, if $ik,jk\in J$, the four-point condition on the
  quartet $\{i,j,k,l\}$ and our assumption that $ij\notin J$ yield
  \emph{(ii)}.

  Conversely, suppose that $J$ is saturated.  To show that $Gr_J(2,n)$
  is nonempty, it suffices to construct an $L$-valued point of
  $Gr_J(2,n)$ for some extension field $L|K$. Such a point is
  represented by a matrix $X\in L^{2\times n}$ whose only vanishing
  $2\times 2$-minors are those indexed by pairs in $J$.  Define
\begin{equation}
  \label{eq:40}
  Z_0(J)=\{l\in [n] : il,jl\in J\}.
\end{equation}
Note that $i,j\notin Z_0(J)$. We define a relation $\sim$ on
$[n]\smallsetminus Z_0(J)$ as follows:
\begin{equation}\label{eq:39}
  k\sim l\quad \text{ if and only if }\quad kl\in J \text{ or }k=l.
\end{equation}
The saturation conditions ensure that $\sim$ is an equivalence
relation.

We decompose $[n]\smallsetminus Z_0(J)=\bigsqcup_{k=1}^t B_k$ into its
equivalence classes.  Notice that $i\not\sim j$, so we may assume that
$i\in B_1$ and $j\in B_2$.  Pick a finite field extension $L|K$
containing at least $t-1$ elements, and choose $t-2$ distinct elements
$c_3,\ldots, c_t\in L^{\ast}$.  For each $k=3,\ldots, t$, we consider
the column vector $v_k:=(1,c_k)$. We set $v_1=(1,0)$ and
$v_2=(0,1)$. Using these vectors, we build a rank two matrix
$X=(X^{(1)} | \ldots | X^{(n)})$ by columns, where
$X^{(l)}=\mathbf{0}$ for all $l\in Z_0(J)$, and $X^{(l)}=v_k$ if and
only if $l\in B_k$.  The vectors $\{v_1,\ldots, v_t\}$ are pairwise
linearly independent, so $X$ induces an $L$-valued point in
$\Gr_J(2,n)$.  This concludes our proof.
\end{proof}

From now on, we assume $J$ is always a vanishing set and we fix
$ij\notin J$. We view the set $\Gr_J(2,n)$ inside $U_{ij}$. We let
$\fa_J$ be the ideal of $R(ij)$ generated by $\{u_{kl}: kl\in J\}$. As
usual, if $kl\notin I(ij)$, we interpret
$u_{kl}=u_{ik}u_{jl}-u_{il}u_{jk}$. Our next result implies that
$\Gr_J(2,n)$ is irreducible.

\begin{lemma}\label{lm:a_JPrime}
The ideal $\fa_J$ of $R(ij)$ is  prime.
\end{lemma}
\begin{proof} Fix $x\in \cTG(2,n)$ such that $J=J(x)$.  We proceed by
  induction on $n$. If $n=3$, then $\fa_J =\fa_{J(ij)}$ is a monomial
  ideal generated by degree-one monomials, hence it is prime.

  Suppose $n>3$ and recall the set $Z_0(J)$ from~\eqref{eq:40}. We
  analyze four cases. First, assume $Z_0(J)$ is nonempty, pick any
  $s\in Z_0(J)$ and let $R(ij)'=K[u_{ik}, u_{jk}: k\neq i,j,s]$. The
  set $J'=J\cap \binom{[n]\smallsetminus \{s\}}{2}$ is the vanishing
  set of the projection of $x$ away from all coordinates containing
  $s$.  We write $\fa_J=\fa_{J'}+ \langle u_{is}, u_{js}\rangle$. The
  ideal $\fa_{J'}\subset R(ij)'$ is prime by the inductive
  hypothesis. Since $R(ij)/\fa_J\simeq R(ij)'/\fa_{J'}$ is an integral
  domain, the result follows.

  On the contrary, assume $Z_0(J)$ is empty. Recall the equivalence
  relation $\sim$ on $[n]$ from~\eqref{eq:39} and its equivalence
  classes $\{B_k: k=1,\ldots,t\}$, where $i\in B_1$ and $j\in B_2$. If
  $B_1\sqcup B_2=[n]$, then $\fa_J=\langle u_{ik}, u_{jl}: k\in B_1,
  l\in B_2\rangle$ is a monomial prime ideal of $R(ij)$ by
  construction.

  Suppose $B_1\sqcup B_2\neq [n]$. If $B_1\neq \{i\}$, we pick an
  element $s\in B_1\smallsetminus \{i\}$. Then, $J'= J\cap
  \binom{[n]\smallsetminus \{s\}}{2}$ is a vanishing set and
  $\fa_{J}=\fa_{J'} +\langle u_{is}\rangle$ by
  Lemma~\ref{lm:CharJ}. The ideal $\fa_{J'}\subset R(ij)'$ is prime by
  the inductive hypothesis, so $R(ij)/\fa_J \simeq
  (R(ij)'/\fa_{J'})\otimes_K K[u_{js}]$ is an integral domain and
  $\fa_J$ is a prime ideal. An analogous statement proves the result
  when $B_2\neq \{j\}$.

  Finally, assume $Z_0(J)=\emptyset$ and that $B_1=\{i\}$,
  $B_2=\{j\}$.  We show that the localization $\fa_J\subset
  K[u_{il}^{\pm},u_{jl}^{\pm}: l\neq i,j]$ is a prime ideal. Using the
  decomposition $[n]=\bigsqcup_{k=1}^t B_k$, we write $\fa_J = \sum_{k
    : |B_k|>1}\fa_{J_k}$, where $\fa_{J_k}=\langle
  u_{ir}u_{jl}-u_{il}u_{jr}: r,l\in B_k\rangle$ for all $k$ such that
  $|B_k|>1$.  For each $k=3,\ldots,t$, we fix an element $q_k\in B_k$.
  The ideals $\fa_{J_k}$ involve disjoint sets of variables and we can
  rewrite them as
  \begin{equation*}
    \fa_{J_k}=\langle u_{iq_k}u_{jl}u_{jq_k}^{-1}u_{il}^{-1}-1: l\in
    B_k\smallsetminus \{q_k\}\rangle
    \subset     K[u_{ir}^{\pm},u_{jr}^{\pm}: r\in B_k][u_{il}^{\pm},
    u_{jl}^{\pm}: l\notin (\{i,j\}\cup B_l )].
\end{equation*} 
From this we see that the set of exponents of the binomial generators
of $\fa_J$ generate a primitive sublattice of $\ZZ^{2(n-2)}$, so
$\fa_J$ is prime by~\cite[Theorem 2.1]{binomialIdealsES}.
\end{proof}

Lemma~\ref{lm:a_JPrime} implies that the coordinate ring of
$\Gr_J(2,n)$ is $(R(ij)/ \fa_J)_{S_{(J\cup\{ij\})^c}}$.  Here, we view
$S_{(J\cup \{ij\})^c}$ as the multiplicative subset generated by the
residue classes of all $u_{kl}$ with $kl \notin J \cup \{ij\}$.

The embedding $\Gr_J(2,n)^{\an} \subset U_{ij}^{\an}$ identifies
multiplicative seminorms on the coordinate ring of $\Gr_J(2,n)$ with
those multiplicative seminorms on $R(ij)$ which vanish precisely on
$\fa_J$.  We view the fiber of $\trop$ over any point $x$ with
vanishing set $J$ as an element in $\Gr_J(2,n)^{\an}$.

We fix a tree $T$ such that $x\in \overline{\cTC_T}$.  As in
Section~\ref{sec:exist-gener-tree}, we fix a partial order $\preceq$
on $[n]\smallsetminus \{i,j\}$ having the cherry property on $T$ with
respect to $i$ and $j$, and we choose a set $I$ compatible with
$\preceq$ and $J(ij)$.  In order to caracterize the fiber
$\trop^{-1}(x)$ in algebraic terms, we will need a suitable
description of the coordinate ring of $\Gr_J(2,n)$. We use the
following adaptation of the system of coordinates $I$ from
Section~\ref{sec:cont-sect-from}.

Recall that the two embeddings from \eqref{eq:diagram} are compatible
with the embeddings in the function field of $R(ij)$. We extend the
map $\ii$ to $\ii\colon K[ u_{kl}: kl \in I][u_{kl}^{-1}: kl\in
\A\smallsetminus J]\to R(ij)_{S_{\A\smallsetminus J}}$ so that
$\ii\circ \rr=\id$ and $\rr\circ \ii=\id$.  The embeddings $\rr$ and
$\ii$ induce the following inclusions
\begin{equation}\label{eq:30}
  \xymatrix@1{
    K[u_{kl}: kl \in I]/\ii^{-1}(\fa_J)\; \ar@{^{(}->}[r]^-{\overline{\ii}} & R(ij)/\fa_J\; \ar@{^{(}->}[r]^-{\overline{\rr}} & K[ u_{kl}: kl \in
    I][u_{kl}^{-1}: kl\in \A\smallsetminus J]}/\langle \rr(\fa_J)\rangle.
\end{equation}
The second map is injective by Lemma~\ref{lm:a_JPrime}.

Using these maps, we identify the coordinate ring of $\Gr_J(2,n)$ with
a suitable localization of the ring on the left-hand side
of~\eqref{eq:30}.  By definition, the ideal $\langle
\rr(\fa_J)\rangle$ is generated by the set $\{\rr(u_{kl}): kl\in J\}$.
The condition $\sigma(x)(u_{kl})=0$ and the formula~\eqref{eq:section}
imply that it is in fact generated by the monomials $\{u_{kl}: kl\in
I\cap J\}$.  Using this fact, we rewrite ~\eqref{eq:30} as follows:
\begin{equation}\label{eq:7}
  \xymatrix@1{
    K[u_{kl}: kl \in I\smallsetminus J]\; \ar@{^{(}->}[r]^-{\overline{\ii}} & R(ij)/\fa_J\; \ar@{^{(}->}[r]^-{\overline{\rr}} & K[ u_{kl}: kl \in
    I\smallsetminus J][u_{kl}^{-1}: kl\in \A\smallsetminus J]}.
\end{equation}

We let $S$ be the multiplicatively closed subset of $K[ u_{kl}: kl \in
I\smallsetminus J][u_{kl}^{-1}: kl\in \A\smallsetminus J]$ generated
by the polynomials
\begin{equation}
  \label{eq:33}
  f_{rs}:= \overline{\rr}(u_{rs}) \in K[u_{kl}: kl\in I\smallsetminus J][u_{kl}^{-1}:
  kl\in \A\smallsetminus J] \qquad \text{ for all }rs \notin J\cup\{ ij\},
\end{equation}
where we view all $u_{rs}$ in $R(ij)/\fa_J$.  The following result
gives a new description of the coordinate ring of $\Gr_J(2,n)$.

\begin{lemma}\label{lem:iso}   The inclusion $\overline{\rr}$
  from~\eqref{eq:7} induces an isomorphism
  \[
  \xymatrix@1{ \overline{\rr}\colon (R(ij) / \fa_J)_{S_{(J\cup
        \{ij\})^c}} \ar[rr]^-{\sim} & & (K[u_{kl}: kl \in I
    \smallsetminus J][u_{kl}^{-1}: kl \in \A \smallsetminus
    J])_S\text{.}} \]
\end{lemma}

\begin{proof} Since the map $\overline{\rr}$ from~\eqref{eq:7}
  satisfies $\overline{\rr}(S_{(J\cup \{ij\})^c})=S$ by~\eqref{eq:33},
  we can extend this map to the localizations in the statement. The
  resulting map is injective by~\eqref{eq:7}, and surjective by
  construction.
\end{proof}

As an immediate corollary, we give a first description of the fibers
of $\trop$. As we said earlier, given $x\in \cT U_{ij}$ with $J=J(x)$,
we know that the fiber $\trop^{-1}(x)$ is contained in the analytic
stratum $\Gr_J(2,n)^{\an}$. Assume $x\in \overline{\cTC_T}$ for some
tree type $T$. Lemma~\ref{lem:iso} shows that the coordinate ring of
$\Gr_J(2,n)$ is isomorphic to $(K[u_{kl}: kl \in I \smallsetminus
J][u_{kl}^{-1}: kl \in \A \smallsetminus J])_S$.  We identify
$\trop^{-1}(x)$ with the set of all multiplicative seminorms $\gamma$
on $(K[u_{kl}: kl \in I \smallsetminus J][u_{kl}^{-1}: kl \in \A
\smallsetminus J])_S$ extending the absolute value on $K$ such that
$\gamma(f_{rs}) = \exp(x_{rs}-x_{ij})$ for all $rs \notin J\cup \{
ij\}$.  Our choice of $I$ ensures that $f_{rs} = u_{rs}$ if $rs \in
I\smallsetminus J$. In particular,
\begin{equation}
  K[u_{kl}^{\pm} :kl\in I\smallsetminus J]\subset
  (K[u_{kl}:
  kl \in I \smallsetminus J][u_{kl}^{-1}: kl \in \A \smallsetminus
  J])_S,\label{eq:27}
  \end{equation}
  and we can restrict each $\gamma$ to the leftmost Laurent polynomial
  ring. Conversely, any $\gamma$ defined on the left-hand side that
  satisfies $\gamma(f_{kl})=\exp(x_{kl}-x_{ij})$ for all $kl\notin J
  \cup \{ij\}$ has a unique extension to the right-hand side and,
  thus, lies in $\trop^{-1}(x)$.

\medskip

Let us now introduce notation for two examples of affinoid algebras
giving rise to non-Archimedean polydiscs and polyannuli. For a general
treatment of affinoid algebras and their Berkovich spectra, we refer
to~\cite[Section 2.1]{berkovichbook}. Throughout the remainder, we
keep the multi-index notation from Section~\ref{subsec:analyti}.
\begin{definition}\label{def:disc}
  For every $r = (r_1, \ldots, r_n) \in \mathbb{R}_{>0}^n$, we set
\begin{equation*}
  K\{ r_1^{-1} x_1, \ldots, r_n^{-1} x_n \} = \left\{\sum_{\alpha\in
      \mathbb{N}_0^n} c_\alpha {x}^\alpha: |c_\alpha| r^\alpha
    \rightarrow 0 \mbox{ as }|\alpha| \rightarrow \infty \right\}.
\end{equation*}
We define a Banach norm on this algebra by $\| \sum_{\alpha} c_\alpha
{x}^\alpha \| = \max\limits_\alpha |c_\alpha| r^\alpha$.

\end{definition}
\noindent As we said before, the algebra $K\{ r_1^{-1} x_1, \ldots,
r_n^{-1} x_n \}$ is an example of an affinoid algebra. We say that it
is \emph{strictly affinoid} if some power of each $r_i$ is contained
in the value group $|K^\ast|$.
\begin{definition}\label{def:annulus} For every $r = (r_1, \ldots,
  r_n) \in \mathbb{R}_{>0}^n$, we set
  \begin{equation*}
    K\{ r_1^{-1} x_1, r_1 x_1^{-1}, \ldots, r_n^{-1} x_n, r_n
    x_n^{-1}  \} =  \left\{\sum_{\alpha \in \mathbb{Z}^n}
      c_\alpha {x}^\alpha: |c_\alpha| r^\alpha \rightarrow 0 \mbox{ as
      }|\alpha| \rightarrow \pm \infty \right\}.
  \end{equation*}
  It is a Banach algebra with respect to the norm $\| \sum_{\alpha}
  c_\alpha {x}^\alpha \| = \max\limits_\alpha |c_\alpha| r^\alpha$.
\end{definition}

Given an affinoid algebra $A$, such as those in
Definitions~\ref{def:disc} and~\ref{def:annulus}, we construct its
\emph{Berkovich spectrum} $\mathcal{M}(A)$.  It is defined as the set
of all multiplicative seminorms on $A$ which are bounded with respect
to the Banach norm.  The set $\mathcal{M}(A)$ is endowed with the
coarsest topology such that for all $f \in A$, the evaluation map
$\operatorname{ev}_f\colon\mathcal{M}(A) \rightarrow \RR$, $\gamma
\mapsto \gamma(f)$ is continuous.  As an example, we remark that the
analytification $(\mathbb{A}^n_K)^{\an}$ of the $n$-dimensional affine
space in Example~\ref{ex:affinespace} is precisely the union of all
$\mathcal{M}(K\{r_1^{-1} x_1, \ldots, r_n^{-1} x_n\})$ for $(r_1,
\ldots, r_n) \in \RR_{>0}^n$. For details, we refer to~\cite[Proof of
Theorem 3.4.1]{berkovichbook}.

Starting from an affinoid algebra $A$, two elements $f,g\in A$ and
$r,s\in \RR_{>0}$, we use~\cite[Remark 2.2.2 (i)]{berkovichbook} to
construct a new affinoid algebra that resembles
Definitions~\ref{def:disc} and~\ref{def:annulus}:
\begin{equation}\label{eq:14}
  A\{r^{-1}f, s g^{-1}\} = A\{r^{-1}x,s y\}/(x-f, gy-1)\text{,}
\end{equation}
where, similarly to Definition~\ref{def:disc}, we set
\[ A\{r^{-1}x,s y\} := \left\{\sum_{m,n\geq 0} a_{mn}x^my^n :
  a_{mn}\in A, \|a_{mn}\|r^ms^{-n} \to 0 \mbox{ as } m+n\to
  \infty\right\}\text{.}
\] 
The algebra homomorphism $A \rightarrow
A\{r^{-1}f, s g^{-1}\}$ introduces a continuous and injective map
\[
 \xymatrix@1{
  \mathcal{M}(A\{r^{-1}f, s g^{-1}\})  \; \ar@{^{(}->}[r] &  \mathcal{M}(A)}.
\]
Its image is the set of all seminorms $\gamma$ on $A$ satisfying $
\gamma(f)\leq r$ and $ \gamma(g) \geq s$, see~\cite[Remark 2.2.2
(i)]{berkovichbook}.

By induction, we extend the construction from~\eqref{eq:14} to any
number of elements in $A$ and positive reals. Of particular interest
to us are the affinoid algebras $A\{r_1^{-1} f_1, r_1 f_1^{-1},
\ldots, r_n^{-1} f_n, r_n f_n^{-1}\}$, where $f_1, \ldots, f_n$ are
elements of $A$ and $r_1, \ldots, r_n\in \RR_{>0}$.  In this case, the
image of the map
\begin{equation}
  \xymatrix@1{\mathcal{M}(A\{r_1^{-1} f_1, r_1 f_1^{-1},   \ldots,
    r_n^{-1} f_n, r_n f_n^{-1}\}) \; \ar@{^{(}->}[r] & \mathcal{M}(A)}\label{eq:23}
\end{equation}
is the subset of all $\gamma \in \mathcal{M}(A)$ such that $
\gamma(f_i) = r_i$ for all $i = 1, \ldots, n$.

\medskip We now state the main result in this section. Fix $i,j$ and
pick a point $x \in \cT U_{ij}$. Let $J=J(x)$ be the vanishing set of
$x$, define $\rho_{kl} = \exp({x_{kl} - x_{ij}})$ for each $kl\neq
ij$, and let $I$ be an index set as above.  We fix a tree type $T$
with $x\in \overline{\cTC_T}$ and a compatible set $I=I(ij,T,J)$. We
associate to $x$ the affinoid algebra
  \begin{equation}
    A = K\{\rho_{kl}^{-1} u_{kl}, \rho_{kl} u_{kl}^{-1}: kl \in I
    \smallsetminus J \}\label{eq:24}
  \end{equation}
 and its Laurent domain
  \begin{equation}
    B = A \left\{\rho_{rs}^{-1} {f}_{rs}, \rho_{rs} {f}_{rs}^{-1}: rs
      \notin    ( I \cup J  \cup \{ij\})
    \right\}.\label{eq:21}
  \end{equation}

\begin{prop}\label{prop:fiber}
  The fiber $\trop^{-1}(x)$ is the affinoid subdomain $\mathcal{M}(B)$
  of $\Gr_J(2,n)^{\an}$.
\end{prop}

\begin{proof} 
  We prove the result by double inclusion. Let $\gamma$ be a point in
  $\trop^{-1}(x)$. From our earlier discussion
  following~\eqref{eq:27}, we identify $\gamma$ with a multiplicative
  seminorm on $K[u_{kl}^{\pm}: kl\in I\smallsetminus J]$ satisfying
  $\gamma(f_{rs})=\exp(x_{rs}-x_{ij}) =\rho_{rs}$ for all $rs\notin
  J\cup\{ ij\}$.

  We now show that any such $\gamma$ can be uniquely extended to a
  bounded multiplicative seminorm on ${A}$, i.e.,\ to an element in
  $\mathcal{M}(A)$.  Fix an element $g = \sum_{\alpha} c_{\alpha}
  u^{\alpha}$ of $A$, as in~\eqref{eq:24}. Then, we may write $g$ as
  the limit of a sequence of Laurent polynomials $g_m \in
  K[u_{kl}^{\pm}: kl\in I\smallsetminus J]$ with respect to the Banach
  norm on $A$.  Since $f_{rs} = u_{rs}$ whenever $rs \in
  I\smallsetminus J$, we know that
  $\gamma(c_{\alpha}u^{\alpha})=|c_{\alpha}|\rho^{\alpha}$.  Since
  $\gamma$ is non-Archimedean, it is therefore bounded by the Banach
  norm on $A$ restricted to $K[u_{kl}^{\pm}: kl\in I\smallsetminus
  J]$.  This property implies the inequalities
\[\gamma(g_n - g_m) \leq  \|g_n - g_m\| \leq \max\{ \|g -g_n\|, \| g - g_m\| \},\] 
where the right-hand side goes to zero for $n,m \to \infty$.  In
addition, the reverse triangle inequality for $\gamma$ ensures that
$\gamma(g_n-g_m)\geq |\gamma(g_n)-\gamma(g_m)|$ for all $n,m$. From
this we conclude that $(\gamma(g_n))_n$ is a Cauchy sequence in
$\RR$. We define $\gamma(g):=\lim\limits_{n\to \infty} \gamma(g_n)$.

By construction, $\gamma$ is a multiplicative bounded seminorm on $A$
extending the one on $K[u_{kl}^{\pm}: kl \in I \smallsetminus
J]$. Moreover, since $\gamma(f_{rs})=\rho_{rs}$ for all $rs\notin
I\cup J \cup \{ij\}$, the injective map from~\eqref{eq:23} ensures
that $\gamma$ lies in $\mathcal{M}(B)$. This proves
$\trop^{-1}(x)\subset \mathcal{M}(B)$.

For the converse, it is easy to see that every element in
$\mathcal{M}(B)$ restricts to a multiplicative seminorm $\gamma$ on
$K[u_{kl}^{\pm}: kl \in I \smallsetminus J]$ with
$\gamma(f_{rs})=\exp(x_{rs}-x_{ij})$ for all $rs\notin J \cup
\{ij\}$. This implies that $\gamma\in \trop^{-1}(x)$.
\end{proof}

\medskip

We end this section by discussing the Shilov boundaries of affinoid
algebras.  Given an affinoid algebra $A$, the Shilov boundary in
$\mathcal{M}(A)$ is the unique inclusion minimal closed subset
$\Gamma$ of $\mathcal{M}(A)$ such that, for every $f \in A$, the
continuous function $\mathcal{M}(A) \rightarrow \RR_{\geq 0}$ defined
by $\gamma \mapsto \gamma(f)$ achieves its maximum in $\Gamma$. In
order to prove that there exists a unique minimal closed subset
$\Gamma$ with this property, Berkovich showed a relation between
$\Gamma$ and the reduction $\widetilde {A}$ where $A$ is strictly
affinoid. We recall its definition, as it appears in~\cite[Section
2.4]{berkovichbook}.
\begin{definition} Given a commutative Banach algebra
  $(A,\|\mathord{\cdot}\|)$, we define its \emph{reduction} as
  \begin{equation*}
    \widetilde{A} = \{f \in A: \|f \| \leq 1\} / \{f \in A: \|f\| < 1\}.
  \end{equation*}
  \end{definition}
  
  By~\cite[Proposition 2.4.4]{berkovichbook}, the Shilov boundary of a
  strictly affinoid algebra $A$ corresponds bijectively to the set of
  irreducible components in $\mbox{Spec}(\widetilde{A})$.  In
  particular, if $\mbox{Spec}(\widetilde{A})$ is irreducible, the
  Shilov boundary consists of a single point $\xi$. The point $\xi$
  satisfies $\xi(f)\geq \gamma(f)$ for all $\gamma \in
  \mathcal{M}(A)$.

  We discuss some examples. If $A$ is our ground field $K$, its
  reduction is the residue field $\widetilde{K}$.  If $A = K \{x_1,
  \ldots, x_n\}$, the reduction $\widetilde{A}$ is isomorphic to the
  polynomial ring $\widetilde{K}[x_1, \ldots, x_n]$ over the residue
  field. The reduction of $A = K\{x_1, x_1^{-1}, \ldots, x_n,
  x_n^{-1}\}$ is isomorphic to the Laurent polynomial ring
  $\widetilde{A} = \widetilde{K}[x_1^{\pm}, \ldots,
  x_n^{\pm}]$. Notice that in all these cases, $\Spec \widetilde{A}$
  is an irreducible scheme. In particular, the Shilov boundary of $A$
  in these three examples has exactly one point.


  The next result shows that, for suitable choices of the field $K$,
  the reduction of the Laurent domain $B$ from~\eqref{eq:21} is an
  integral domain.
  \begin{thm}\label{thm:reduction} Assume that all $\rho_{kl}$ are
    contained in the value group $|K^\ast|$. Then, the reduction
    $\widetilde{{B}}
    $ is isomorphic to a localization of the Laurent polynomial ring
    $\widetilde{K}[u_{kl}^{\pm 1}: kl \in I\smallsetminus J]$.
\end{thm}

\begin{proof} Since all $\rho_{kl}$ are contained in the value group
  of $K$, the affinoid algebra $A$ is isomorphic to $K\left\{u_{kl},
    u_{kl}^{-1}: kl \in I\smallsetminus J \right\}$. Its reduction is
  isomorphic to $\widetilde{K}[u_{kl}^{\pm}: kl\in I\smallsetminus
  J]$. Since each polynomial $f_{kl}$ has Banach norm one when viewed
  in $A$ under the aforementioned isomorphism, we use~\cite[\S 7.2.6,
  Proposition 3]{bgr} to conclude that $\widetilde{B}$ is isomorphic
  to a localization of $\widetilde{A}$.
\end{proof}

By~\cite[Proposition 3.7]{gubler12}, tropicalizations are invariant
under field extensions of complete non-Archimedean valued fields. We
conclude:
\begin{corollary}\label{cor:Shilov} The fiber $\trop^{-1}(x) =
  \mathcal{M}(B)$ contains a unique Shilov boundary point.
\end{corollary}
\begin{proof} For any $x \in \cT\Gr(2,n)$, there is a complete
  non-Archimedean extension field $L|K$ whose value group contains all
  $\rho_{kl}$'s. We denote by $\trop_L$ the tropicalization map on
  $\Gr(2,n)^{\an}$ with respect to the Pl\"ucker embedding in
  $\pr^{\binom{n}{2}-1}_L$. By Proposition~\ref{prop:fiber}, we have
  $\trop^{-1}(x) = \mathcal{M}(B)$, and $\trop_L^{-1}(x) =
  \mathcal{M}(B_L) $, where $B_L$ denotes the complete base change
  induced by the extension $L|K$.  By Theorem~\ref{thm:reduction}, the
  reduction $\widetilde{B_L}$ is an integral domain, and so $B_L$
  contains a unique Shilov boundary point~\cite[Proposition
  2.4.4]{berkovichbook}. The natural map $\mathcal{M}(B_L) \rightarrow
  \mathcal{M}(B)$ induces a surjection between the corresponding
  Shilov boundaries. This concludes our proof.
\end{proof}

\section{Tropical multiplicities and piecewise linear structures}
\label{sec:trop-mult-integr}

In this Section, we discuss the piecewise linear structures on both
the tropical and analytic Grassmannians. On the tropical side, this
structure is encoded by initial degenerations of the Pl\"ucker ideal,
viewed in each stratum $\{E_J\}_J$ from~\eqref{eq:22}.  Combinatorial
information about these degenerations is encoded by the tropical
multiplicities.  This structure admits an interpretation in terms of
generalized phylogenetic trees, as will be explained in forthcoming
work of the first author.  On the analytic side, the piecewise linear
structure corresponds to the notion of skeleta of the ambient tori. We
will see that both structures are related by the section $\sigma$ from
Theorem~\ref{thm:MainThm}.

We start by discussing the polyhedral structure on the tropical
Grassmannian and, in particular, the notion of initial
degenerations. We follow the exposition of~\cite[Section 5]{gubler12},
adapting the definitions to follow our max convention. We keep the
notations and conventions of Section~\ref{sec:trop-analyt-vari}.  For
the remainder, we let $L|K$ be an algebraically closed field extension
that is complete with respect to a non-Archimedean valuation extending
that of $K$. We assume that the valuation map $\nu\colon L^{\ast}
\rightarrow \RR$ is surjective.  We let $L^{\circ}=\{l\in L:
\nu(l)\geq 0\}$ be the valuation ring of $L$ and $L^{\circ\circ}$ its
maximal ideal. We let $\widetilde{L}=L^{\circ}/L^{\circ\circ}$ be its
residue field.

We fix a multiplicative split torus $\GG_m^n$ over the field $K$, with
character lattice $\Lambda$.  Given a cocharacter $x\in
\Lambda_{\RR}^{\vee}$, we consider the associated \emph{tilted group
  ring} $L^{\circ}[\Lambda]^x$:
\[
L^{\circ}[\Lambda]^x:=\left\{\sum_{u} c_u \chi^{u} \in L[\Lambda]:
  \nu(c_u) + x(u) \geq 0 \text{ for all }u\right\}.
\]
The tilted group ring contains the valuation ring $L^{\circ}$.

Let $X\subset \GG_m^n$ be a closed subscheme with defining ideal
$\mathfrak{I} \subset K[\Lambda]$. We set
$\mathfrak{I}_L=\mathfrak{I}\otimes_K L$ and define
\[
X^{x}_L:=\Spec(L^{\circ}[\Lambda]^{-x}/(\mathfrak{I}_L \cap
L^{\circ}[\Lambda]^{-x})).
\]
This affine scheme is flat over $L^{\circ}$ and its generic fiber is
isomorphic to $X_L = X \times_{\Spec K} \Spec L$. Its special fiber
$\init_x X $ is the \emph{initial degeneration} of $X$ with respect to
$x$. It is defined by the \emph{initial ideal} of $\mathfrak{I}$ with
respect to $x$, namely, $\init_x\mathfrak{I}=(\mathfrak{I}_L\cap
L^{\circ}[\Lambda]^{-x})\otimes_{L^{\circ}}\widetilde{L}$. This ideal
is generated by all initial forms $\init_x(f)$ in
$\widetilde{L}[\Lambda]$, where $f\in \mathfrak{I}_L$. The initial
forms are defined as follows. Given $f=\sum_{u\in
  \Lambda}c_u\chi^{u}\in L[\Lambda]$, its initial form with respect to
$x$ is the polynomial $\init_x(f)=\sum_{u\in \Lambda}
\overline{\lambda c_u\chi^{u}(\tau)}\chi^{u}$, where $\tau\in
\GG_m^n(L)$ is such that $\trop(\tau)=x$ and $\lambda\in L$ satisfies
$\nu(\lambda)=\max\{ -\nu(c_u)+x(u): u\in \Lambda\}$. The bar denotes
the class of an element in the residue field $\widetilde{L}$. Note
that in the present paper, $\trop(\tau)$ is defined as the cocharacter
$u \mapsto \log | u(\tau)|$, whereas \cite{gubler12} uses $- \log |\,
\mathord{\cdot}\,|$.

By the Fundamental Theorem of tropical geometry (see~\cite[Theorem
5.6]{gubler12}), we know that $x\in \cT X$ if an only if $\init_xX$ is
nonempty. Both $\init_xX$ and $\init_x\mathfrak{I}$ only depend on the
choice of $L$ up to a base change~\cite[Remark 5.4]{gubler12}.

\begin{definition}
  The multiplicity $m_x$ of a point $x \in \cT X$ is the number of
  irreducible components of $\init_xX$, counted with multiplicities.
\end{definition}

In particular, this number is $1$ if and only if the scheme $\init_xX$
is irreducible and generically reduced. It is important to remark
that, unlike most references in the literature, here we wish to
consider multiplicities of all points in $\cT X$ and not only of the
so-called regular points of $\cT X$, i.e.,\ those points around which
$\cT X$ behaves locally like a linear space (see~\cite[Definition
3.1]{ElimTheory}).

\medskip

To study the initial degenerations of $\Gr(2,n)$, we decompose
$\Gr(2,n)$ into its strata $\Gr_J(2,n)$ defined in~\eqref{eq:20}, and
compute their initial degenerations. Without loss of generality, we
assume that our complete valued field $K$ is algebraically closed and
has surjective valuation map $\nu\colon K^{\ast} \rightarrow \RR$. If
this is not the case, we replace $K$ by an appropriate field extension
$L$ as explained above.

To match the previous definitions, we first find an ambient torus for
each variety $\Gr_J(2,n)$. Using the Pl\"ucker embedding and
projecting away from the coordinates indexed by $J$, we view
$\Gr_J(2,n)$ inside $\GG_m^{\binom{n}{2}-|J|}/\GG_m$. In particular,
when $ij\notin J$, we identify this quotient with the torus
$\GG_m^{\binom{n}{2}-1-|J|}=\Spec K[u_{kl}^{\pm}: kl\notin J
\cup\{ij\}]$.  We produce a set of generators for the defining ideal
$\mathfrak{I}$ of $\Gr_J(2,n)$ by removing all monomials in
$\fb_J:=\langle u_{kl}: kl\in J\rangle$ from the three-term Pl\"ucker
relations (\ref{eq:3TermPluecker}) expressed in the $u$-variables. In
particular, we identify $\mathfrak{I}$ with the quotient ideal
$(I_{2,n}+\fb_J)/\fb_J \subset K[u_{kl}^{\pm}: kl\notin J\cup
\{ij\}][u_{kl}: kl\in J]/\fb_J$. The affine coordinates $\{u_{kl}:
kl\notin J\cup \{ij\}\}$ are a basis of the character lattice
$\Lambda$.  We degenerate $\Gr_J(2,n)$ inside this torus with respect
to all points in $\Lambda^{\vee}_{\RR}$. The tropicalization of
$\Gr_J(2,n)$ in $\Lambda_{\RR}^{\vee}$ consists of all points of the
form $y=(x_{kl}-x_{ij})_{kl\notin J\cup \{ij\}}$, where $x\in
\cTG(2,n)$ has vanishing set $J$.


From now on, we assume $ij\notin J$.  Our next goal is to compute the
initial ideal $\init_{y}\mathfrak{I}$ inside
$\widetilde{K}[u_{kl}^{\pm}: kl\notin J \cup\{ij\}]$.  Unfortunately,
the tropical basis property of the 3-term Pl\"ucker relations
reflected in the four-point conditions will not simplify our
calculations.  Instead, we make use of two prior tools: the coordinate
systems $I$ from Section~\ref{sec:exist-gener-tree} and the
polynomials $f_{rs}$ from~\eqref{eq:33}.

Let us recall some notation.  Given $y$ and its associated point $x\in
\cTG(2,n)$, we pick a tree type $T$ satisfying $x\in \CijTJ$. As in
Section~\ref{sec:exist-gener-tree}, we fix a partial order $\preceq$
on $[n]\smallsetminus \{i,j\}$ that has the cherry property on $T$
with respect to $i$ and $j$ and we choose a set $I$ compatible with
$\preceq$ and $J(ij)$.  We let $f_{rs}$ with $rs\notin J\cup \{ij\}$
be as in~\eqref{eq:33}.  Our construction of the set $I$ in terms of
Pl\"ucker relations ensures that the Laurent polynomials
$u_{rs}-f_{rs}$ belong to $\mathfrak{I}$.  The next result says,
precisely, that these polynomials play the role of a Gr\"obner basis
for $\init_y\mathfrak{I}$.

\begin{lemma}\label{lm:InGrobner}
  Let $y\in \cTG_J(2,n)\subset \Lambda_{\RR}^{\vee}$ and $I =
  I(ij,T,J)$ be as above.  Then:
  \begin{equation}
    \label{eq:38}
    \init_{y}\mathfrak{I} 
    =\langle u_{kl}-\init_{y}({f_{kl}}):kl\notin (I\cup
    J\cup \{ij\})\rangle  \subset \widetilde{K}[u_{kl}^{\pm}:
    kl\notin J\cup \{ij\}].
  \end{equation}
\end{lemma}
\begin{proof}
  We prove the result by double inclusion.  The identity
  $\sigma(x)(f_{kl})=\sigma(x)(u_{kl})=\exp(x_{kl}-x_{ij})$ ensures
  that the weight of $\init_y(f_{kl})$ equals $x_{kl}-x_{ij}$, and so
\[
\init_{y}(u_{kl}-f_{kl}) = u_{kl}-\init_{y}(f_{kl}) \in
\widetilde{K}[u_{rs}: rs\notin J\cup \{ij\}] \qquad \text{ for all }kl
\notin J \cup \{ij\}.
\]
Since $u_{kl}-f_{kl}\in \mathfrak{I}$, the inclusion $\supseteq$ in~\eqref{eq:38} holds.

For the converse, it suffices to show that any $y$-homogeneous Laurent
polynomial $f\in \init_{y}\mathfrak{I}$ lies in the right-hand side
ideal. By~\cite[Lemma 2.12]{FJT}, we know that such $f$ is the initial
form of an element $g\in \mathfrak{I} $. Furthermore, since we know
that $u_{kl}-\init_y(f_{kl})$ lies in $\init_y \mathfrak{I}$, we may
assume that $f$ is a $y$-homogeneous Laurent polynomial in the
variables $u_{kl}$, where $kl\in I\smallsetminus J$.

We write $g=u^{\alpha}h$, where $h$ is a polynomial in
$(I_{2,n}+\fb_J)/\fb_J$.  We claim that we can choose $g$ in such a
way that $h$ does not involve any of the variables $u_{kl}$ with
$kl\notin I$. This follows by the properties of $I$.  Indeed, using
the expression
\[
u_{kl}^s=\sum_{r=0}^s \binom{s}{r}
\init_{y}(f_{kl})^{s-r}(u_{kl}-\init_{y}(f_{kl}))^r \qquad \text{ for
  all }s\geq 0,
\]
we can rewrite $h$ as a polynomial in the variables indexed by $I$
(with possibly negative exponents for the variables in
$\A\smallsetminus J$) and the polynomials $\{u_{kl}-\init_{y}(f_{kl}):
kl\notin I\cup J\cup \{ij\}\}$. Namely, we lift $h$ to a polynomial in
$I_{2,n}$ modulo $\fb_J$ and use the previous expression to obtain
\begin{equation*}
  h =\sum_{\alpha}\; h_{\alpha}\!\!\!\!\!
  \prod_{kl\notin I\cup J \cup
    \{ij\}}\!\!\!\!\!(u_{kl}-\init_{y}(f_{kl}))^{\alpha_{kl}}
  \mod \fb_J,
\end{equation*}
where the right-hand side lies in $I_{2,n}$, and each $h_{\alpha}$ is
a Laurent polynomial involving only variables indexed by pairs in
$I\smallsetminus J$. The vector $\alpha$ has nonnegative entries.  By
construction, $h_0\in I_{2,n}$ and it is a Laurent polynomial in
$\{u_{kl}: kl\in I\}$, hence $h_0=0$ by
Proposition~\ref{pr:transBasis}~\emph{(i)}.

Since all the terms $(u_{kl}-\init_{y}(f_{kl}))$ are $y$-homogeneous,
we see that all homogeneous components of $g$ belong to the ideal in
the right of~\eqref{eq:38}. This concludes our proof.
\end{proof}
\begin{rem}
  An alternative proof of the inclusion $\subseteq$ in~\eqref{eq:38}
  goes as follows. The quotient of $\widetilde{K}[u_{kl}^{\pm}: kl
  \notin J \cup \{ij\}]$ by the ideal on the right-hand side
  of~\eqref{eq:38} is an integral domain.  It has dimension
  $|I\smallsetminus J|$, and it maps surjectively to the coordinate
  ring of $\init_y\Gr_J(2,n)$. On the other hand, $\init_y\Gr_J(2,n)$
  is a flat degeneration of $\Gr_J(2,n)$, so its dimension is also
  $|I\smallsetminus J|$. Therefore, both ideals agree.
\end{rem}

As a consequence of this lemma, we see that the coordinate ring of the
initial degeneration is the localization of a Laurent polynomial ring.
\begin{thm}\label{thm:InitDegGrass}
  Given any $y\in \Lambda_{\RR}^{\vee}$, the initial degeneration
  $\init_y\Gr_J(2,n)$ is an integral scheme.
\end{thm}
\begin{proof}
  If $y$ is generic, i.e., if it does not lift to a point $x\in
  \cTG(2,n)$ with vanishing set $J$, we know that
  $\init_y\mathfrak{I}$ is the unit ideal. Thus,
  $\init_y\Gr_J(2,n)=\Spec (0)$ and the result follows.

  Otherwise, assume $y=(x_{kl}-x_{ij})$ with $x\in \cTG(2,n)$ and
  $J=J(x)$. Pick a tree $T$ such that $x\in \CijTJ$ and choose a
  compatible set $I=I(ij,T,J)$.  We write the coordinate ring of
  $\init_y\Gr_J(2,n)$ using Lemma~\ref{lm:InGrobner}:
\begin{equation*}
  \widetilde{K}[\init_y\Gr_J(2,n)]= 
  \frac{\widetilde{K}[u_{kl}^{\pm}: kl
    \notin J \cup \{ij\}]}{\langle u_{kl}-\init_{y}({f_{kl}}):kl\notin (I\cup
    J\cup \{ij\})\rangle}
  \simeq
  \widetilde{K}[u_{kl}^{\pm} :kl\in I\smallsetminus
  J]_{\mathscr{S}}.
\end{equation*}
Here, $\mathscr{S}$ is the multiplicatively closed set generated by
$\{\init_y(f_{rs}):rs\notin J\cup \{ij\}\}$.  The coordinate ring $
\widetilde{K}[\init_y\Gr_J(2,n)]$ is an integral domain, as we wanted to show.
\end{proof}

As a corollary, we determine the tropical multiplicities of all points
in $\cTG(2,n)$:
\begin{corollary}\label{cor:multiplicities}
  The multiplicity of every point in $\cTG(2,n)$, viewed in its
  ambient torus, is 1.
\end{corollary}
\noindent As a side remark, we observe that in the case where $x$ is a
regular point in $\cTGo(2,n)$, i.e., a point lying on the interior of
a cone associated to a trivalent tree, this statement was implicit
in~\cite[Theorem 3.4]{TropGrass}.  Indeed, the proof of this result,
together with~\cite[Corollary 4.4]{TropGrass}, ensures that the
initial ideal $\init_xI_{2,n}$ of any regular point is a binomial
prime ideal.

Note that the proofs of Theorems~\ref{thm:reduction}
and~\ref{thm:InitDegGrass} show that, for every point $x$ in the
tropical Grassmannian with vanishing set $J$, the coordinate ring of
the associated initial degeneration $\init_y\Gr_J(2,n)$ is isomorphic
to the reduction of the affinoid algebra given by the fiber
$\trop^{-1}(x)$. This is a concrete manifestation of the relationship
between analytic and initial degenerations discussed in~\cite[Section
4]{BPR}. In the curve case, tropical multiplicities one everywhere
accounts for the existence of a faithful tropicalization
by~\cite[Theorem 6.24]{BPR}.  \medskip

\begin{example}\label{ex:degn=4}
  We illustrate the previous results when $n=4$ by computing all
  initial degenerations for the Grassmannian $\Gr(2,4)$.  As we know
  from Section~\ref{sec:trop-grassm}, there are four tree types on
  four leaves: the quartets $(12|34), (13|24), (14|23)$ and the star
  tree $(1234)$.

  We start with the stratum $\Gr_0(2,4)$ corresponding to
  $J=\emptyset$. By symmetry, it suffices to exhibit the initial
  degenerations for points in $\cTC_{14|23}\cup \cTC_{12|34}\cup
  \cTC_{(1234)}$. We view these points inside $\cT U_{12}$.  These
  three initial degenerations will be generated by the initial form of
  the unique Pl\"ucker equation $u_{34}-u_{13}u_{24}+u_{14}u_{23}$.

  Pick $x\in \cTC_{14|23}\smallsetminus \cTC_{(1234)}$ and set
  $y=(x_{kl}-x_{12})_{kl\neq 12}$. The quartet $(14|23)$ is a
  caterpillar tree with backbone $\{1,2\}$. We set $I=I(12)$ as in
  Example~\ref{ex:caterpillar}, so the only pair remaining is $34$ and
  $f_{34}:=u_{13}u_{24}-u_{14}u_{23}$. Here,
  $\init_y(f_{34})=u_{13}u_{24}$ is a monomial and
  $\init_y\mathfrak{I}=\langle u_{34}-u_{13}u_{24}\rangle$. We
  conclude that $\init_y\Gr_0(2,4)=\Spec \widetilde{K}[u_{14}^{\pm},
  u_{24}^{\pm}, u_{13}^{\pm}, u_{23}^{\pm}]$.

  Next, we fix $x\in \cTC_{12|34}\smallsetminus \cTC_{(1234)}$. We
  pick $I=\{13,23,34,14\}$ as in Example~\ref{ex:cherry}. In this
  case, $f_{24}= u_{13}^{-1}(u_{34}+u_{14}u_{23})$, and
  $\init_y(f_{24})= u_{13}^{-1}u_{14}u_{23}$ is again a
  monomial. Thus, $\init_y\mathfrak{I}=\langle
  u_{24}-u_{13}^{-1}u_{14}u_{23}\rangle$ and $\init_y\Gr_0(2,4)=\Spec
  \widetilde{K}[u_{13}^{\pm}, u_{23}^{\pm}, u_{34}^{\pm},
  u_{14}^{\pm}]$. If $x\in \cTC_{(1234)}$, we choose the same $I$ but
  now $\init_y(f_{24})=f_{24}$. In this case $\init_y
  \mathfrak{I}=\mathfrak{I}$ and $\init_y \Gr_0(2,4)= \Spec
  \widetilde{K}[u_{13}^{\pm}, u_{23}^{\pm}, u_{34}^{\pm},
  u_{14}^{\pm}]$.

  On the boundary, the four-point conditions impose some constraints
  on the valid subsets $J$ that give nonempty strata $\Gr_J(2,4)$, as
  we saw on Lemma~\ref{lm:CharJ}.  By symmetry, we assume $12\notin J$
  and we work over $U_{12}$. We need to analyze the closures of cones
  associated to the three quartets.

  First, assume $x\in (\overline{\cTC_{13|24}}\cup
  \overline{\cTC_{14|23}})\cap \cT U_{12}$. Both quartets are
  caterpillar trees with backbone $\{1,2\}$, and we set
  $I=\{13,23,14,24\}$.  The initial form $\init_y(f_{34})$ is a
  monomial.  There are fifteen possibilities for $J$ and
  $\init_y\Gr_J(2,4)=\Spec \widetilde{K}[u_{kl}^{\pm}: kl\in
  I\smallsetminus J]$ for all of them.

  Finally, suppose $x\in \overline{\cTC_{12|34}}\cap \cT U_{12}$.  We
  take the order $3\prec 4$. The possible values for $J$ determine two
  valid choices for a compatible set $I$.  If $13, 23\notin J$, there
  are two options for $J$, namely $J_1=\{43\}$ and $J_2=\{14,24,43\}$.
  We choose $I=\{13,23,14,43\}$. In the first case,
  $f_{24}:=u_{13}^{-1}u_{14}u_{23}$ is a monomial and
  $\init_y\Gr_{J_1}(2,4)=\Spec \widetilde{K}[u_{13}^{\pm},
  u_{23}^{\pm}, u_{14}^{\pm}]$. In the second case, $\mathfrak{I}=
  (0)$ so $\init_y\Gr_{J_2}(2,4)= \Spec \widetilde{K}[u_{13}^{\pm}]$.

  On the contrary, suppose that either $13$ or $23$ lie in $J$. Then,
  again by symmetry between $3$ and $4$, we may further assume that
  either $14$ or $24$ are also in $J$ and so $I=\{13,23,14,24\}$. This
  leaves seven options: $J_1=\{13,14,23,43\}$, $J_2=\{13,23,24,43\}$,
  $J_3=\{13,14,23,24,43\}$, $J_4=\{13,14,24,43\}$, $J_5=\{23,24,43\}$,
  $J_6=\{14,23,43\}$ and $J_7=\{14,23,24,43\}$. For all these sets we
  have $\mathfrak{I}=(0)$, thus $\init_y\Gr_J(2,4)=\Spec
  \widetilde{K}[u_{kl}: kl\in I\smallsetminus J]$.
\end{example}

We end this Section by discussing piecewise linear structures on the
analytic Grassmannian.  Our key tool will be the section $\sigma$
constructed in Section~\ref{sec:cont-sect-from}.  As before, we fix a
vanishing set $J$ as in Lemma~\ref{lm:CharJ} and its associated
stratum $\Gr_{J}(2,n)$.  Denote by $\Sigma$ the image of $\sigma$, and
$\Sigma_J := \Sigma \cap \Gr_{J}(2,n)^{\an}$.
 
Starting from the set $J$ and a pair $ij\notin J$, we consider the
family $\mathcal{I}(J)$ of index sets $I$ of size $2(n-2)$ that are
compatible with $J(ij)$ and with a partial order $\preceq$ on
$[n]\smallsetminus \{i,j\}$ that has the cherry property on some tree
$T$ whose associated cone $\CijTJ$ is nonempty. This collection is
nonempty by Proposition~\ref{pr:CompatibleI}.

For every $I \in \mathcal{I}(J)$, we write $\GG_m^{I \smallsetminus J}
= \Spec K[u_{kl}^{\pm}: kl \in I \smallsetminus J]$.
Lemma~\ref{lem:iso} and~\eqref{eq:27} show that the coordinate ring
$\Gr_J(2,n)$ is isomorphic to a localization of $K[u_{kl}^{\pm}: kl
\in I \smallsetminus J]$. This induces a natural open embedding
\begin{equation*}
  \xymatrix@1{
    \varphi_{I,J} \colon \Gr_J(2,n)^{\an}\;\;\ar@{^{(}->}[r] & (\GG_m^{I
      \smallsetminus J})^{\an}.}
\end{equation*}
Following Example~\ref{ex:affinespace}, we denote by $\mathcal{S}_{I
  \smallsetminus J}$ the skeleton of $(\GG_m^{I \smallsetminus
  J})^{\an}$. Note that the image of the open embedding
$\varphi_{I,J}$ contains the set of norms
$\mathcal{S}_{I\smallsetminus J}$.
 
\begin{prop}\label{pr:skeleton}
  The sets $\Sigma_J$ and $\bigcup_{I \in \mathcal{I}(J)}
  \varphi_{I,J}^{-1}(\mathcal{S}_{I \smallsetminus J})$ agree as
  subsets of $\Gr_{J}(2,n)^{\an}$.
\end{prop}

\begin{proof} We prove the result by double inclusion.  Assume that $
  \sigma(x) \in \Sigma_J$ for some $x\in \cTG_J(2,n)$.  Choose $ij
  \notin J$ and a tree $T$ such that $x \in \overline{\cTC_T}\cap \cT
  U_{ij}$. Let $\preceq$ be an order on $[n]\smallsetminus \{i,j\}$
  with the cherry property on $T$ and $I\in \mathcal{I}(J)$ be
  compatible with $\preceq$ and $J(ij)$. Note that $\sigma(x)$
  vanishes on $\fa_J$. By construction, $\sigma(x)$ gives rise to the
  seminorm on $K[\Gr_J(2,n)]$ induced by the point $\delta_{I
    \smallsetminus J}(\pi_{I \smallsetminus J}(x))$ belonging to the
  skeleton $(\mathbb{G}_m^{I \smallsetminus J})^{\an}$. Therefore,
  $\varphi_{I,J} (\sigma(x)) \in \mathcal{S}_{I \smallsetminus J}$.

  Conversely, let $\xi \in \bigcup_{I \in \mathcal{I}(J)}
  \varphi_{I,J}^{-1}(\mathcal{S}_{I \smallsetminus J})$ and set $x :=
  \trop (\xi)$ in the coordinate system induced by the Pl\"ucker
  embedding~\eqref{eq:Pluecker}.  Since $\xi \in
  \varphi_{I,J}^{-1}(\mathcal{S}_{I \smallsetminus J})$ for some $I$,
  it is immediate to check that $\xi$ (viewed as a seminorm on
  $R(ij)$) is maximal among all multiplicative seminorms $\gamma$ on
  $R(ij)$ satisfying $\gamma(u_{kl})=\exp(x_{kl}-x_{ij})$ for all
  $kl\neq ij$.  By Lemma~\ref{lm:MaxCondition}, we conclude that
  $\xi=\sigma(x)$, so it lies in $\Sigma$.
\end{proof}

Recall that in the course of constructing $\sigma$, it was necessary
to find a suitable index set $I$ depending on $x \in \cTG(2,n)$. The
proof above shows that, whenever $\xi \in
\varphi_{I,J}^{-1}(\mathcal{S}_{I \smallsetminus J})$, $I$ is such a
suitable index set for $x = \trop \xi$.
 
We identify the image of our section as a skeleton in the sense of
Ducros \cite[(4.6)]{DucrosPL}.  Note that Ducros' skeleta carry
rational (not in general integral) piecewise linear structures.

 \begin{corollary}\label{cor:skeleton}
   The set $\Sigma_J$ is a skeleton in $\Gr_J(2,n)^{\an}$ in the sense
   of Ducros \cite{DucrosPL}.
 \end{corollary}
 \begin{proof}
   The result follows from Proposition~\ref{pr:skeleton} because $
   \bigcup_{I \in \mathcal{I}(J)} \varphi_{I,J}^{-1}(\mathcal{S}_{I
     \smallsetminus J})$ is a skeleton in $\Gr_J(2,n)^{\an}$ by
   \cite[(5.1) Th\'eor\`eme]{DucrosPL}.
 \end{proof}

\medskip

We use the previous results to endow $\Sigma$ with a piecewise linear
structure as defined by Berkovich, see \cite{ber99} and \cite{ber04}.
Roughly speaking, a \emph{piecewise linear space} is a locally compact
space provided with an atlas of
$R_S$-polyhedra~\cite[Section~1]{ber04}. Here, $S$ is a sub-semiring
of $\mathbb{R}$ containing $0$, and $R$ is a nonempty $S$-submonoid of
$\mathbb{R}^\ast_+$. In our setting, we let $S$ be the set of
nonnegative integers. We fix $R = |K^\ast|$ if the absolute value on
$K$ is nontrivial, and $R = \exp(\mathbb{Z})$ if the valuation is
trivial. Furthermore, we consider the additive situation obtained from
Berkovich's multiplicative setting by applying the logarithm map.  A
\emph{logarithmic $R_S$-polyhedron} in $\mathbb{R}^n$ is the solution
set of a finite number of linear inequalities of the form $\{a_1 t_1 +
\ldots + a_n t_n \leq c\}$ where $a_i\in \ZZ$ and either $c \in \log
|K^{\ast}|$ when the absolute value on $K$ is nontrivial or $c\in \ZZ$
when this absolute value is trivial.

We define $\Sigma_{I,J} := \varphi_{I,J}^{-1}(\mathcal{S}_{I
  \smallsetminus J})$. Since $\varphi_{I,J}$ is an open embedding, it
induces a homeomorphism $\varphi_{I,J}\colon \Sigma_{I,J} \rightarrow
\mathcal{S}_{I \smallsetminus J}$. Hence, $\Sigma_{I,J}$ is
homeomorphic to $ \mathbb{R}^{I \smallsetminus J}$.  We consider the
tropicalization map
 \[\trop\colon \Gr_J(2,n)^{\an} \longrightarrow \cTG_J(2,n)\subset \mathbb{R}^{\binom{n}{2}-|J|-1}\]
 induced by the coordinates of the ambient torus
 $\GG_m^{\binom{n}{2}-|J|-1}$.
 
 Pick a point $\xi \in \Sigma_{I,J}$ and let $x = \trop(\xi)$.  We
 write the polynomials $f_{rs}$ from~\eqref{eq:33} as $f_{rs} =
 \sum_{\alpha} c^{(rs)}_\alpha u^\alpha$ for $rs \notin J\cup \{ij\}$,
 where $\alpha \in \NN_0^{I\smallsetminus (J\cup H)}\times
 \ZZ^{H\smallsetminus J}$.  By construction, the point $x$ satisfies
\begin{equation}\label{eq:skeleton-pieces} x_{rs} - x_{ij} = \max_\alpha \{ \log |c^{(rs)}_\alpha| + \sum_{kl\in
    I\smallsetminus J} \alpha_{kl}(x_{kl}-  x_{ij}) \} \qquad \text{ for all } rs
  \notin J\cup \{ij\}.
\end{equation}
Conversely, pick a point $x\in \RR^{\binom{[n]}{2}\smallsetminus
  (J\cup\{ij\})}$ satisfying~\eqref{eq:skeleton-pieces}. We project to
the set of coordinates in $I \backslash J$ to obtain a point in the
skeleton $\mathcal{S}_{I \smallsetminus J}$. This point is of the form
$\varphi_{I,J}(\xi)$ for some $\xi \in \Sigma_{I,J}$.  The conditions
\eqref{eq:skeleton-pieces} ensure that $\trop(\xi) = x$. We conclude
that $\trop(\Sigma_{I,J})$ is the solution space
to~\eqref{eq:skeleton-pieces}.  We rewrite this system as a finite
disjunction of finite sets of linear inequalities and obtain a
logarithmic $R_S$-polytopal structure on $\trop(\Sigma_{I,J})\subseteq
\RR^{\binom{[n]}{2}\smallsetminus (J\cup\{ij\})}$.  Since
$\Sigma_{I,J}$ is contained in the image of the section $\sigma$, the
tropicalization map restricts to a homeomorphism $\Sigma_{I,J}
\rightarrow \trop(\Sigma_{I,J})$, which induces a logarithmic
$R_S$-polytopal structure on $\Sigma_{I,J}$.  We define a piecewise
linear structure on $\Sigma_J$ by taking the sets $\Sigma_{I,J}$ as
the charts of an atlas. The tropical coordinates are expressed
by~\eqref{eq:skeleton-pieces}, so this structure coincides with the
natural piecewise linear structure on the tropicalization given by the
cocharacter lattice of the ambient torus~\cite[Theorem 3.3]{gubler12}.

\section{The Petersen Graph}
\label{sec:petersen-graph}
In this Section, we focus our attention on the open subvariety
$\Gr_0(2,n)$ of the Grassmannian. This variety admits an action by an
$n$-dimensional split torus which induces a geometric quotient in the
sense of Mumford~\cite{GIT}. On the tropical side, this action is
encoded by the lineality space of $\cTGo(2,n)$. The associated
spherical simplicial complex is pure of dimension $n-4$. When $n=5$,
it is the Petersen graph, and it represents the boundary complex of
$\overline{M_{0,5}}$.  The restriction to $\Gr_0(2,n)$ of the map
$\sigma$ constructed in Section~\ref{sec:cont-sect-from} is compatible
with this action and allows us to embed the quotient of $\cTGo(2, n)$
by its lineality space into the analytification of the quotient of
$\Gr_0(2,n)$ by the split torus $\GG_m^n$.  In particular, when $n=5$,
our construction allows us to view the Petersen graph inside this
analytic quotient space.


Recall from Section~\ref{sec:trop-grassm} that we view $\Gr_0(2,n)$
inside the torus $\GG_m^{\binom{n}{2}} / \GG_m $ via the Pl\"ucker embedding.
We let $\Lambda$ be its character lattice and we identify the
cocharacter lattice with $\ZZ^{\binom{n}{2}}/\ZZ\tightcdot
\mathbf{1}$.  We identify the cocharacter lattice $\Lambda^{\vee}$ of
the torus with $M^{\vee}/\ZZ\tightcdot\mathbf{1}$, where $M=
\ZZ^{\binom{n}{2}}$.  The tropical fan $\cT \Gr_0(2,n) \subset
\Lambda_{\RR}^{\vee}$ is the tropicalization of $\Gr_0(2,n)^{\an}$
with respect to this fixed embedding.

We define the sublattice $L_{\ZZ}$ of $M^{\vee}$ generated by the
cut-metrics $\{\mathbf{l_1}, \ldots, \mathbf{l_n}\}$, i.e.,
\begin{equation}
(\mathbf{l_k})_{rs}=
\begin{cases}
  1& \text{ if } r=k \text{ or } s =k,\\
0 & \text{ otherwise}.
\end{cases}\label{eq:42}
\end{equation}
We let $L$ be the $\RR$-span of $L_{\ZZ}$ inside $M^{\vee}_{\RR}$.

The set $\cTGo(2,n)$ is a subfan of the Gr\"obner fan of the Pl\"ucker
ideal $I_{2,n}\subset K[\Lambda]$.  All cones of $\cTGo(2,n)$ contain
the linear space $\overline{L}=L/\RR\tightcdot\mathbf{1}$ in
$\Lambda_{\RR}^{\vee}$ spanned by the classes of the
\emph{cut-metrics}. 
We call $\overline{L}$ the \emph{lineality space} of $\cTG_0(2,n)$.
When $n=2$, we have $\overline{L}=\{0\}$ and when $n>2$, its dimension
is $n-1$.  We view the pointed fan $\cTGo(2,n)/\overline{L}$ inside
$\Lambda_\RR^{\vee}/\overline{L}$.

The lattice $L_{\ZZ}$ inside $M^{\vee}$ induces the following action
on $\Gr(2,n)$:
\begin{equation}\label{eq:action}
  \varrho\colon \GG_m^n\times \Gr(2,n)\longrightarrow \Gr(2,n) \qquad ((t_i)_{i \in [n]}, (p_{kl})_{kl}) \longmapsto (t_kt_lp_{kl})_{kl}.
\end{equation}

The orbit of any point in $\Gr(2,n)$ contains boundary points in its
closure. An easy calculation confirms that the only closed orbits in
$\Gr(2,n)$ are those associated to the $\binom{n}{2}$ points at
infinity $(0:\ldots: 0: 1: 0: \ldots: 0)$. Nonetheless, when
restricted to the open cell $\Gr_0(2,n)$, all orbits are closed.
Since $\GG_m^n$ is reductive and its action on $\Gr_0(2,n)$ is closed,
the action induces a geometric quotient $\Gr_0(2,n)/\GG_m^n$
(see~\cite[\S~1.2, Theorem~1.1 and Amplification~1.3]{GIT}). This
quotient is the affine scheme given by the ring of invariants under
this action. In what follows, we construct an embedding
$\varphi'\colon \Gr_0(2,n)/\GG_m^n\hookrightarrow \Spec
K[M/L^{\vee}_{\ZZ}]$. As expected, this embedding is compatible with
the Pl\"ucker map $\varphi$.

For simplicity, we denote $A_0 = K[\Gr_0(2,n)]$.  The monomial map
$\varrho$ from~\eqref{eq:action} as well as the second projection $p_2
\colon \GG_m^n \times \Gr_0(2,n) \to \Gr_0(2,n)$ both correspond to
ring homomorphisms $\varrho^\sharp, p_2^\sharp \colon A_0 \to
K[t_1^{\pm}, \ldots, t_n^{\pm}] \otimes A_0$.  Explicitly, we have
$\varrho^\sharp(u_{rs}) = t_rt_st_i^{-1}t_j^{-1} \otimes u_{rs} $ and
$p_2^\sharp (u_{rs}) = 1 \otimes u_{rs}$.  An easy induction reveals
that the ring of invariants under the action $\varrho$ equals
\[
A_0^{\GG_m^n} := \{ f \in A_0
: \varrho^\sharp(f) =
p_2^\sharp(f)\}=K[u_{kl}u_{ik}^{-1}u_{jl}^{-1},
u_{kl}u_{il}^{-1}u_{jk}^{-1}: k,l\neq i,j].
\]
The projection
$\Lambda^{\vee}=M^{\vee}/\ZZ\tightcdot\mathbf{1}\twoheadrightarrow
M^{\vee}/L_{\ZZ}$ dualizes to a homomorphism $M / L_{\ZZ}^\vee
\rightarrow \Lambda$ of lattices and hence induces a morphism of
tori. Consider the composition
\[\Gr_0(2,n) \longhookrightarrow \Spec K[\Lambda]
\longrightarrow \Spec K[M / L^{\vee}_{\ZZ}].\] It is invariant with
respect to the action $\varrho$, hence it factors over the geometric
quotient:
\[ \xymatrix{
  \Gr_0(2,n) \ar@{^(->}[rr]^{\varphi} \ar@{->>}[d] && \Spec K[\Lambda] \ar@{->}[d] \\
  \Gr_0(2,n)/\GG_m^n \ar@{^(->}[rr]^{\varphi'} && \Spec
  K[M/L_\ZZ^{\vee}]\; . }\]

Our next result, which appeared already in \cite[Lemma 6]{Mega09},
says that the pointed fan $\cTGo(2,n)/\overline{L}$ (with induced
multiplicities) in the $\RR$-vector space $M_{\RR}^{\vee}/L$ is precisely
the tropicalization of the geometric quotient $\Gr_0(2,n)/\GG_m^n$.

\begin{prop}\label{pro:QuotExists}
  The tropicalization $\cT(\Gr_0(2,n)/\GG_m^n)$ under ${\varphi'}$ is
  the pointed fan $\cTG_0(2,n)/\overline{L}$.
\end{prop}

We end by showing that the map $\sigma$ from Theorem~\ref{thm:MainThm}
respects the quotient structure on its source and target spaces and,
thus, it induces a map $\sigma'\colon \cTG_0(2,n)/\overline{L}\to
(\Gr_0(2,n)/\GG_m^n)^{\an}$.

\begin{lemma}\label{lem:QuotIsRespected}
  Let $x, y \in \cTGo(2,n)$, and let $\sigma$ be the section to
  $\trop$ from Theorem~\ref{thm:MainThm}.  Then, $\sigma(x)$ and
  $\sigma(y)$ induce the same element in $(\Gr_0(2,n)/\GG_m^n)^{\an}$
  if and only if $x - y \in \overline{L}$.
\end{lemma}

\begin{proof}
  The ``only if'' part follows from Proposition~\ref{pro:QuotExists}
  and holds for any continuous section to the tropicalization map.
  For the converse, note that since $x$ and $y$ lie in the same cone
  $\cTC_T$ of $\cTGo(2,n)$, $\sigma(x)$ and $\sigma(y)$ can be defined
  using the same set of coordinates $I=I(ij,T,\emptyset)$ for any pair
  of indices $i,j$. The condition $x-y \in \overline{L}$ implies that
  $\sigma(x)$ and $\sigma(y)$ agree on all monomials generating
  $A_0^{\GG_m^n}$, thus $\sigma(x) = \sigma(y)$ on $A_0^{\GG_m^n}$.
\end{proof}

\begin{corollary}\label{cor:QuotientEmbedding} The tropicalization map $\trop\colon
  (\Gr_0(2,n)/\GG_m^n)^{\an} \to M^{\vee}_{\RR}/L$ has a continuous
  section $\sigma'\colon \cTGo(2,n)/\overline{L} \to
  (\Gr_0(2,n)/\GG_m^n)^{\an}$.
\end{corollary}

\begin{example} In~\cite[Example 9.10]{StSolvingSystems}, Sturmfels
  showed that the spherical complex associated to the tropical
  Grassmannian $\cTGo(2,5)/\overline{L}$ is the Petersen graph.  The
  embedding $\sigma'$ from Corollary~\ref{cor:QuotientEmbedding}
  allows us to view the Petersen graph inside the analytification
  $(\Gr_0(2,5)/\GG_m^5)^{\an}$.
\end{example}

\section*{Acknowledgments}
\label{sec:acknowledgements}

We wish to thank Antoine Ducros and Sam Payne for very fruitful
conversations and Josephine Yu for helping us with some computations
using \texttt{Macaulay2}~\cite{M2} and \texttt{Gfan}~\cite{gfan}.  We
would like to express our gratitude to the anonymous referee for
his/her careful reading and detailed comments. The first author was
supported by an Alexander von Humboldt Postdoctoral Research
Fellowship.

\def\cprime{$'$}


\begin{thebibliography}{10}

\bibitem{BPR}
M.~Baker, S.~Payne, and J.~Rabinoff.
\newblock Nonarchimedean geometry, tropicalization, and metrics on curves.
\newblock \url{arXiv:1104.0320}, 2011.

\bibitem{berkovichbook}
V.~G. Berkovich.
\newblock {\em Spectral theory and analytic geometry over non-{A}rchimedean
  fields}, volume~33 of {\em Mathematical Surveys and Monographs}.
\newblock Amer. Math. Soc., Providence, RI, 1990.

\bibitem{ber99}
V.~G. Berkovich.
\newblock Smooth {$p$}-adic analytic spaces are locally contractible.
\newblock {\em Invent. Math.}, 137(1):1--84, 1999.

\bibitem{ber04}
V.~G. Berkovich.
\newblock Smooth {$p$}-adic analytic spaces are locally contractible. {II}.
\newblock In {\em Geometric aspects of {D}work theory. {V}ol. {I}, {II}}, pages
  293--370. Walter de Gruyter GmbH \& Co. KG, Berlin, 2004.

\bibitem{BG}
R.~Bieri and J.~R.~J. Groves.
\newblock The geometry of the set of characters induced by valuations.
\newblock {\em J. Reine Angew. Math.}, 347:168--195, 1984.

\bibitem{BJSST}
T.~Bogart, A.~N. Jensen, D.~Speyer, B.~Sturmfels, and R.~R. Thomas.
\newblock Computing tropical varieties.
\newblock {\em J. Symbolic Comput.}, 42(1-2):54--73, 2007.

\bibitem{bgr}
S.~Bosch, U.~G{\"u}ntzer, and R.~Remmert.
\newblock {\em Non-{A}rchimedean analysis}, volume 261 of {\em Grundlehren der
  Mathematischen Wissenschaften}.
\newblock Springer-Verlag, Berlin, 1984.

\bibitem{Buneman74}
P.~Buneman.
\newblock A note on the metric properties of trees.
\newblock {\em J. Combinatorial Theory Ser. B}, 17:48--50, 1974.

\bibitem{Mega09}
M.~A. Cueto, E.~A. Tobis, and J.~Yu.
\newblock An implicitization challenge for binary factor analysis.
\newblock {\em J. Symbolic Comput.}, 45(12):1296--1315, 2010.

\bibitem{DraismaTropSecant}
J.~Draisma.
\newblock A tropical approach to secant dimensions.
\newblock {\em J. Pure Appl. Algebra}, 212(2):349--363, 2008.

\bibitem{DucrosPL}
A.~Ducros.
\newblock Espaces de {B}erkovich, polytopes, squelettes et th{\'e}orie des
  mod{\`e}les.
\newblock \url{ arXiv:1203.6498}, 2012.

\bibitem{EKL}
M.~Einsiedler, M.~M. Kapranov, and D.~Lind.
\newblock Non-{A}rchimedean amoebas and tropical varieties.
\newblock {\em J. Reine Angew. Math.}, 601:139--157, 2006.

\bibitem{binomialIdealsES}
D.~Eisenbud and B.~Sturmfels.
\newblock Binomial ideals.
\newblock {\em Duke Math. J.}, 84(1):1--45, 1996.

\bibitem{FJT}
K.~Fukuda, A.~Jensen, and R.~R. Thomas.
\newblock Computing {G}r{\"o}bner fans.
\newblock {\em Math. Comp.}, 76(260):2189--2212, 2007.

\bibitem{M2}
D.~R. Grayson and M.~E. Stillman.
\newblock Macaulay2, a software system for research in algebraic geometry.
\newblock Available at \href{http://www.math.uiuc.edu/Macaulay2/}%
  {http://www.math.uiuc.edu/Macaulay2/}.

\bibitem{gub07-2}
W.~Gubler.
\newblock The {B}ogomolov conjecture for totally degenerate abelian varieties.
\newblock {\em Invent. Math.}, 169(2):377--400, 2007.

\bibitem{gub07-1}
W.~Gubler.
\newblock Tropical varieties for non-{A}rchimedean analytic spaces.
\newblock {\em Invent. Math.}, 169(2):321--376, 2007.

\bibitem{gubler12}
W.~Gubler.
\newblock A guide to tropicalizations.
\newblock In {\em Algebraic and combinatorial aspects of Tropical Geometry},
  volume 589 of {\em Contemp. Math.} Amer. Math. Soc., Providence, RI, 2013.
\newblock \url{ arXiv:1108:6126}.

\bibitem{treeOfLife}
T.~Hodkinson, J.~Parnell, and C.~Wai-Kai, editors.
\newblock {\em Reconstructing the Tree of Life: Taxonomy and Systematics of
  Species Rich Taxa}.
\newblock Systematics Association Special Volumes. CRS Press, 2007.

\bibitem{gfan}
A.~N. Jensen.
\newblock Gfan, a software system for {G}r{\"o}bner fans and tropical
  varieties.
\newblock Available at
\url{http://home.imf.au.dk/jensen/software/gfan/gfan.html}, 2011.

\bibitem{Manon10}
C.~Manon.
\newblock Valuations from representation theory and tropical geometry.
\newblock \url{arXiv:1006.0038}, 2010.

\bibitem{Manon11}
C.~Manon.
\newblock Toric degenerations and tropical geometry of branching algebras.
\newblock \url{arXiv:1103.2484}, 2011.

\bibitem{GIT}
D.~Mumford, J.~Fogarty, and F.~Kirwan.
\newblock {\em Geometric Invariant Theory}.
\newblock Ergebnisse Der Mathematik Und Ihrer Grenzgebiete. 2. Folge.
  Springer-Verlag, 1994.

\bibitem{ASCB}
L.~Pachter and B.~Sturmfels, editors.
\newblock {\em Algebraic statistics for computational biology}.
\newblock Cambridge University Press, 2005.

\bibitem{PayneAnalytification}
S.~Payne.
\newblock Analytification is the limit of all tropicalizations.
\newblock {\em Math. Res. Lett.}, 16(3):543--556, 2009.

\bibitem{FibersOfTrop}
S.~Payne.
\newblock Fibers of tropicalization.
\newblock {\em Math. Z.}, 262(2):301--311, 2009.

\bibitem{TropGrass}
D.~Speyer and B.~Sturmfels.
\newblock The tropical {G}rassmannian.
\newblock {\em Adv. Geom.}, 4(3):389--411, 2004.

\bibitem{StInv93}
B.~Sturmfels.
\newblock {\em Algorithms in invariant theory}.
\newblock Texts and Monographs in Symbolic Computation. Springer-Verlag,
  Vienna, 1993.

\bibitem{StSolvingSystems}
B.~Sturmfels.
\newblock {\em Solving systems of polynomial equations}, volume~97 of {\em CBMS
  Regional Conference Series in Mathematics}.
\newblock Published for the Conference Board of the Mathematical Sciences,
  Washington, DC, 2002.

\bibitem{ElimTheory}
B.~Sturmfels and J.~Tevelev.
\newblock Elimination theory for tropical varieties.
\newblock {\em Math. Res. Lett.}, 15(3):543--562, 2008.

\bibitem{thu}
A.~Thuillier.
\newblock {\em Th{\'e}orie de potentiel sur les courbes en g{\'e}om{\'e}trie
  analytique non archim{\'e}dienne. {A}pplications {\`a} la th{\'e}orie
  d'{A}rakelov}.
\newblock PhD thesis, Universit{\'e} de {R}ennes 1, 2005.
\newblock viii+ 184 p.

\end{thebibliography}

\vspace{1ex}

\noindent
\textbf{\small{Authors' addresses:}}

\smallskip                        
\noindent
\small{M.\ A.\ Cueto,  Mathematics Department,
  Columbia University, 2990 Broadway, New York, NY, 10027, USA.
\\
\noindent \emph{Email address:} \url{macueto@math.columbia.edu}}
\vspace{2ex}

\noindent
\small{M.\ H\"abich and A.\ Werner, Institut f\"ur Mathematik, Goethe-Universit\"at
  Frankfurt, Robert-Mayer-Str.\ 6-8, 60325 Frankfurt a.\ M., Germany. 
\\
\noindent \emph{Email addresses:} \{\url{haebich},\url{werner}\}\url{@math.uni-frankfurt.de}
}

\end{document}